%% file: Diabolic_22_nov.tex
\documentclass[10pt,a4paper]{article}
\usepackage{amsmath,amssymb,amsthm} 
\usepackage{bbm}
\usepackage[usenames,dvipsnames]{pstricks}
\usepackage{epsfig}
\usepackage{pst-grad} 
\usepackage{pst-plot} 

\newcommand{\eps}{\varepsilon}
\newcommand{\bt}{\begin{thr}{\bf Theorem. }}
\newcommand{\satz}{\begin{thr}{\bf Theorem. }\rm}
\newcommand\E{\mathbb{E}}
\newcommand\R{\mathbb{R}}

\newcommand\Rp{\R_{\,+}}
\newcommand\iin{_{i=1}^N}
\newcommand\jin{_{j=1}^N}
\newcommand\cij{c(i,j)}
\newcommand\piij{\pi(i,j)}

\newcommand\cpi{\langle c, \pi \rangle}
\renewcommand\i{\infty}
\renewcommand\phi{\varphi}
\newcommand\la{\lambda}
\newcommand\lan[1]{\langle #1 \rangle}
\newcommand\hphi{\widehat{\varphi}}
\newcommand\hpsi{\widehat{\psi}}
\newcommand\hpi{\widehat{\pi}}
\newcommand\hh{\widehat{h}}

\newcommand\N{\mathbb{N}}
\newcommand\Z{\mathbb{Z}}
\newcommand\T{\mathbb{T}}
\newcommand{\IM}{I_{\text{middle}}}

\newcommand{\Prel}{P^{\mathrm{{rel}}}}
\newcommand{\Drel}{D^{\mathrm{{rel}}}}
\newcommand{\Pipart}{\Pi^{\mathrm{part}}}

\newcommand\I{\mathbf{1}}

\newcommand\ldot{\,.\,}

\newtheorem{theorem}{Theorem}[section]  

\newtheorem{definition}[theorem]{Definition}
\newtheorem{lemma}[theorem]{Lemma}
\newtheorem{proposition}[theorem]{Proposition}

\newtheorem{example}[theorem]{Example}

\title{On the Duality Theory for the Monge--Kantorovich Transport Problem}

\author{Mathias Beiglb\"ock, Christian L\'eonard, Walter Schachermayer\thanks{The first author acknowledges financial support from the Austrian Science
Fund (FWF) under grant P21209. The third author acknowledges support from the Austrian Science Fund
(FWF) under grant P19456, from the Vienna Science and Technology Fund (WWTF) under grant
MA13 and by the Christian Doppler Research Association (CDG).
All authors thank A.~Pratelli for helpful discussions on the topic of this paper. We also thank L.~Summerer for his advice. }}

\begin{document}

\date{}
\maketitle


\section{Introduction}
This article, which is an accompanying paper to \cite{BeLS09a}, consists of two parts: In section 2 we present a version of
Fenchel's perturbation method for the duality theory of the Monge--Kantorovich problem of optimal transport. The treatment is elementary as we suppose that
the spaces $(X,\mu), (Y,\nu)$, on which the optimal transport problem \cite{Vill03,Vill09} is defined, simply equal the finite set
$\{1,\dots,N\}$ equipped with uniform measure. In this setting the optimal transport problem reduces to a finite-dimensional linear programming problem.

The purpose of this first part of the paper is rather didactic: it
should stress some features of the linear programming nature of
the optimal transport problem, which carry over also to the case
of general polish spaces $X, Y$ equipped with Borel probability
measures $\mu,\nu$, and general Borel measurable cost functions
$c: X\times Y \to [0,\infty]$. This general setting is analyzed in
detail in \cite{BeLS09a}; section 2 below may serve as a
motivation for the arguments in the proof of Theorems 1.2 and 1.7
of \cite{BeLS09a} which pertain to the general duality theory.

\medskip

The second --- and longer --- part of the paper, consisting of sections 3 and 4 is of a quite different nature.

\medskip

Section 3 is devoted to illustrate a technical feature of
\cite[Theorem 4.2]{BeLS09a} by an explicit example. The technical
feature is the appearance of the singular part $\hh^s$ of the dual
optimizer $\hh\in L^1 (X\times Y, \pi)^{**}$ obtained in
(\cite[Theorem 4.2]{BeLS09a}). In Example 3.1 below we show that,
in general, the dual optimizer $\hh$ does indeed contain a
non-trivial singular part. In addition, this example allows to
observe in a rather explicit way how this singular part ``builds
up'', for an optimizing sequence $(\phi_n \oplus \psi_n)^\i_{n=1}
\in L^1 (X\times Y,\pi)$ which converges to $\hh$ with respect to
the weak-star topology. The construction of this example, which is
a variant of an example due to L.~Ambrosio and A.~Pratelli
\cite{AmPr03}, is rather longish and technical. Some motivation
for this construction will be given at the end of Section 2.

\medskip

Section 4 pertains to a  modified version of the duality relation in the Monge-Kantorovich transport problem. Trivial counterexamples such as \cite[Example 1.1]{BeLS09a} show that in the case of a measurable cost function $c:X\times Y\to [0,\infty]$ there may be a duality gap.
 The main result (Theorem 1.2) of \cite{BeLS09a} asserts that one may avoid this difficulty by considering a suitable relaxed form of the primal problem; if one does so, duality holds true in complete generality.
In a different vein, one may leave the primal problem unchanged,
and overcome the difficulties encountered in the above mentioned
simple example by considering a slightly modified dual problem
(cf.\ \cite[Remark 3.4]{BeLS09a}). In the last part of the article
we consider a certain twist of the construction given in section
3, which allows us to prove that this dual relaxation does not
lead to a general duality result.

\section{The finite case}\label{s2}

In this section we present the duality theory of optimal transport for the finite case: Let $X=Y=\{1,\ldots, N\}$ and let $\mu=\nu$ assign probability $N^{-1}$ to each of the points $1,\ldots, N$. Let $c=(c(i,j))_{i,j=1}^N$ be an $\Rp$-valued $N\times N$ matrix.

The problem of optimal transport then becomes the subsequent linear optimization problem
\begin{align}\label{1}
\langle c, \pi \rangle := \sum\iin \ \sum\jin \pi(i,j) \, c(i,j) \to \min, ~~~~ \pi\in\R^{N^2},
\end{align}
under the constraints
\begin{align*}
& \sum\jin \pi(i,j) =  N^{-1},  & i=1,\ldots, N, \\
& \sum\iin \pi(i,j) =  N^{-1}, & j=1,\ldots, N, \\
&  \ \pi(i,j)\geq 0,  & i,j=1,\ldots, N.
\end{align*}

Of course, this is an easy and standard problem of linear optimization; yet we want to treat it in some detail in
order to develop intuition and concepts for the general case considered in \cite{BeLS09a} as well as in section 3 .

For the two sets of \emph{equality} constraints we introduce $2N$ Lagrange multipliers $(\phi(i))\iin$ and
$(\psi(j))\jin$ taking values in $\mathbb{R}$, and for the inequality constraints (4) we introduce Lagrange
multipliers $(\rho_{ij})^N_{i,j=1}$ taking values in $\R_+$. The Lagrangian functional $L(\pi,\phi,\psi,\rho)$ then is given by
\begin{align*}
L(\pi,\phi,\psi,\rho)
 = & \sum\iin\sum\jin \cij \, \piij
\\ & -\sum\iin \phi(i) \left( \sum\jin \piij-N^{-1} \right)
\\ & - \sum\jin \psi (j) \left( \sum\iin \piij-N^{-1} \right)
\\ & - \sum\iin \sum\jin \rho(i,j) \piij,
\end{align*}
where the $\pi(i,j), \phi(i)$ and $\psi(j)$ range in $\R$, while the $\rho(i,j)$ range in $\R_+$.

It is designed in such a way that
\[
C(\pi):= \sup_{\phi,\psi,\rho} L(\pi,\phi,\psi,\rho) = \cpi +
\chi_{\Pi(\mu,\nu)}(\pi),
\]
where $\Pi(\mu,\nu)$ denotes the admissible set of $\pi$'s, i.e., the probability measures on $X\times Y$ with marginals $\mu$ and $\nu$, and $\chi_A(\,\ldot\,)$ denotes the indicator function of a set $A$ in the sense of convex function theory, i.e., taking the value $0$ on $A$, and the value $+\i$ outside of $A$.

In particular, we have
\begin{align*}
P := \inf_{\pi\in\mathbb{R}^{N^2}} C(\pi) = \inf_{\pi\in \mathbb{R}^{N^2}} \sup_{\phi,\psi,\rho} L(\pi,\phi,\psi,\rho),
\end{align*}
where $P$ is the optimal value of the primal optimization problem
(1).

To develop the duality theory of the primal problem (1) we pass from {\it inf sup L} to {\it sup inf L}.
Denote by $D(\phi,\psi,\rho)$ the dual function
\begin{align*}
D(\phi,\psi,\rho)
   =& \inf_{\pi\in\mathbb{R}^{N^2}} L (\pi,\phi,\psi,\rho)
\\ =& \inf_{\pi\in\mathbb{R}^{N^2}} \sum\iin \sum\jin \piij [\cij - \phi(i) - \psi(j) - \rho(i,j)]
\end{align*}
$$+ N^{-1} \left[ \sum\iin\phi(i) + \sum\jin\psi(j) \right].$$
Hence we obtain as the optimal value of the dual problem
\begin{equation}\label{FiniteDual}
D:= \sup_{\phi,\psi,\rho} D(\phi,\psi,\rho) =
(\E_\mu[\phi]+\E_\nu[\psi])-\chi_\Psi(\phi,\psi)
\end{equation}
where $\Psi$ denotes the admissible set of $\phi, \psi, \rho$, i.e.~satisfying
\[
\phi(i) + \psi(j) +\rho (i,j) = \cij, ~~~~ 1\leq i,j \leq N,
\]
for some non-negative ``slack variables'' $\varrho_{i,j}.$

Let us show that there is no duality gap, i.e., the values of $P$ and $D$ coincide. Of course, in the present
finite dimensional case, this equality as well as the fact that the {\it inf sup} (resp.\ {\it sup inf}) above is a
{\it min max} (resp.\ {\it a max min}) easily follows from general compactness arguments. Yet we want to verify things
directly using the idea of ``complementary slackness'' of the primal and the dual constraints (good references are, e.g.~\cite{PeSU08,EkTe99,AuEk06}).

We apply ``Fenchel's perturbation map'' to explicitly show the equality $P=D$.
Let $T:\R^{N^2}\to\R^N\times\R^N$ be the linear map defined as
\[
T\left( \big(\piij \big)_{1\leq i,j \leq N} \right) = \left( \left( \sum\jin \piij \right)\iin , \left( \sum\iin \piij
\right)\jin \right)
\]
so that the problem \eqref{1} now can be phrased as
\[
\cpi = \sum\iin\sum\jin \cij\, \piij \to min, ~~~~ \pi\in\Rp^{N^2},
\]
under the constraint
\[
T(\pi) = \left( (N^{-1}, \ldots, N^{-1}),(N^{-1}, \ldots, N^{-1}) \right).
\]
The range of the linear map $T$ is the subspace $E \subseteq\R^N \times\R^N$, of codimension 1, formed by the pairs
$(f,g)$ such that $\sum\limits\iin f(i)=\sum\limits\jin g(j),$ in other words $\mathbb{E}_\mu[f]=\mathbb{E}_\nu [g]$.
We consider $T$ as a map from $\R^{N^2}$ to $E$ and denote by $E_+$ the positive orthant of $E$.

Let $\Phi: E_+ \to [0,\i]$ be the map
\[
\Phi(f,g) = \inf \left\{ \cpi, \ \pi\in\Rp^{N^2}, \ T(\pi)=(f,g) \right\}.
\]
We shall verify explicitly that $\Phi$ is an $\R_+$-valued, convex, lower semi-continuous, positively homogeneous map on $E_+$.

The finiteness and positivity of $\Phi$ follow from the fact that, for $(f,g)\in E_+$, the set of $\pi\in\R_+^{N^2}$
with $T(\pi)=(f,g)$ is non-empty and from the non-negativity of $c$.
As regards the convexity of $\Phi$, let $(f_1,g_1),(f_2,g_2) \in E_+$ and find $\pi_1,\pi_2\in\Rp^{N^2}$
such that $T(\pi_1)=(f_1,g_1), \ T(\pi_2)=(f_2,g_2)$ and
$\langle c,\pi_1\rangle <\Phi(f_1,g_1)+\varepsilon$ as well as $\langle c,\pi_2\rangle <\Phi(f_2,g_2)+\varepsilon$. Then
$$\Phi\left(
\frac{(f_1,g_1)+(f_2,g_2)}{2}\right )
\leq \left\langle c, \frac{\pi_1 +\pi_2}{2}
\right\rangle < \frac{\Phi(f_1,g_1)+\Phi(f_2,g_2)}{2} +\varepsilon,$$ which proves the convexity of $\Phi$.

If $((f_n,g_n))^\infty_{n=1} \in E_+$ converges to $(f,g)$ find
$(\pi_n)^\infty_{n=1}$ in $\R_+^{N^2}$ such that $T(\pi_n)=(f_n,g_n)$ and $\langle c,\pi_n\rangle < \Phi(f_n,g_n)+n^{-1}$.
Note that $(\pi_n)^\infty_{n=1}$ is bounded in $\R^{N^2}_+$, so that there is a subsequence
$(\pi_{n_k})^\infty_{k=1}$ converging to $\pi\in\Rp^{N^2}$.
Hence $\Phi(f,g,)\leq \langle c,\pi\rangle$ showing the lower semi-continuity of $\Phi$.
Finally note that $\Phi$ is positively homogeneous,
i.e., $\Phi(\la f, \la g)=\la\Phi(f,g)$, for $\la\geq 0$. \\

The point $(f_0,g_0)$ with $f_0=g_0=(N^{-1},\dots ,N^{-1})$ is in $E_+$ and $\Phi$ is bounded in a neighbourhood $V$
of $(f_0,g_0).$ Indeed, fixing any $0<a<N^{-1}$
the subsequent set $V$ does the job
$$V=\{(f,g)\in E \ : \vert f(i)-N^{-1}\vert <a, \ \vert g(j)-N^{-1}\vert <a, \quad \mbox{for} \ 1\le i,j\le N\}.$$
The boundedness of the lower semi-continuous convex function
$\Phi$ on $V$ implies that the subdifferential of $\Phi$ at
$(f_0,g_0)$ is non-empty. Considering $\Phi$ as a function on
$\R^{2N}$ (by defining it to equal $+\infty$ on $\R^{2N}\backslash
E_+)$ we may find an element $(\hphi, \hpsi)\in \R^N\times\R^N$ in
this subdifferential. By the positive homogeneity of $\Phi$ we
have
\[
\Phi(f,g) \geq \lan{(\hphi, \hpsi),(f,g)}= \lan{\hphi,f}+\lan{\hpsi,g},
 \ \ \ \mbox{for} \ (f,g)\in\mathbb{R}^N\times\mathbb{R}^N,
\]
and
\[
P = \Phi(f_0,g_0) = \lan{\hphi,f_0}+\lan{\hpsi,g_0}.
\]
By the definition of $\Phi$ we therefore have, for each $\pi\in\Rp^{N^2}$,
\begin{align*}
\cpi & \geq \inf_{\tilde{\pi}\in\Rp^{N^2}}  \{\langle c,\tilde{\pi}\rangle : T(\pi)= T (\tilde{\pi})\}
\\ &  = \Phi(T(\pi))
\\ & \geq \langle T(\pi), (\hat{\phi},\hat{\psi})\rangle
\\ &  =  \sum\iin\sum \jin \pi(i,j)  \, [\hphi(i)+\hpsi(j)]
\end{align*}
so that
\begin{equation}
\cij \geq \hphi(i)+\hpsi(j),  \quad\quad\quad \mbox{for} \ 1\le i,j\le n.
\end{equation}
By compactness, there is $\hpi\in\Pi(\mu,\nu)$, i.e., there is an element $\hpi\in\Rp^{N^2}$ verifying $T(\hpi)=(f_0,g_0)$ such that
\begin{equation}
\lan{c,\hpi} = \lan{\hphi+\hpsi,\hpi}.
\end{equation}

Summing up, we have shown that $\hpi$ and $(\hphi,\hpsi)$ are primal and dual optimizers and that the value of the
primal problem equals the value of the dual problem, namely $\lan{\hphi+\hpsi,\hpi}$.

To finish this elementary treatment of the finite case, let us consider the case when we allow the cost function $c$ to
 take values in $[0,\i]$ rather than in  $[0,\i[$. In this case the primal problem simply loses some dimensions: for
the $(i,j)$'s where $\cij=\i$ we must have $\piij=0$ so that we consider
\begin{equation}
\langle c, \pi \rangle := \sum\iin \ \sum\jin \pi(i,j) \, c(i,j) \to \min, ~~~~ \pi\in\Rp^{N^2}, \nonumber
\end{equation}
where we now optimize over $\pi\in\R^{N^2}_+$ with $\piij=0$ if $\cij=\i$. For the problem to make sense we clearly must
have that there is at least one $\pi\in\Pi(\mu,\nu)$ with $\cpi<\i$. If this non-triviality condition is satisfied, the
above arguments carry over without any non-trivial modification.

\bigskip

We now analyze explicitly the well-known ``complementary slackness conditions'' and interpret them in the present context. For a pair $\hpi$ and $(\hphi,\hpsi)$ of primal and dual optimizers we have
\[
\cij > \hphi(i)+\hpsi(j)  ~\Rightarrow ~ \hpi(i,j)=0,
\]
and
\[
\hpi(i,j)>0~ \Rightarrow ~ \cij = \hphi(i)+\hpsi(j).
\]
Indeed, these relations follow from the admissibility condition $c \geq \hat{\phi}+\hat{\psi}$ and the duality relation $\lan{\hpi,c-(\hphi+\hpsi)}=0$.

This motivates the following definitions in the theory of optimal transport (see, e.g., \cite{RaRu96}
for (a) and \cite{ScTe08} for (b).)

\begin {definition}
Let $X=Y=\{1,\ldots,N\}$ and $\mu=\nu$ the uniform distribution on $X$ and $Y$ respectively, and let $c: X\times Y \to \Rp$ be given.
\begin{description}
\item[(a)] A subset $\Gamma\subseteq X\times Y$ is called ``cyclically $c$-monotone'' if, for $(i_1,j_1), \ldots, (i_n, j_n) \in \Gamma$ we have
\begin{align}\label{P6}
\sum_{k=1}^n c(i_k, j_k) \leq \sum_{k=1}^n c(i_k, j_{k+1}),
\end{align}
where $j_{n+1}=j_1$.
\item[(b)] A subset $\Gamma \subseteq X\times Y$ is called ``strongly cyclically $c$-monotone'' if there are functions $\phi,\psi$ such that $\phi(i)+\psi(j)\leq \cij$, for all $(i,j)\in X\times Y$, with equality holding true for $(i,j) \in \Gamma$.
\end{description}
\end{definition}

In the present finite setting, the following facts are rather obvious (assertion (iii) following from the above discussion):

\begin{description}
\item[(i)] The support of each primal optimizer $\hpi$ is cyclically $c$-monotone.
\item[(ii)] Every $\pi\in\Pi(\mu,\nu)$ which is supported by a cyclically $c$-monotone set $\Gamma$, is a primal
optimizer.
\item[(iii)] A set $\Gamma\subseteq X\times Y$ is cyclically $c$-monotone iff it is strongly cyclically $c$-monotone.
\end{description}

In general, one may ask, for a given Monge--Kantorivich transport optimization problem, defined on polish spaces $X,Y$, equipped with Borel probability measures
$\mu,\nu$, and a Borel measurable cost function $c:X\times Y\to [0,\i]$, the following natural questions:\\

{\bf (P)} Does there exist a primal optimizer to (1), i.e.~a Borel measure  $\hpi \in\Pi(\mu,\nu)$ with marginals $\mu,\nu$, such that
\begin{align*}\nonumber
\int\limits_{X\times Y}{c} \ d\widehat\pi=\inf\limits_{\pi\in\Pi(\mu,\nu)} \int\limits_{X\times Y} c \ d\pi=:P
\end{align*}
holds true?\\

{\bf (D)} Do there exist dual optimizers to \eqref{FiniteDual}, i.e.~Borel functions $(\hphi,\hpsi)$ in $\Psi(\mu,\nu)$ such that
\begin{align}\label{H4}
\int\limits_X\hphi \ d\mu + \int\limits_Y\hpsi \ d\nu =\sup\limits_{(\varphi,\psi)\in\Psi(\mu,\nu)} \left(\int\limits_X \varphi \ d\mu +\int\limits_Y\psi \ d\nu
\right) =:D,
\end{align}
where $\Psi(\mu,\nu)$ denotes the set of all pairs of $[-\i ,+\i [$-valued integrable Borel functions $(\phi,\psi)$ on $X,Y$ such that $\varphi(x) +\psi(y)
\le x(x,y)$, for all $(x,y)\in X\times Y$? \\

{\bf (DG)} Is there a duality gap, or do we have $P=D$, as it should -- morally speaking -- hold true?\\

These are three natural questions which arise in every convex optimization problem. In addition, one may ask the following two questions pertaining to the
special features of the Monge--Kantorovich transport problem.\\

{\bf (CC)} Is every cyclically $c$-monotone transport plan $\pi\in\Pi(\mu,\nu)$ optimal, where we call $\pi\in\Pi(\mu,\nu)$ cyclically $c$-monotone if there is
a Borel subset $\Gamma \subseteq X\times Y$ of full support $\pi(\Gamma)=1$, verifying condition \eqref{P6}, for any $(x_1,y_1),\dots ,(x_n,y_n)\in\Gamma$?\\

{\bf (SCC)} Is every strongly cyclically $c$-monotone transport plan $\pi\in\Pi(\mu,\nu)$ optimal, where we call $\pi\in\Pi(\mu,\nu)$ strongly cyclically $c$-monotone
if there are Borel functions $\phi :X\to [-\i ,+\i [$ and $\psi :Y\to[-\i ,+\i [$, satisfying $\phi(x)+\psi(y)\le c(x,y)$, for all $(x,y)\in X\times Y$, and
$\pi\{\phi +\psi =c\}=1$?

\medskip

Much effort has been made over the past decades to provide
increasingly general answers to the questions above. We mention
the work of R\"uschendorf \cite{R96} who adapted the notion of
cyclical monotonicity from Rockafellar \cite{Rock66}.
Rockafellar's work pertains to the case $c (x,y)=-\langle
x,y\rangle$, for $x,y\in\R^n$, while R\"uschendorf's work pertains
to the present setting of general cost functions $c$, thus
arriving at the notion of cyclical $c$-monotonicity. Intimately
related is the notion of the $c$-conjugate $\phi^c$ of a function
$\phi$.

We also mention G.~Kellerer's fundamental work on the duality theory; in \cite{Kell84} he established that $P=D$ provided that $c:X\times Y\to [0,\infty]$ is lower semi-continous, or merely Borel-measurable and uniformly bounded.

The seminal paper \cite{GaMc96} proves (among many other results)
that we have a positive answer to question (CC) above in the
following situation: every cyclically $c$-monotone transport plan
is optimal provided that the cost function $c$ is continuous and
$X,Y$ are compact subsets of $\R^n$. In \cite[Problem
2.25]{Vill03} it is asked whether this extends to the case
$X=Y=\R^n$ with the squared euclidian distance as cost function.
This  was answered independently in \cite{Prat07} and
\cite{ScTe08}: the answer to (CC) is positive for general polish
spaces $X$ and $Y$, provided that the cost function $c:X\times
Y\to [0,\infty]$ is continuous (\cite{Prat07}) or lower
semi-continuous and finitely valued (\cite{ScTe08}). Indeed, in
the latter case, a transport plan is optimal  if and only  if it
is strongly $c$-monotone.


\bigskip

Let us briefly resume the state of the art pertaining to the five questions above.

As regards the most basic issue, namely (DG) pertaining to the question whether duality makes sense at all, this is analyzed in detail --- building on a
lot of previous literature --- in section 2 of the
accompanying paper \cite{BeLS09a}: it is shown there that, for a properly relaxed version of the primal problem, question (DG) has an affirmative answer in a
perfectly general setting, i.e.~for arbitrary Borel-measurable cost functions $c: X\times Y\to [0,\i ]$ defined on the product of two  polish spaces $X,Y$, equipped
with Borel probability measures $\mu,\nu$.

\medskip

As regards question (P) we find the following situation: if the cost function $c: X\times Y\to [0,\i ]$ is {\it lower semi-continuous}, the answer to question (P)
is always positive. Indeed, for an optimizing sequence $(\pi_n)^\i_{n=1}$ in $\Pi(\mu,\nu)$, one may apply Prokhorov's theorem to find a weak limit
$\hpi =\lim_{k\to\i} \pi_{n_k}$. If $c$ is lower semi-continuous, we get
\begin{align*}
\int\limits_{X\times Y} c \ d\hpi \le \lim\limits_{k\to\i} \int\limits_{X\times Y} c \ d\pi_{n_k},
\end{align*}
which yields the optimality of $\hpi$.

On the other hand, if $c$ fails to be lower semi-continuous, there is little reason why a primal optimizer should exist (see, e.g., \cite[Example 2.20]{Kell84}).

\medskip

As regards (D), the question of the existence of a dual optimizer is more delicate than for the primal case (P): it was shown in \cite[Theorem 3.2]{AmPr03} that, for
$c: X\times Y\to \R_+$, satisfying a certain moment condition, one may assert the existence of {\it integrable} optimizers $(\hphi ,\hpsi)$. However, if one drops
this moment condition, there is little reason why, for an optimizing sequence $(\phi_n ,\psi_n)^\i_{n=1}$ in (D) above, the $L^1$-norms should remain bounded.
Hence there is little reason why one should be able to find {\it integrable} optimizers $(\hphi ,\hpsi)$ as shown by easy examples (e.g.\ \cite[Examples 4.4, 4.5]{BeSc08}), arising
in rather regular situations.

Yet one would like to be able to pass to {\it some kind of limit $(\hphi ,\hpsi)$}, whether these functions are integrable or not. In the case when $\hphi$ and/or
$\hpsi$ fail to be integrable, special care then has to be taken to give a proper sense to \eqref{H4}.

This situation was the motivation for the introduction of the notion of {\it strong cyclical $c$-monotonicity} in \cite{ScTe08}:
this notion (see (SCC) above) characterizes the optimality of a given $\pi\in\Pi(\mu,\nu)$ in terms of a ``complementary slackness condition'', involving some $(\phi ,\psi)\in\Psi
(\mu,\nu)$, playing the role of a dual optimizer $(\hphi,\hpsi)$. The crucial feature is that {\it we do not need any integrability of the functions
$\phi$ and $\psi$} for this notion to make sense. It was shown in \cite{BeSc08} that, also in situations where there are no integrable optimizers $(\hphi,\hpsi)$,
one may find Borel measurables functions $(\phi,\psi)$, taking their roles in the setting of (SCC) above.

This theme was further developed in \cite{BeSc08}, where it was shown that, for $\mu\otimes\nu$-a.s.~finite, Borel measurable $c:X\times Y\to [0,\i ]$, one may
find Borel functions $\hphi :X\to [-\i ,+\i )$ and $\hpsi :Y\to [-\i ,\i ),$ which are dual optimizers if we interpret \eqref{H4} properly: instead of considering
\begin{align}\label{H11}
\int\limits_X \hphi \ d\mu +\int\limits_Y \hpsi \, d\nu,
\end{align}
which needs integrability of $\hpsi$ and $\hpsi$ in order to make sense, we consider
\begin{align}\label{H12}
\int\limits_{X\times Y} (\hphi (x) +\hpsi (y))\,  d\pi (x,y),
\end{align}
where the transport plan $\pi\in\Pi(\mu,\nu)$ is assumed to have finite transport cost $\int_{X\times Y} c(x,y) d\pi (x,y) <\i$. If \eqref{H11} makes
sense, then its value coincides with the value of \eqref{H12}; the crucial feature is that, \eqref{H12} also makes sense in cases when \eqref{H11} does not make
sense any more as shown in \cite[Lemma 1.1]{BeSc08}. In particular, the value of \eqref{H12} does not depend on the choice of the transport plan $\pi\in\Pi(\mu,\nu)$,
provided $\pi$ has finite transport cost $\int_{X\times Y} c(x,y) d\pi(x,y) <\i$.

\bigskip

Summing up the preceding discussion on the existence (D) of a dual optimizer $(\hphi,\hpsi)$: this question has a -- properly interpreted -- positive answer
provided that the cost function $c:X\times Y \to [0,\i]$ is $\mu\otimes\nu$-a.s.~finite (\cite[Theorem 2]{BeSc08}).

But things become much more complicated if we pass to cost functions
$c:X\times Y \to [0,\i]$ assuming the value $+\i$ on possibly ``large'' subsets of $X\times Y$.

In \cite[Example  4.1]{BeLS09a} we exhibit an example, which is a variant of an example due to G.~Ambrosio and A.~Pratelli \cite[Example 3.5]{AmPr03}, of a lower semicontinuous cost function
$c:[0,1)\times [0,1)\to[0,\i]$, where $(X,\mu)=(Y,\nu)$ equals $[0,1)$ equipped with Lebesgue measure, for which there are no {\it Borel measurable} functions
$\hphi,\hpsi$ verifying $\hphi(x)+\hpsi(y)\le c(x,y)$, minimizing \eqref{H12} above.

In this example, the cost function $c$ equals the value $+\i$ on
``many'' points of $X\times Y=[0,1)\times [0,1).$
In fact, for each $x\in[0,1[$, there are precisely two points $y_1 ,y_2 \in[0,1[$ such that $c(x,y_1) <\i$ and $c(x,y_2)<\i$, while for all other
$y\in[0,1[$, we have $c(x,y)=\i$. In addition, there is an optimal transport plan $\hpi\in\Pi(\mu,\nu)$ whose support equals the set $\{(x,y)\in[0,1)\times [0,1):
c(x,y) <\i\}.$

In this example one may observe the following phenomenon: while there {\it do not} exist Borel measurable functions $\hphi :[0,1)\to [-\i ,+\i)$ and
$\hpsi :[0,1)\to[-\i, \i)$ such that $\hphi (x)+\hpsi(y)=c(x,y)$ on $\{c(x,y)<\i\}$, there {\it does} exist a Borel function $\hh:[0,1)\times [0,1)
\to[-\i,\i)$ such that $\hh(x,y)=c(x,y)$ on $\{c(x,y)<\i\}$ and such that $\hh(x,y)=\lim_{n\to\i} (\phi_n (x)+\psi_n(y))$ where $(\phi_n,\psi_n)^\i_{n=1}$
are properly chosen, bounded Borel functions. The point is that the limit holds true (only) in the norm of $L^1([0,1[\times [0,1[, \hpi)$ as well as $\hpi$-a.s.

In other words, in this example we are able to identify some kind of dual optimizer $\hh\in L^1 ([0,1)\times [0,1),\hpi)$ which, however, is not of the
form $\hh(x,y)=\hphi(x)+\hpsi(y)$ for some Borel functions $(\hphi,\hpsi)$, but only a $\hpi$-a.s.~limit of such functions $(\phi_n(x)+\psi_n(y))^\i_{n=1}$.

In \cite[Theorem 4.2]{BeLS09a} we established a result which shows
that much of the positive aspect of this phenomenon, i.e.~the
existence of an optimal $\hh
\in L^1(\hpi)$,
encountered in the context of the above example, can be carried over to a general setting.
For the convenience of the reader we restate this theorem and the notations required to   formulate it.

Fix a finite transport plan $\pi_0\in \Pi(\mu,\nu,c) := \left\{
\pi\in\Pi(\mu,\nu)  : \int_{X\times Y} c\,d\pi <\infty \right\}$.
We denote by $\Pi^{(\pi_0)}(\mu,\nu)$ the set of elements
$\pi\in\Pi(\mu,\nu)$ such that $\pi \ll \pi_0$ and $\big\|
\frac{d\pi}{d\pi_0} \big\|_{L^\infty(\pi_0)}<\infty$. Note that
$\Pi^{(\pi_0)}(\mu,\nu)=\Pi(\mu,\nu)\cap L^\infty(\pi_0) \subseteq
\Pi(\mu,\nu,c)$. We shall replace the usual Kantorovich
optimization problem over the set $\Pi(\mu,\nu,c)$
 by the optimization over the smaller
set $\Pi^{(\pi_0)}(\mu,\nu)$. Its value is
\begin{align}
 \label{Walters7}
 P^{(\pi_0)}
 &= \inf \{ \langle c,\pi\rangle =\textstyle{\int} c\, d\pi  : \pi\in\Pi^{(\pi_0)}(\mu,\nu)\}.
\end{align}
As regards the dual problem, we define, for $\eps>0$,
\begin{eqnarray}
    D^{(\pi_0,\eps)}
    &=& \sup\Big\{\int \phi\,d\mu+\int\psi\,d\nu:\  \phi\in L^1(\mu),\psi\in
    L^1(\nu),\nonumber\\
     &&\qquad\qquad\qquad    \int_{X\times Y} (\phi(x)+\psi(y) -c(x,y))_+\,d\pi_0\le\eps
        \Big\}\quad \textrm{and}\nonumber\\
    \label{Walters7D}
    D^{(\pi_0)}&=&\lim_{\eps\to 0}D^{(\pi_0,\eps)}.
\end{eqnarray}
Define the ``summing'' map $S$ by
\begin{align*}
S:  L^1(X,\mu) \times L^1(Y,\nu) &\to L^1 (X\times Y,\pi_0)\\
(\varphi,\psi) &\mapsto \varphi\oplus\psi,
\end{align*}
where $\phi\oplus \psi$ denotes the function $\phi(x)+\psi(y)$ on $X\times Y$.
Denote by $L_S^1(X\times Y,\pi_0)$ the $\|.\|_1$-closed linear subspace of $L^1(X\times Y,\pi_0)$ spanned by
$S(L^1(X,\mu)\times L^1(Y,\nu))$. Clearly $L_S^1(X\times Y,\pi_0)$ is a Banach space under the norm $\|.\|_1$
induced by $L^1(X\times Y,\pi_0)$.

We shall also need the bi-dual $L_S^1(X\times Y,\pi_0)^{**}$ which
may be identified with a subspace of $L^1(X\times Y,\pi_0)^{**}$.
In particular, an element $h\in L_S^1(X\times Y,\pi_0)^{**}$ can
be decomposed into $h=h^r+h^s,$ where $h^r\in L^1(X\times Y,\pi_0)
$ is the regular part of the finitely additive measure $h$ and
$h^s $ its purely singular part.

\begin{theorem}\label{LeonardDuality}
Let $c: X\times Y \to [0,\infty]$ be Borel measurable, and let
$\pi_0\in\Pi(\mu,\nu,c)$ be a finite transport plan.
We have
\begin{align}\label{NiceEq}
P^{(\pi_0)}=D^{(\pi_0)}.
\end{align}
There is an element $\hat h\in L_S^1(X\times Y,\pi_0)^{**}$ such
that $\hat h\leq c$ 
and
$$D^{(\pi_0)}=\langle \hat h, \pi_0\rangle.$$
If $  \pi \in \Pi^{(\pi_0)}(\mu,\nu)$ (identifying $
\pi$ with $\frac{d  \pi}{d\pi_0}$) satisfies $\int c\, d \pi \le
P^{(\pi_0)} +\alpha$ for some $\alpha \geq 0$, then
\begin{equation}\label{eq-1}
    |\langle \hat h^s,   \pi\rangle |\leq \alpha.
\end{equation}
In particular, if $  \pi$ is an optimizer of \eqref{Walters7},
then $\hat h^s $ vanishes on the set $\{\frac{d \pi}{d\pi_0}>
0\}$.
\\
In addition, we may find a sequence of elements $(\varphi_n,\psi_n)\in
L^1(\mu)\times L^1(\nu)$  such that
\begin{align*}
\varphi_n\oplus\psi_n\to \hat h^r,\ \pi_0\mbox{-a.s.},\qquad
\|(\varphi_n\oplus\psi_n-\hat h^r)_+\|_{L_1(\pi_0)}\to 0\
 \end{align*}
 and
\begin{align}\label{Walters9}
\lim_{\delta\to0}\sup_{A\subseteq X\times Y,\pi_0(A)<\delta}\lim_{n\to \infty}
-\langle (\varphi_n\oplus\psi_n)\I_A,\pi_0\rangle = \|\hat h^s\|_{L_1(\pi_0)^{**}}.
\end{align}
\end{theorem}

The assertion of the theorem extends the phenomenon of \cite[Example  4.1]{BeLS09a} to a general setting. There is, however, one additional complication, as compared
to the situation of this specific example: in the above theorem we only can assert that we find the optimizer $\hh$ in $L^1(\hpi)^{**}$ rather than
in $L^1(\hpi)$. The question arises whether this complication is indeed unavoidable. The purpose of the subsequent section is to construct an example showing
that the phenomenon of a non-vanishing singular part $\hh^s$ of $\hh=\hh^r + \hh^s$ may indeed arise in the above setting. In
addition, the example gives a good illustration of the subtleties of the situation described by the theorem above.

\section{The singular part of the dual optimizer}\label{SingPart}

In this section we refine the construction of Examples 4.1 and 4.3 in \cite{BeLS09a} (which in turn are variants of an example due to G.~Ambrosio and
A.~Pratelli \cite[Example 3.2]{AmPr03}). We assume that the reader is familiar with these examples and freely use the notation from this paper.
\\
In particular, for an irrational $\alpha\in[0,1)$ we write, for $k\in\Z,\footnote{In \cite{BeLS09a} the constructions are carried out for $\N$ instead of $\Z$, but for our purposes the latter choice turns out to be better suited.}$
\begin{align}\label{BB1}
\begin{split}
\varrho_k(x)=1  + \# \{0\le i<k:x\oplus i\alpha\in [0,\tfrac12 )\} \\
- ~ \# \{0\le i<k: x\oplus i\alpha\in [\tfrac12 ,1)\}, \hspace{-0.1cm}
\end{split}
\end{align}
where, for $k<0$, we mean by $0\le i<k$ the set $\{k+1,k+2,\ldots ,0\}$ and $\oplus$ denotes addition modulo 1. We also recall that the function $h:[0,1)\times[0,1)\to\Z$ is defined in \cite[Example  4.3]{BeLS09a} as
\begin{align}\label{dxy}
h(x,y)=
\begin{cases}
\varrho_k(x), & k\in\Z \ \mbox{and} \ y=x\oplus k\alpha \\
\i, & \mbox{otherwise.}
\end{cases}
\end{align}
In  \cite[Example  4.3]{BeLS09a} we considered the $[0,\i]$-valued cost function $c(x,y):= h_+(x,y).$
We now construct an example restricting $h_+(x,y)$ to a certain subset of $[0,1)\times [0,1).$
\begin{example}
{\it
There is an irrational $\alpha\in [0,1)$ and a map $\tau :[0,1)\to\Z$ such that, for
\begin{align*}
&\Gamma_0=\{(x,x),x\in [0,1)\}, \\
&\Gamma_1=\{(x,x\oplus \alpha):x\in [0,1)\}, \\
&\Gamma_\tau=\{(x,x\oplus \tau(x)\alpha):x\in [0,1)\}
\end{align*}
and letting
\begin{align*}
c(x,y)=
\begin{cases}
h_+(x,y),  & \mbox{for} \ x\in \Gamma_0 \cup \Gamma_1 \cup \Gamma_\tau \\
\i , & \mbox{otherwise}
\end{cases}
\end{align*}
the following properties are satisfied.
\begin{enumerate}
\item
The maps
\begin{align*}
T^0_\alpha (x)=x, \qquad T^1_\alpha (x) =x\oplus \alpha, \qquad
T^{(\tau)}_\alpha (x)=x\oplus (\tau(x)\, \alpha)
\end{align*}
are measure preserving bijections from $[0,1)$ to $[0,1).$
Denote by $\pi_0,\pi_1,\pi_\tau$ the corresponding transport plans in $\Pi(\mu,\nu)$, i.e.
\begin{align*}
\pi_0=(id, id)_\# \mu, \quad \pi_1=(id, T_\alpha)_\#\mu, \quad
\pi_\tau=(id, T^{(\tau)}_\alpha)_\#\mu,
\end{align*}
and let $\pi=(\pi_0+\pi_1+\pi_\tau)/ 3.$
\item The transport plans $\pi_0$ and $\pi_1$ are optimal while $\pi_\tau$ is not. In fact, we have
\begin{align}\label{Seite10}
\langle c,\pi_0\rangle=\langle c,\pi_1\rangle=1 \ \mbox{while} \
\langle c,\pi_\tau\rangle \geq\langle h,\pi_\tau\rangle >1.
\end{align}
\item There is a sequence $(\varphi_n, \psi_n)_{n=1}^\infty$ of bounded Borel functions such that
\begin{align}
    & (a) \ \varphi_n(x)+\psi_n(y)\leq c(x,y), \quad \mbox{for }x\in X,y\in
    Y,\label{iii,a}\\
    & (b) \ \lim_{n\to\i} \Big(\int_X\phi_n(x) \ d\mu(x)+\int_Y\psi_n(y), d\nu(y)\Big)=1,
    \label{iii,b}\\
    & (c) \ \lim_{n\to\i}(\varphi_n(x)+\psi_n(y)) = h(x,y), \ \pi\mbox{-almost surely.}\label{iii,c}
\end{align}

\item Using the notation of  \cite[Theorem  4.2]{BeLS09a} we find
that for each dual optimizer $\hh\in L^1(\pi)^{**},$ which
decomposes as $\hh=\hh^r+\hh^s$ into its regular part $\hh^r \in
L^1 (\pi)$ and its purely singular part $\hh^s\in L^1 (\pi)^{**},$
we have
\begin{align}\label{tag4}
\hh^r=h, \ \pi\mbox{-a.s.,}
\end{align}
and the singular part $\hh^s$ satisfies
$\|\hh^s\|_{L^1(\pi)^{**}}=\langle h,\pi_\tau\rangle-1>0$. In
particular, the singular part $\hh^s$ of $\hh$ does not vanish.
The finitely additive measure $\hh^s$ is supported by
$\Gamma_{\tau}$, i.e.\ $\langle \hh^s, \I_{\Gamma_0}+\I_{\Gamma_1}
\rangle=0.$
\end{enumerate}
}
\end{example}

\medskip

We shall use a special irrational $\alpha \in [0,1)$, namely
$$\alpha =\sum_{j=1}^\infty \frac 1{M_j}, $$ where $M_j= m_1m_2\ldots m_j= M_{j-1}m_j,$ and $(m_j)_{j=1}^\infty$ is a sequence of prime numbers $m_j\geq 5$
tending sufficiently fast to infinity, to be specified below. We let
$$\alpha_n:= \sum_{j=1}^n\frac1{M_j}, $$ which, of course, is a rational number.

We will need the following lemma. We thank Leonhard Summerer for showing us the proof of Lemma \ref{RelPrimeLemma}.

\begin{lemma}\label{RelPrimeLemma}
It is possible to choose a sequence $m_1, m_2,\ldots$ of primes growing arbitrarily fast to infinity, such that with $M_1 =m_1, M_2=m_1\cdot m_2, \ldots, M_n=m_1 \cdots m_n,\ldots$ we have, for each
$n\in \N,$
$$\sum_{j=1}^n\frac1{M_j}=\frac{P_n}{M_n},$$ with  $P_n$ and $M_n$ relatively prime.
\end{lemma}

\begin{proof}{}
We have
$$\sum_{j=1}^n\frac1{M_j}= \frac{m_2 \ldots m_{n} + \ldots +m_{n}+1}{M_n}=:\frac{P_n}{M_n},$$
thus $P_n$ and $M_n$ are relatively prime, if and only if
\begin{align}\label{mFirst}
m_1& \nmid &  m_2 \cdots m_{n}+\ & m_3 \cdots m_{n} + \ldots  + && m_{n}+1&\\
m_2& \nmid   &                              & m_3 \cdots m_{n} + \ldots  + && m_{n}+1&\\
      & \vdots      & &  &&  \vdots & \\                  \label{mLast}
m_{n-1} & \nmid &    & &&  m_n+1.&
\end{align}
We claim that these conditions are, e.g., satisfied provided that we choose $m_1,m_2,\ldots$ such that $m_i \geq 3$ and
\begin{align}\label{KongOne}
m_{i+1}&\equiv +1 ~(m_i)\\
m_{i+j}&\equiv - 1 ~ (m_i) \mbox{ if $j\geq 2$}.\label{KongTwo}
\end{align}
for all $i\geq 1$. Indeed \eqref{KongOne}, \eqref{KongTwo} imply that for $k\in\{1,\ldots,n-1\}$ we have modulo $( m_k)$
\begin{align*}
m_{k+1}\cdots m_n &+& m_{k+2}\cdots m_n& +&m_{k+3}\cdots m_n&  +&\ldots&+& m_n & +&1 & \equiv\\
 (\pm   1)           &+&  (\pm 1)                  & +& (\mp1)&  +&\ldots&+& (-1) & +&(+1)  &,
\end{align*}
where in the second line the $(n-k+1)$ summands start to alternate after the second term. Thus, for even $n-k$, this amounts to
\begin{align*}
m_{k+1}\cdots m_n &+& m_{k+2}\cdots m_n& +&m_{k+3}\cdots m_n&  +&\ldots&+& m_n & +&1 & \equiv\\
 (-   1)           &+&  (-1)                  & +& (+1)&  +&\ldots&+& (-1) & +&(+1)  &\equiv -1,
\end{align*}
while we obtain, for odd $n-k$,
\begin{align*}
m_{k+1}\cdots m_n &+& m_{k+2}\cdots m_n& +&m_{k+3}\cdots m_n&  +&\ldots&+& m_n & +&1 & \equiv\\
 (+  1)           &+&  (+1)                  & +& (-1)&  +&\ldots&+& (-1) & +&(+1)  &\equiv +2
\end{align*}
Hence \eqref{mFirst}-\eqref{mLast} are satisfied as the $m_n$ where chosen such that $m_n > 2$.

We use induction to construct a sequence of primes  satisfying \eqref{KongOne} and \eqref{KongTwo}. Assume that $m_1,\ldots, m_i$ have been defined.
By the chinese remainder theorem  the system of congruences
$$ x\equiv -1 \ (m_1), \ldots,\quad x\equiv -1 \ (m_{i-1}),\quad x\equiv +1 \ (m_i)$$
has a solution $x_0 \in \{1, \ldots, m_1\ldots m_i\}$. By
Dirichlet's theorem, the arithmetic progression $x_0+k m_1\ldots
m_i, k\in\N$ contains infinitely many primes, so we may pick one
which is as large as we please. The induction continues.
   \end{proof}

For $\beta\in [0,1),$ denote by $T_{\beta}:[0,1)\to [0,1),
T_\beta(x):= x \oplus\beta$ the addition of $\beta$ modulo 1. With
this notation we have $T_{\alpha_n}^{M_n}=id$ and, by Lemma
\ref{RelPrimeLemma}, it is possible to choose $m_1,\ldots,m_n$ in
such a way that $M_n $ is the smallest such number in $\N$. Our
aim is to construct a function $\tau:[0,1)\to \Z$ such that the
map
\begin{equation*}
T_\alpha^{(\tau)}: \left\{
    \begin{array}{rcl}
      [0,1) & \to & [0,1) \\
      x & \mapsto & T_\alpha^{(\tau)} (x)
=T_\alpha^{\tau(x)}(x) \\
    \end{array}
    \right.,
 \end{equation*}
defines, up to a $\mu$-null set, a measure preserving bijection on
$[0,1)$, and such that the corresponding transport plan
$\pi_\tau\in \Pi(\mu,\nu)$, given by $\pi_\tau=(id,
T_\alpha^{(\tau)})_{\#}\mu$, has the properties listed above with
respect to the cost function $c(x,y)$ which is the restriction of
the function $h_+(x,y)$ to $\Gamma_0\cup\Gamma_1\cup\Gamma_\tau.$
We shall do so by an inductive procedure, defining bounded
$\Z$-valued functions $\tau_n$ on $[0,1)$ such that the maps
$T_{\alpha_n}^{(\tau_n)}$ are measure preserving bijections on
$[0,1)$. The map $T_\alpha ^{(\tau)}$ then will be the limit of
these $T_{\alpha_n}^{(\tau_n)}$.

\medskip\noindent{\bf Step n=1:} Fix a prime $M_1=m_1\geq 5$, so
that $\alpha_1= \frac1{M_1}$. Define
$$I_{k_1}:=\left[\tfrac{k_1-1}{M_1},\tfrac{k_1}{M_1}\right), \ k_1=1,\ldots, M_1,$$
so that $(I_{k_1})_{k_1=1}^{M_1}$ forms a partition of $[0,1)$  and $T_{\alpha_1}$ maps
$I_{k_1}$ to $ I_{k_1+ 1}$, with the convention  $M_1 +1=1$. We also introduce the notations
$$L^1:=[0, \tfrac12 - \tfrac{1}{2M_1}) \ \mbox{and} \
R^1:=[\tfrac12 + \tfrac{1}{2M_1},1)$$ for the segments left and right of the middle interval $$\IM^1:=I_{(M_1+1)/2}=[\tfrac12 - \tfrac1{2M_1},
\tfrac12 +\tfrac1{2M_1}).$$
We define the functions $\varphi^1, \psi^1$ on $[0,1)$ such that
$\varphi^1(x)+\psi^1(x)\equiv 1$ and
$$ \varphi^1(x)+\psi^1(T_{\alpha_1}(x))=
\begin{cases}
0& x \in L^1\\
1& x \in \IM^1 \\
2& x \in R^1\\
\end{cases}
$$
which leads to the relation
$$ \varphi^1(T_{\alpha_1}(x))=\varphi^1(x)+
\begin{cases}
1,& x \in L^1,\\
0,& x \in \IM^1, \\
-1,& x \in R^1. \\
\end{cases}
$$
Making the choice $\varphi^1\equiv 0$ on $I_{1}$ this leads to
\begin{align}\label{choice}
\varphi^1(x)&=\begin{cases}
k_1-1, & x \in I_{k_1}, k_1\in\{1,\ldots, (M_1+1)/2\},\\
M_{1}+1-k_1, & x \in I_{k_1}, k_1\in\{ (M_1+3)/2, M_1\},
\end{cases}\\
\psi^1(x)&=1-\varphi^1(x).\nonumber
\end{align}
The function $\varphi^1$ starts at $0$, increases until the middle interval, stays constant when stepping to the interval right of the middle, and then
decreases, reaching $1$ on the final interval $I_{M_1}$.

The idea is to define the map $\tau_1 : [0,1)\to \Z$ in such a way
that the map
\begin{equation*}
T_{\alpha_1}^{(\tau_1)}: \left\{
    \begin{array}{rcl}
      [0,1) & \to & [0,1) \\
      x & \mapsto &T_{\alpha_1}^{\tau_1(x)}(x) \\
    \end{array}
    \right.,
 \end{equation*}
 is a measure preserving bijection enjoying the following property: the map
    $$
x\mapsto \phi^1 (x) +\psi^1( T_{\alpha_1}^{(\tau_1)} (x)),
    $$
equals the value two on a large set while it has concentrated a negative mass which is close to $-1$ on a small set.
\\
This can be done, e.g., by shifting the first interval $I_1$ to
the interval $I_{(M_1 -1)/2}$, which is left of the middle one,
while  we shift the intervals $I_2,\dots ,I_{(M_1-1)/2}$ by one
interval to the left. On the right hand side of $[0,1)$ we proceed
symmetrically while the middle interval simply is not moved.

\par\medskip
\input{devilish1}

\noindent The step function is $\varphi^1$ and the arrows indicate
the action of $T_{\alpha_1}^{(\tau_1)}.$ This figure corresponds
to the value $M_1=11.$
\par\bigskip

More precisely, we set
\begin{align}\label{cal}
\tau_1(x)=
\begin{cases}
\tfrac{M_1-3}2,& x \in I_{1},\\
-1,& x \in I_{k_1}, k_1\in\{2,\ldots, (M_1-1)/2\},\\
0, & x\in I_{(M_1+1)/2},\\
1,& x \in I_{k_1}, k_1\in\{ (M_1+3)/2,\dots , M_1\},\\
-\tfrac{M_1-3}2, &x \in I_{M_1}.
\end{cases}
\end{align}
Then $T_{\alpha_1}^{(\tau_1)}$ induces a permutation of the intervals $(I_{k_1})^{M_1}_{k_1=1}$ and a short calculation shows that
\begin{align}\label{shortcal}
\varphi^1(x)+\psi^1(T_{\alpha_1}^{(\tau_1)}(x))=
\begin{cases}
2, & x \in I_{k_1}, k_1\in \{2,\dots ,(M_1-1)/2,\\
    &\phantom{x \in I_{k_1}, k_1\in} (M_1+3)/2,\dots ,M_1-1\},\\
-\tfrac{M_1-5}2,& x \in I_{k_1}, k_1=1, M_1, \\
1,& x \in I_{(M_1+1)/2}.\\
\end{cases}
\end{align}
Next figure is a representation of this ``quasi-cost'' at level
$n=1,$ with the same value $M_1=11$ as in  Figure 1.
\input{devilish2}
\par\bigskip

\noindent\textit{Assessment of Step $n=1.$}\ Let us resume what we
have achieved in the first induction step. For later use we
formulate things only in terms of $\phi^1 (\cdot)$ rather than
$\psi^1(\cdot)= 1-\phi^1(\cdot).$
\\
For the set $J^g_1=\{2,\dots ,\frac{M-1}{2}\} \cup \{\frac{M+3}{2},\dots ,M_1-1\}$ of
{\it ``good\footnote{We use the term ``good'' rather than ``regular'' as the abbreviation $r$ is already taken by the word ``right''.}}
indices''  we have
\begin{align}\label{Seite14}
\phi^1(x)-\phi^1 (T_{\alpha_1}^{(\tau_1)} (x)) =1, \qquad x\in I_{k_1}, k_1\in J_1^g,
\end{align}
while for the set $J_1^s=\{1,M_1\}$ of {\it ``singular indices''}  we have
\begin{align}\label{Seite14a}
\phi^1 (x)-\phi^1(T_{\alpha_1}^{(\tau_1)}(x))=-\frac{M_1-3}{2}, \qquad x\in I_{k_1}, \ k_1\in J^s,
\end{align}
so that
\begin{align*}
 \sum\limits_{k_1\in J_1^s}\ \int\limits_{I_{k_1}} [\phi^1(x)-\phi^1(T_{\alpha_1}^{(\tau_1)}(x))]\, dx = -\frac{M_1-3}{2} \ \frac{2}{M_1}
=-1+\frac{3}{M_1}.
\end{align*}
For the middle interval $I^1_\text{middle} =I_{(M_1+1)/2}$ we have $\phi^1(x)-\phi^1 (T^{\tau_1}_{\alpha_1}(x))=0.$
\\
We also note for later use that, for $x\in[0,1)$, the orbit $(T^i_{\alpha_1}(x))^{\tau_1(x)}_{i=1}$ never visits $\IM^1.$ Here we mean that $i$ runs
through $\{\tau_1(x), \tau_1(x)+1,\ldots ,-1\}$ when $\tau_1(x)<0$ and runs through the empty set when $\tau_1(x)=0.$

\medskip\noindent{\bf Step n=2:}
We now pass from $\alpha_1=\frac1{M_1}$ to
$\alpha_2=\frac1{M_1}+\frac1{M_2}$, where $M_2=M_1m_2 =m_1 m_2$
and where $m_2$, to be specified below, satisfies the relations of
Lemma \ref{RelPrimeLemma} and is large compared to $M_1$. For
$1\le k_1 \le M_1$ and $1\le k_2\le m_2$ we denote by
$I_{k_1,k_2}$ the interval
$$
I_{k_1,k_2} =\left[ \tfrac{k_1-1}{M_1} +\tfrac{k_2-1}{M_2} , \tfrac{k_1-1}{M_1} +\tfrac{k_2}{M_2}\right).
$$
Similarly as above we will also use the  notations  $L^2=[0, \frac12 - \frac1{2M_2}),
 R^2=[\frac12 + \frac1{2M_2}, 1),$ and $I^2_{\text{middle}}=I_{(M_1+1)/2,(m_2+1)/2}=[\frac12 -\frac1{2M_2}, \tfrac12 + \tfrac1{2M_2})$.
\\
We now define functions $\varphi^2, \psi^2$ such that
$\varphi^2(x)+\psi^2(x)\equiv 1$ and
$$
\varphi^2(x)+\psi^2(T_{\alpha_2}(x))=
\begin{cases}
0,& x\in L^2,\\
1,& x\in I^2_{\text{middle}},\\
2,& x\in R^2.
\end{cases}
$$
This is achieved if we set, e.g., $\varphi^2\equiv 0$ on $I_{1,1}$, and
\begin{align}\label{K1a}
 \varphi^2(T_{\alpha_2}(x))&=\varphi^2(x)+
\begin{cases}
1& x \in L^2,\\
0& x \in \IM^2, \\
-1& x \in R^2,\\
\end{cases}
\\
\psi^2(x)&=1-\varphi^2(x).\nonumber
\end{align}
Yet another way to express this is to say that for $j\in
\{0,\ldots,M_2-1\}$ we have
\begin{equation}\label{21}
\begin{split}
\varphi^2(T_{\alpha_2}^j(x))=
    \#&\{i\in\{0,\ldots,j-1\}: T_{\alpha_2}^i(x) \in L^2 \}\\
    &-\#\{i\in \{0,\ldots,j-1\}: T_{\alpha_2}^i(x) \in R^2 \}
\end{split}
,\quad x\in I_{1,1},
\end{equation}
in analogy to \eqref{BB1}.

While the function $\phi^1(x)$ in the first induction step was
increasing from $I_1$ to $I_{(M_1+1)/2}$ and then decreasing from
$I_{(M_1+3)/2}$ to $I_{M_1}$, the function $\phi^2(x)$ displays a
similar feature on each of the intervals $I_{k_1}$: roughly
speaking, i.e.\ up to terms controlled by $M_1$, it increases on
the left half of each such interval and then decreases again on
the right half. The next lemma makes this fact precise. We keep in
mind, of course, that $m_2$ will be much bigger than $M_1$.

\begin{lemma}[Oscillations of $\phi^2$]\label{L2.2}
The function $\phi^2$ defined in \eqref{K1a} has the following properties.\\
\begin{enumerate}
    \item $|\phi^2(x)-\phi^2(x\oplus\tfrac1{M_2})|\le \ 4M^2_1, \quad x\in[0,1).$
    \item For each $1\le k'_1, k^{''}_1 \le M_1$ we have
$${\phi^2}_{|I_{k'_1, (m_2+1)/2}} -{\phi^2}_{|I_{k''_1,1}} \geq \tfrac{m_2}{2M_1}-10M^3_1.$$
\end{enumerate}
\end{lemma}

\begin{proof} Let us begin with the proof of (i).

\noindent$\bullet$ \textit{ Proof of (i).}\ While
$T^{M_1}_{\alpha_1}= id$ holds true, we have that
$T^{M_1}_{\alpha_2}$ is only close to the identity map. In fact,
as $T_{\alpha_2} (x) =x\oplus \tfrac{m_2+1}{M_2},$ we have
\begin{align}\label{K2}
& T^{M_1}_{\alpha_2} (x) =x\oplus \tfrac{M_1}{M_2}.
\end{align}
Somewhat less obvious is the fact that $T^{m_2-2}_{\alpha_2}$ also
is close to the identity map. In fact
\begin{align}\label{K3}
T^{m_2-2}_{\alpha_2}(x)=x\ominus \tfrac2{M_2}.
\end{align}
Indeed, by \eqref{KongOne} applied to $i=1$, there is $c\in\N$
such that $m_2=cM_1+1$. Hence
\begin{align*}
T^{m_2-2}_{\alpha_2} (x) & = x\oplus (m_2-2)\frac{m_2+1}{M_2}
\\ &  = x \oplus (cM_1-1) \frac{m_2+1}{M_2}
\\ &  = x \oplus \frac{cM_2-m_2+(m_2-2)}{M_2} =x\ominus \frac2{M_2}.
\end{align*}
Here is one more remarkable feature of the map $T^{m_2-2}_{\alpha_2}$.

\bigskip
\noindent {\it Claim:} {\it For $x\in[0,1)$ the orbit $(T^i_\alpha
(x))^{m_2-2}_{i=1}$ visits the intervals $L^2=[0,\tfrac12
-\tfrac1{2M_2})$ and $R^2=[\tfrac12 +\tfrac1{2M_2} ,1)$
approximately equally often. More precisely, the difference of the
visits of these two intervals is bounded in absolute value by
$4M_1$.}

\medskip
Indeed, by Lemma \ref{RelPrimeLemma}, the orbit $(T^i_{\alpha_2}
(x))^{M_2}_{i=1}$ visits each of the intervals $I_{k_1,k_2}$
exactly one time so that it visits $L^2$ and $R^2$ equally often,
namely $\tfrac{M_2-1}{2}$ times. The $M_1$ many disjoint subsets
$\left(T_{\alpha_2}^{j(m_2-2)}\left(T^i_{\alpha_2}
(x)\right)^{m_2-2}_{i=1}\right)^{M_1}_{j=1}$ of this orbit are
obtained by shifting them successively by $2/M_2$ to the left
\eqref{K3}. As the difference
    $(T^i_{\alpha_2} (x))^{M_2}_{i=1}
\setminus \left(T_{\alpha_2}^{j(m_2-2)}\left(T^i_{\alpha_2}
(x)\right)^{m_2-2}_{i=1}\right)^{M_1}_{j=1}$ consists only of
$2M_1$ many points we have that the difference of the visits of
$\left(T_{\alpha_2}^{j(m_2-2)}\left(T^i_{\alpha_2}
(x)\right)^{m_2-2}_{i=1}\right)^{M_1}_{j=1}$ to $L^2$ and $R^2$ is
bounded by $4M_1$. This implies that the difference of the visits
of $(T^i_{\alpha_2} (x))^{m_2-2}_{i=1}$ to $L^2$ and $R^2$ can be
estimated by $4M_1$ too: indeed, if this orbit visits $4M_1+k$
many times $L^2$ more often then $R^2$ (or vice versa) for some
$k\geq 0$, then
$(T^{m_2-2}_{\alpha_2}(T^i_{\alpha_2}(x)))^{m_2-2}_{i=1}$ visits
$L^2$ at least $4M_1+k-4$ many times more often than $R^2$ etc.
and finally
$(T_{\alpha_2}^{M_1(m_2-2)}(T^i_{\alpha_2}(x)))^{m_2-2}_{i=1}$
visits $L^2$ at least $k$ many times more often than $R^2$ which
yields a contradiction. Hence we have proved the claim.

    To prove assertion (i) note that by \eqref{K2} and
\eqref{K3}
\begin{align}\label{p16}
T_{\alpha_2}^{\tfrac{M_1-1}{2}(m_2-2)} \circ T_{\alpha_2}^{M_1} (x) = x\oplus \tfrac1{M_2}
\end{align}
We deduce from the claim that the difference of the visits of the orbit $(T^i_{\alpha_2})^{\tfrac{M_1-1}{2}(m_2-2)+M_1}_{i=0}$ to $L^2$ and $R^2$ is bounded
in absolute value by $\tfrac{M_1-1}{2}(4M_1)+M_1$ which proves (i).

\medskip\noindent$\bullet$ \textit{ Proof of (ii).}\
As regards (ii) suppose first $k'_1=k''_1=:k_1$. Note that, for
$x\in I^{\text{left}}_{k_1}: = [\tfrac{k_1-1}{M_1},
\tfrac{k_1-1}{M_1}+ \tfrac1{2M_1}-\tfrac{2M_1+1}{2M_2})$, we have
that the orbit $(T^i_{\alpha_2}(x))^{M_1-1}_{x=0}$ visits $L^2$
one time more often than $R^2$, namely $\tfrac{M_1+1}{2}$ versus
$\tfrac{M_1-1}{2}$ times. If we start with $x\in I_{k_1,1}$ then,
for $1\le j< \tfrac{m_2}{2M_1} -1$ we have that
$T^{jM_1}_{\alpha_2}(x) \in I^{\text{left}}_{k_1}.$ Hence, for the
orbit
$(T^i_{\alpha_2})_{i=0}^{(\lfloor\tfrac{m_2}{2M_1}\rfloor-1)M_1}$,
the difference of the visits to the interval $L^2$ and $R^2$
equals $\lfloor\tfrac{m_2}{2M_1}\rfloor-1$, the integer part of
$\tfrac{m_2}{2M_1}-1$. Combining this estimate with the estimate
(i) as well as the fact that the distance between $x\oplus
\left(\lfloor\tfrac{m_2}{2M_1}\rfloor-1\right)\tfrac{M_1}{M_2}$
and $x\oplus\tfrac{m_2-1}{2M_2}$ is bounded by
$\tfrac{2M_1-1}{M_2}$, we obtain, for $x\in I_{k_1,1}$ and $y\in
I_{k_1,\tfrac{m_2+1}{2}},$ that
\begin{align*}
\phi^2(y)-\phi^2(x) & \geq
\phi^2(T_{\alpha_2}^{(\lfloor\tfrac{m_2}{2M_1}\rfloor-1)M_1}
(x))-\phi^2(x)
 -\Big|\phi^2(y)-\phi^2 (T_{\alpha_2}^{(\lfloor\tfrac{m_2}{2M_1}\rfloor-1)M_1} (x))\Big|
\\ &  \geq (\lfloor\tfrac{m_2}{2M_1}\rfloor-1)-  (2M_1-1) (4M^2_1)
\\ &  \geq \tfrac{m_2}{2M_1} -8M^3_1.
\end{align*}
Passing to the general case $1\le k'_1,k''_1\le M_1$ observe that $T_{\alpha_2}^{k''_1- k'_1}$ maps $I_{k'_1,\tfrac{m_2+1}{2}}$ to $I_{k''_1,
\tfrac{m_2+1}{2}+k''_1- {k'_1}}.$
Using again (i) we obtain estimate (ii).
\end{proof}

We now are ready to do the inductive construction for $n=2$. For
$m_2$ satisfying the conditions of Lemma 3.1 and to be specified
below, we shall define $\tau_2:[0,1)\to\{-\tfrac{M_2-1}{2},\dots
,0,\dots ,\tfrac{M_2-1}{2}\}$, where $M_2=m_2m_1$, such that the
map
\begin{equation*}
T^{(\tau_2)}_{\alpha_2}: \left\{
    \begin{array}{rcl}
      [0,1) & \to & [0,1) \\
      x & \mapsto & T^{(\tau_2)}_{\alpha_2}(x):= \ T^{\tau_2(x)}_{\alpha_2}(x) \\
    \end{array}
    \right.
 \end{equation*}
has the following properties.

\begin{enumerate}
\item
The measure-preserving bijection $T^{(\tau_2)}_{\alpha_2}:[0,1)\to [0,1)$ maps each interval $I_{k_1}$ onto $T^{(\tau_1)}_{\alpha_1} (I_{k_1}).$ It induces a permutation of the intervals
$I_{k_1,k_2},$ where $1\le k_1\le M_1, 1\le k_2\le m_2.$
\item
When $\tau_{2}(x) >0$, we have
\begin{align}\label{p24a}
T^i_{\alpha_{2}}(x) \notin I^{2}_{\text{middle}}, \qquad i=0,\ldots ,\tau_{2}(x),
\end{align}
and, when $\tau_{2}(x)<0$, we have
\begin{align}\label{p24b}
T^i_{\alpha_{2}}(x)\notin I^{2}_{\text{middle}}, \qquad i=\tau_{2}(x),\ldots ,0.
\end{align}
\item
On the {\it ``good''} intervals
$I_{k_1}$, where $k_1\in J_1^g=\{2,\dots ,\tfrac{M_1-1}{2}\} \cup \{\tfrac{M_1+3}{2}, \dots ,M_1-1\}$, for which we have, by \eqref{Seite14},
$$\phi^1 (x)-\phi^1 (T^{(\tau_1)}_{\alpha_1}(x))=1,$$ the function $\tau_2$ will satisfy the estimates
\begin{align}\label{p33}
\mu[I_{k_1} \cap \{\tau_{2}\neq\tau_{1}\}]\leq \tfrac{M_1}{m_2} \mu [I_{k_1}],
\end{align}
and
\begin{align}\label{K8a}
\sum\limits_{k_1\in J_1^g}\ \ \int_{I_{k_1}} |1-\phi^2 (x)+\phi^2
(T^{(\tau_2)}_{\alpha_2} (x))|\,dx < \frac{4M^2_1}{m_2}.
\end{align}
\item
On the {\it ``singular''} intervals $I_{k_1}$, where $k_1 \in J_1^s =\{1,M_1\}$, for which we have , by \eqref{Seite14a},
$$\phi^1 (x)-\phi^1 (T^{(\tau_1)}_{\alpha_1}(x))=-\frac{M_1-3}{2},$$
we split $\{1,\dots ,m_2\}$ into a set $J^{k_1,g}$ of {\it
``good''} indices, and a set $J^{k_1, s}$ of {\it ``singular''}
indices, such that
$$\phi^2(x)-\phi^2 (T^{(\tau_2)}_{\alpha_2}(x))=0,\quad \mbox{for} \ x\in I_{k_1,k_2}, k_2\in J^{k_1,g},$$
while
$$\phi^2(x)-\phi^2 (T^{(\tau_2)}_{\alpha_2}(x)) < -\tfrac{m_2}{2M_1} +20M^3_1 \quad \mbox{for} \ x\in I_{k_1,k_2}, k_2\in J^{k_1,s}$$
where $J^{k_1,s}$ consists of $M_1(M_1-3)$  many elements of
$\{1,\dots ,m_2\}.$
\\
Hence we have a total ``singular mass'' of
\begin{align}\label{p18}
\sum\limits_{k_1\in J_1^s}\ \sum\limits_{k_2 \in J^{k_1,s}}\
\int_{I_{k_1,k_2}} [\phi^2 (x) -\phi^2
(T^{(\tau_2)}_{\alpha_2}(x))]\, dx < -1+\tfrac{3}{M_1} +
\tfrac{c(M_1)}{m_2},
\end{align}
where $c (M_1)$ is a constant depending only on $M_1$.
\item
On the middle interval $I^1_\text{middle} =I_{\frac{M_1+1}{2}}$ we simply let $\tau_2 =\tau_1 =0$.
\end{enumerate}

Let us illustrate graphically an interesting property of this
construction, namely the shape of the quasi-cost function
$\varphi^2+\psi^2\circ T_{\alpha_2}^{(\tau_2)}$.

\medskip
\input{devilish3}
\bigskip

It will sometimes be more convenient to specify to which interval
$I_{l_1,l_2}$ the interval $I_{k_1,k_2}$ is mapped under
$T^{(\tau_2)}_{\alpha_2}$, instead of spelling out the value of
$\tau_2$ on the interval $I_{k_1,k_2}$. Note that by Lemma
\ref{RelPrimeLemma}, for each map associating to $(k_1,k_2)$ a
pair $(l_1,l_2)$, there corresponds precisely one value
$\tau_2|_{I_{k_1,k_2}} : I_{k_1,k_2}\to \{-M_2+1,\dots ,0,\dots
,M_2-1\}$ such that \eqref{p24a} (resp.\ \eqref{p24b}) is
satisfied and $T^{(\tau_2)}_{\alpha_2} (I_{k_1,k_2})=I_{l_1,l_2}$.

\medskip

Let us start with a ``good'' interval $I_{k_1}$, with $k_1 \in
J_1^g$ as in (iii) above, say $k_1\in\{2,\dots
,\tfrac{M_1-1}{2}\}$, for which we have $\tau_1(x)=-1$. Then the
intervals $I_{k_1,2},\dots, I_{k_1,m_2}$ are mapped under
$T^{\tau_1(x)}_{\alpha_2}(x) = T^{-1}_{\alpha_2}(x)$ onto the
intervals $I_{k_1-1,1},\dots ,I_{k_1-1,m_2-1}$. Defining
$\tau_2(x)=\tau_1(x)$ on these intervals we get for $x\in
I_{k_1,k_2}$, where $2\le k_1\le \frac{M-1}{2}, 2\le k_2 \le m_2,$

\begin{align}\label{K11}
1=\phi^1(x)-\phi^1(T^{(\tau_1)}_{\alpha_1}(x))=\phi^2(x)-\phi^2 (T^{(\tau_2)}_{\alpha_2}(x)).
\end{align}

We still have to define the value of $\tau_2(x),$ for $x\in I_{k_1,1}$. The map $T^{(\tau_2)}_{\alpha_2}$ has to map $I_{k_1,1}$ to
the remaining gap $I_{k_1-1,m_2}$, which happens
to be its left neighbour. We do not explicitly calculate the unique number $\tau_2|_{I_{k_1,1}}\in\{-M_2+1,\dots ,M_2-1\}$, satisfying \eqref{p24a}
(resp.~\eqref{p24b}), which does the job, but only use the conclusion of Lemma \ref{L2.2} to find that,
for $x\in I_{k_1,1}$ such that
$T^{(\tau_2)}_{\alpha_2}(x) \in I_{k_1-1,m_2}$,

\begin{align}\label{Nachtrag}
 |1-[\phi^2(x)-\phi^2 (T^{(\tau_2)}_{\alpha_2}(x))]|\le 4M^2_1+1.
 \end{align}
This takes care of the ``good'' intervals $I_{k_1}$, where
$k_1\in\{2,\dots ,\tfrac{M_1-1}{2}\}.$

\vskip 2cm
\input{devilish4a}

 For the ``good'' intervals
$I_{k_1}$, where $k_1\in \{\tfrac{M_1+3}{2}, \dots ,M_1-1\}$ we
have $\tau_1(x)=1$ so that $T^{(\tau_1)}_{\alpha_2}$ maps the
intervals $I_{k_1,1},\dots ,I_{k_1,m_2-1}$ to $I_{k_1+1,2}, \dots
,I_{k_1+1,m_2}$. Again we define $\tau_2(x)=\tau_1(x)=1,$ for $x$
in these intervals so that we obtain the identity \eqref{K11}, for
$\tfrac{M_1+3}{2} \le k_1\le M_1-1$ and $1\le k_2\le m_2-1.$
Finally, $T^{(\tau_2)}_{\alpha_2}$ has to map $I_{k_1,m_2}$ to the
interval $I_{k_1+1,1}$ so that again we derive an estimate as in
\eqref{Nachtrag}.

\input{devilish4b}

This finishes item (iii) i.e.\ the definition of $\tau_2$ on the
``good'' intervals $I_{k_1}.$ Noting that on this set we have
$\tau_1 \ne \tau_2$ only on $M_1-3$ many intervals of length
$\tfrac1{M_2}$ we obtain the estimate \eqref{K8a}.

\bigskip

To show (iv) let us first consider the ``singular'' interval
$I_1$, on which we have $\tau_1(x)=\tfrac{M_1-3}{2}$ and $\phi^1
(T^{(\tau_1)}_{\alpha_1}(x))
=\phi^1(T^{(\tau_1)}_{\alpha_1}(x))-\phi^1(x)= \tfrac{M_1-3}{2}.$
For the subintervals $I_{1,k_2}$ of $I_1$, define the set of good
indices as $J^{1,g}=J^{1,g,l}\cup J^{1,g,r}$ where
$$J^{1,g,l}=\{\tfrac{(M_1-3)(M_1-1)}{2}+1,\dots , \tfrac{m_2-1}{2}\}, \quad J^{1,g,r}=\{\tfrac{m_2+1}{2},\dots ,m_2-\tfrac{(M_1-3)(M_1+1)}{2}\}.$$
Let us start by considering $k_2\in J^{1,g,r}.$ We define
$$\tau_2(x)=\tau_1(x)+\frac{M_1-3}{2} M_1=\frac{(M_1-3)(M_1+1)}{2}, \quad x\in I_{1,k_2},k_2 \in J^{1,g,r}.$$
First note that $T^{(\tau_2)}_{\alpha_2}$ then maps the intervals
$I_{1,k_2}$, for $k_2 \in J^{1,g,r}$, to the intervals
$$
I_{\frac{M_1-1}{2}, \frac{m_2+1}{2}+ \frac{(M_1-3)(M_1+1)}{2}},\
\dots\ , I_{\frac{M_1-1}{2}, m_2}.
$$
Observe that, for $x$ as above, the orbit
$(T^{i}_{\alpha_2}(x))^{\tau_2(x)-1}_{i=0}$ always lies in the
right halfs of the respective intervals $I_{k_1}$.

Let us count how often the orbit $(T^{i}_{\alpha_2}(x))^{\tau_2(x)-1}_{i=0}$ visits $L^2$ and $R^2$ respectively, for $x\in I_{1,k_2}$ and
$k_2\in J^{1,g,r}$. The first $\tau_1(x)=\tfrac{M_1-3}{2}$ elements
of this orbit are all in $L^2$ which yields, similarly as in the induction step $n=1$,
$$\phi^2(T^{(\tau_1)}_{\alpha_2}(x))-\phi^2(x)=\phi^1 (T^{(\tau_1)}_{\alpha_1}(x))-\phi^1(x) =\frac{M_1-3}{2}.$$
But the next $M_1$ many elements of this orbit, namely
$$(T^{i}_{\alpha_2}(x))^{\tau_1(x)+M_1-1}_{i=\tau_1(x)}$$ visit
$R^2$ one time more often than $L^2$ as the unique element of this
orbit which lies in $\IM^1$ belongs to the right half of
$\IM^1$.\\
This phenomenon repeats on the orbit
$(T^{i}_{\alpha_2}(x))^{\tau_1(x)+\tfrac{M_1-3}{2}M_1-1}_{i=0}$
for $\tfrac{M_1-3}{2}$ many times so that
\begin{align}\label{K17}
\begin{split}
\phi^2(x)-\phi^2(T^{(\tau_2)}_{\alpha_2}(x))
& =\phi^2(x)-\phi^2(T^{(\tau_1)}_{\alpha_2}(x))) +\phi^2 (T^{(\tau_1)}_{\alpha_2}(x))-\phi^2 (T^{(\tau_2)}_{\alpha_2}(x)
\\ & = -\frac{M_1-3}{2} + \frac{M_1-3}{2} \\
& =0, \quad  \mbox{for} \ x\in I_{1,k_2} \ \mbox{and} \ k_2\in  J^{1,g,r}.
\end{split}
\end{align}
This takes care of $I_{1,k_2}$ with $k_2\in J^{1,g,r}.$

For $x\in I_{1,k_2}$ with $k_2\in J^{1,g,l}$, the left half of the
``good'' intervals, we define symmetrically
$$
\tau_2(x)=\tau_1(x)-\tfrac{M_1-3}{2}M_1=-\tfrac{(M_1-3)(M_1-1)}{2}.
$$
A similar analysis as above shows that $T^{(\tau_2)}_{\alpha_2}$ maps the intervals $I_{1,k_2},$ where $k_2\in J^{1,g,l},$ to the intervals
$I_{\tfrac{M_1-1}{2},1},\dots ,I_{\tfrac{M_1-1}{2},\tfrac{m_2-1}{2} - \tfrac{(M_1-3)(M_1-1)}{2}}.$
Hence by a symmetric reasoning we again obtain equality \eqref{K17} for $x$ in the intervals $I_{1,k_2},$ and for $k_2\in J^{1,g,r}$ too.

Now we have to deal with the ``singular'' subintervals $I_{1,k_2}$, where $k_2\in J^{1,s}$, and the singular indices are given by
\begin{eqnarray*}
J^{1,s} &=&\{1,\dots ,m_2\} \setminus J^{1,g} \\
&=&\{1,\dots ,\tfrac{(M_1-3)(M_1-1)}{2}\} \cup \{m_2 -
\tfrac{(M_1-3)(M_1+1)}{2}+1,\dots ,m_2\},
\end{eqnarray*}
which consists of $M_1(M_1-3)$ many indices.
\\
The map $T^{(\tau_2)}_{\alpha_2}$ has to map these intervals $I_{1,k_2},$ where $k_2\in J^{1,s}$, to the ``remaining gaps'' $I_{\tfrac{M_1-1}{2},l_2}$ in the interval
$I_{\tfrac{M_1-1}{2}}$, where $l_2\in\{\tfrac{m_2+1}{2} - \tfrac{(M_1-3)(M_1-1)}{2},\dots ,\tfrac{m_2+1}{2}+\tfrac{(M_1-3)(M_1+1)}{2}-1\}.$
Note that the corresponding intervals $I_{\tfrac{M_1-1}{2},l_2}$ are -- roughly speaking -- in the middle of the interval $I_{\tfrac{M_1-1}{2}}$, while the
intervals $I_{1,k_2}$, with
$k_2\in J^{1,s}$, are at the boundary of $I_1$.
\\
To define $\tau_2$ on $I_{1,k_2}$, for $k_2\in J^{1,s}$, choose
any function $\tau_2$ taking values in $\{-M_2+1,\dots ,M_2-1\}$,
satisfying \eqref{p24a} (resp.~\eqref{p24b}) as above, which
induces a bijection between the intervals $(I_{1,k_2})_{k_2\in
J^{1,s}}$ and the intervals $I_{\tfrac{M_1-1}{2},l_2}$ considered
above.

\input{devilish5}
\par\medskip

For each such $\tau_2$ we obtain, for $x\in I_{1,k_2},k_2\in
J^{1,s}$, from Lemma \ref{L2.2}
\begin{equation}\label{K19}
\begin{split}
\phi^2(x)-\phi^2 (T^{(\tau_2)}_{\alpha_2}(x))&\le
-\frac{m_2}{2M_1} +10M^3_1 +2\frac{(M_1-3)(M_1-1)}{2}4M^2_1
\\ &\le -\frac{m_2}{2M_1} +20M^4_1. \qquad \hspace{4.2cm}
\end{split}
\end{equation}
Indeed, the leading term $\tfrac{-m_2}{2M_1}$ and the first error
term $10M^3_1$ in the first line above  come from Lemma
\ref{L2.2}-(ii) when comparing the difference of the value of
$\phi^2$ on the interval $I_{1,1}$ to that of
$I_{\tfrac{M_1-1}{2}, \tfrac{m_2+1}{2}}$. For the difference of
the value of $\phi^2$ on $I_{1,k_2}$ and
$I_{\tfrac{M_1-1}{2},l_2}$, for arbitrary $k_2\in J^{1,s}$ and
$l_2\in\{ \tfrac{m_2+1}{2}-\tfrac{(M_1-3)(M_1-1)}{2},\dots
,\tfrac{m_2+1}{2}+\tfrac{(M_1-3)(M_1-1)}{2}\}$ we apply for both
cases at most $\tfrac{(M_1-3)(M_1+1)}{2}$ times estimate (i) of
Lemma \ref{L2.2} which gives  \eqref{K19}.
\\
In particular, for $m_2>40M^5_1$, which of course we shall assume,
we have that
$$\phi^2(x)-\phi^2 (T^{(\tau_2)}_{\alpha_2}(x)) \le 0, \ \qquad \mbox{for} \ x\in I_{1,k_2}, k_2\in J^{1,s}.$$
There are $M_1(M_1-3) =M^2_1-3M_1$ many intervals $I_{1,k_2}$ with
$k_2\in J^{1,s}$ each of length $1/M_2.$ Hence we may estimate the
``singular mass'' on the interval $I_1$ by
\begin{align}
\label{K21}
\begin{split}
\sum\limits_{k_2\in J^{1,s}} \int_{I_{1,k_2}}[ \phi^2(x)-\phi^2
(T^{(\tau_2)}_{\alpha_2}(x))]\, dx
 & \le \big( - \frac{m_2}{2M_1} +20M^4_1\big) (M^2_1 -3M_1)\frac1{M_2}
\\ & \le -\frac12 +\frac{3}{2M_1} +  \frac{c(M_1)}{2m_2}.
\end{split}
\end{align}
where $c(M_1)$ is a constant depending on $M_1$ only.\footnote{We shall find it convenient in the sequel to write $c(M_1,M_2,\ldots ,M_i)$ for constants
depending only on the choice of the numbers $M_1,M_2,\ldots ,M_i$. The concrete numerical value of this expression may change, i.e.~become bigger, from one
line of reasoning to the next one, but at every stage it will be clear that an explicit bound for the respective meaning of the constant $c(M_1,M_2,\ldots ,M_i)$
could be given, at least in principle. In fact, we shall always have that the constants $c(M_1,M_2,\ldots ,M_i)$ used in the sequel are dominated by a
polynomial in the variables $M_1,M_2,\ldots ,M_i.$}

We still have another ``singular'' interval at the present
induction step $n=2$, namely $I_{M_1}$. The analysis for this case
is symmetric to the analysis of $I_1$ and -- after properly
defining $\tau_2$ on this interval $I_{M_1}$ -- we arrive at the
same estimate \eqref{K21}. In total, the thus obtain \eqref{p18}
by doubling the right hand side of \eqref{K21}, showing that the
``singular mass'' essentially equals $-1.$

Finally define the sets $J^g_2$ (resp.~$J^s_2$) of ``good'' (resp.~``singular'') indices at level 2 as
\begin{align*}
\begin{split}
J^g_2&=\{(k_1,k_2):(k_1\in J^g_1 \ \mbox{and}  \ 1\le k_2\le m_2), \  \mbox{or}
 \ (k_1\in J^s_1 \ \mbox{and} \ k_2\in J^{k_1,g})\},  \\
 J^s_2&=\{(k_1,k_2): k_1\in J^s_1 \mbox{ and} \ k_2\in J^{k_1,2}\}.
\end{split}
\end{align*}

This finishes the inductive step for $n=2$.

\vskip6pt
\noindent{\bf General inductive step.}\
    Suppose that the prime
numbers $m_1, \dots ,m_{n-1}$ have been defined. We use the
notation $\alpha_{n-1}=\tfrac{1}{M_1}+\dots +\tfrac{1}{M_{n-1}}$,
where $M_{n-1}=m_1\cdot m_2\cdot \ldots \cdot m_{n-1}.$
\\
For a prime $m_n$ satisfying the condition of Lemma
\ref{RelPrimeLemma}, and to be specified below, let
$M_n=m_1\cdot\ldots \cdot m_n$ and
$$
    L^n=\left[0,\frac12 -\frac{1}{2M_n}\right),\
    R^n=\left[\frac12 +\frac{1}{2M_n}, 1\right),\
    I^n_\text{middle} =\left[\frac12-\frac{1}{2M_n}, \frac12 +\frac{1}{2M_n}\right).
$$
For $1\le k_1\le m_1,\dots ,1\le k_n\le m_n,$ let
$$I_{k_1,\dots ,k_n} = [\tfrac{k_1-1}{M_1}+\tfrac{k_2-1}{M_2}+\dots +\tfrac{k_n-1}{M_n}\ ,
\tfrac{k_1-1}{M_1}+\tfrac{k_2-1}{M_2}+\dots +\tfrac{k_n}{M_n}).
$$
For $x\in I_{1,\dots ,1}$ and $j\in\{0,\dots ,M_n\}$ we define,
similarly as in \eqref{21}, $\phi^n(x)=0$ and
\begin{align}\label{istep}
\begin{split}
\phi^n(T^j_{\alpha_n}(x))=\quad &\#\{i\in \{0,\ldots,j-1\}:
T_{\alpha_2}^i(x) \in L^n \} \hspace{0.15cm}
\\   -  &\#\{i\in \{0,\ldots,j-1\}: T_{\alpha_2}^i(x) \in R^n \},
\end{split}
\end{align}
where $\alpha_n=\alpha_{n-1}+\tfrac{1}{M_n}$ and $M_n=M_{n-1}
m_n.$ We also let $\psi^n(x)=1-\phi^n(x),$ for $x\in[0,1)$.

\begin{lemma}[Oscillations of $\phi^n$]\label{res-1}
For given $M_1,\dots ,M_{n-1}$ there is a constant $c(M_1,\dots
,M_{n-1})$ depending only on $M_1,\dots ,M_{n-1}$, such that for
all $m_n$ as above we have
\begin{enumerate}
    \item $|\phi^n(x)-\phi^n(x\oplus \tfrac{1}{M_n})|\le c(M_1,\dots ,M_{n-1}),$

    \item for each $1\le k'_1, k''_1 \le M_1,\dots ,1\le k'_{n-1},k''_{n-1}\le m_{n-1},$
$$
    {\phi^n}_{| I_{k'_1,\dots ,k'_{n-1},(m_n+1)/2}}
    -{\phi^n}_{|_{I_{k''_1,\dots ,k''_{n-1},1}}} \geq \tfrac{m_n}{2M_{n-1}}-c(M_1,\dots ,M_{n-1}),
$$
    \item for each $1\le k'_1, k''_1 \le M_1,\dots ,1\le
k'_{n-1},k''_{n-1}\le m_{n-1},$ and $1\le k'_n,k''_n\le m_n,$ with
$\min \{k'_n,m_n-k'_n\}< M_{n-1}$ and $\min\{k''_n,m_n-k''_n\}
<M_{n-1}$ we have
\begin{align}\label{I3alpha}
\left|{\phi^n}_{|_{I_{k'_1,\dots
,k'_{n-1},k'_n}}}-{\phi^n}_{|_{I_{k''_1,\dots
,k''_{n-1},k''_n}}}\right|\le c(M_1,\dots ,M_{n-1}).
\end{align}
\end{enumerate}
\end{lemma}

\begin{proof}{}
We may and do assume that $m_n\geq 5M_{n-1}.$

\noindent$\bullet$\textrm{ Proof of (i).}\
    We have $T_{\alpha_n}(x)=T_{\alpha_{n-1}}(T_{1/M_n}(x))$
so that
\begin{align}\label{I3a}
T_{\alpha_n}^{M_{n-1}}(x)=x\oplus \tfrac{M_{n-1}}{M_n}=x\oplus \tfrac{1}{m_n},
\end{align}
in perfect analogy to \eqref{K2}. As regards the analogue to
\eqref{K3} things now are somewhat more complicated. First note
that there is a unique number $1\le q_{n-1}\le M_{n-1}-1$ such
that
\begin{align}\label{I3b}
T^{q_{n-1}}_{\alpha_{n-1}}(x)=x\ominus \frac{1}{M_{n-1}}, \qquad x\in[0,1).
\end{align}
Indeed, by Lemma \ref{RelPrimeLemma}, when $q_{n-1}$ runs through
$\{1,\dots ,M_{n-1}-1\}$, the left hand side assumes the values
$x\ominus\tfrac{l_{n-1}}{M_{n-1}}$, where $l_{n-1}$ also runs
through $\{1,\dots ,M_{n-1}-1\}$.

\medskip
\noindent{\it Claim: Letting
$r_n=\lfloor\tfrac{m_n}{M_{n-1}}\rfloor$, the integer part of
$\tfrac{m_n}{M_{n-1}}$, and taking $q_{n-1}$ as in \eqref{I3b}, we
have
$$T^{r_n M_{n-1}+q_{n-1}}_{\alpha_n}(x)=x\oplus \frac{d_{n-1}}{M_n},$$
where $|d_{n-1}|<M_{n-1}.$}

\medskip\noindent
    Indeed, write $m_n$ as $m_n=r_nM_{n-1}+e_{n-1},$ for some $1\le e_{n-1} \le M_{n-1}$ to obtain
\begin{align*}
T^{r_n M_{n-1}+q_{n-1}}_{\alpha_n}(x) & = (T^{M_{n-1}}_{\alpha_n})^{r_n} \circ T^{q_{n-1}}_{\alpha_{n-1}} \circ T^{q_{n-1}}_{\tfrac{1}{M_n}} (x)
\\ & =x\oplus r_n\frac{M_{n-1}}{M_n}\ominus \frac{1}{M_{n-1}}\oplus \frac{q_{n-1}}{M_n}
\\ & =x\oplus \frac{m_n}{M_n} \ominus \frac{e_{n-1}}{M_n}\ominus\frac{1}{M_{n-1}}\oplus \frac{q_{n-1}}{M_n}
\\ & =x\oplus \frac{q_{n-1} - e_{n-1}}{M_n} =: x\oplus \frac{d_{n-1}}{M_n}
\end{align*}
which proves the claim.
\\
Define $s^{(1)}_{n-1}=q_{n-1}$ if $d_{n-1}=q_{n-1}-e_{n-1}
>0$ and $s^{(1)}_{n-1}=q_{n-1}+M_{n-1}$ otherwise, to obtain by
\eqref{I3a} and \eqref{I3b} that
$$T_{\alpha_n}^{r_nM_{n-1}+s^{(1)}_{n-1}}(x)=x\oplus \tfrac{l^{(1)}_{n-1}}{M_n},$$
for some $l^{(1)}_{n-1} \in \{1,\dots ,M_{n-1}\}.$ We also deduce
from \eqref{I3a} that $l^{(1)}_{n-1}$ must actually be in
$\{1,\dots ,M_{n-1}-1\}.$
\\
Repeat the above argument to find $s^{(2)}_{n-1}$ with $-2M_{n-1} < s^{(2)}_{n-1} <2M_{n-1}$ such that
$$T_{\alpha_n}^{2r_n M_{n-1}+s^{(2)}_{n-1}}(x)=x\oplus \tfrac{l_{n-1}^{(2)}}{M_n},$$
for some $l_{n-1}^{(2)}\in\{1,\dots ,M_{n-1}-1\}$. Continuing in the same way, we find numbers $s^{(j)}_{n-1}$, for $j=1,2,\dots ,M_{n-1}-1$ verifying
$-jM_{n-1}<s^{(j)}_{n-1} <j M_{n-1}$ such that
\begin{align}\label{I3bb}
T_{\alpha_n}^{jr_nM_{n-1}+s^{(j)}_{n-1}}(x)=x\oplus \tfrac{l^{(j)}_{n-1}}{M_n},
\end{align}
for some $l^{(j)}_{n-1}\in\{1,\dots ,M_{n-1}-1\}.$ Note that,
under the assumption $m_n \gg M_{n-1}$ so that $r_n \gg M_{n-1},$
 the elements in \eqref{I3bb} are all different. Therefore
$(l^{(j)}_{n-1})^{M_{n-1}-1}_{j=1}$ runs through all elements of
$\{1,\dots ,M_{n-1}-1\}$ when $j$ runs through $\{1,\dots
,M_{n-1}-1\}$; in particular there must be some $j_0$ such that
$$T_{\alpha_n}^{j_0 r_nM_{n-1}+s^{(j_0)}_{n-1}}(x)=x\oplus \tfrac{1}{M_n},$$
in analogy to \eqref{p16}.

Now observe that there is a constant $c(M_1,\dots , M_{n-1})$,
depending only on $M_1,\dots ,M_{n-1},$ such that, for $x\in
[0,1)$, the difference of the number of visits of the orbit
$(T^i_{\alpha_n}(x))^{r_n M_{n-1}+q_{n-1}}_{i=0}$ to $L^n$ and
$R^n$ is bounded in absolute value by the constant $c(M_1,\dots ,
M_{n-1})$. The argument is analogous to the corresponding one in
the proof of the claim which is part of the proof of  Lemma
\ref{L2.2}-(i), and therefore skipped.
\\
The numbers $j_0$ as well as $s^{(j_0)}_{n-1}$ are bounded in
absolute value by $M_{n-1}^2$ so that the difference of the visits
of the orbits $(T_{\alpha_n}^i (x))_{i=0}^{j_0
r_nM_{n-1}+s^{(j_0)}_{n-1}}$ to $L^n$ and $R^n$ are bounded in
absolute value by some constant $c(M_1,\dots ,M_{n-1})$. This
finishes the proof of assertion (i).

\medskip\noindent$\bullet$\textrm{ Proof of (ii).}\
    Suppose first, as in the proof of Lemma \ref{L2.2}-(ii), that
$(k'_1,\dots ,k'_{n-1})=(k''_1,\dots ,k''_{n-1})=:(k_1,\dots
,k_{n-1}).$ For $x\in I_{k_1,\dots ,k_{n-1},1}$ we have that each
of the orbits $(T_{\alpha_n}^{jM_{n-1}+i}(x))^{M_{n-1}-1}_{i=0}$,
for $j=0,\dots ,\lfloor\tfrac{m_n}{2M_{n-1}}\rfloor-1$ visits
$L^n$ one time more often than $R^n$. Hence
$$
\phi^n(T_{\alpha_n}^{\lfloor\tfrac{m_n}{2M_{n-1}}\rfloor
M_{n-1}}(x))-\phi^n(x)=\lfloor\tfrac{m_n}{2M_{n-1}}\rfloor>
\tfrac{m_n}{2M_{n-1}}-1.
$$
Noting that
$$
T_{\alpha_n}^{\lfloor\tfrac{m_n}{2M_{n-1}}\rfloor M_{n-1}}(x)=x\oplus \lfloor\tfrac{m_n}{2M_{n-1}}\rfloor \tfrac{M_{n-1}}{M_n}
$$
and
$$
\tfrac{m_n+1}{2M_n} - \lfloor\tfrac{m_n}{2M_{n-1}}\rfloor\tfrac{M_{n-1}}{M_n} \le \tfrac{M_{n-1}}{M_n},
$$
we obtain (ii) by using assertion (i), and possibly passing to a bigger constant $c(M_1,\dots ,M_{n-1})$.
\\
Finally the passage to general $(k'_1,\dots ,k'_{n-1})$ and
$(k''_1,\dots ,k''_{n-1})$ is done again, similarly as in the
proof of Lemma \ref{L2.2}, by repeated application of (i) and by
passing once more to a bigger constant $c(M_1,\dots ,M_{n-1})$.

\medskip\noindent$\bullet$\textrm{ Proof of (iii).}\
    Fix
$1\le k'_1,k''_1\le M_1,\dots ,1\le k'_{n-1}, k''_{n-1}\le
m_{n-1}$ and $1\le k'_n,k''_n\le m_n$ as above. Suppose, e.g.,
$k'_n\le M_{n-1}$ and $m_n-k''_n\le M_{n-1}$, the other three
cases being similar. Denote by $(k'''_1,\dots ,k'''_{n-1})$ the
index so that $I_{k'''_1,\dots ,k'''_{n-1}}= I_{k''_1,\dots
,k''_{n-1},k''_n}\oplus \tfrac{1}{M_{n-1}}$,
i.e.~$I_{k'''_1,\ldots ,k'''_{n-1}}$ is the right neighbour of
$I_{k''_1,\dots ,k''_{n-1}}.$ Now find $0\le q_{n-1} <M_{n-1}$
such that $T^{q_{n-1}}_{\alpha_{n-1}}$ maps $I_{k'_1,\dots
,k'_{n-1}}$ onto $I_{k'''_1,\dots ,k'''_{n-1}}$. Hence
$T^{q_{n-1}}_{\alpha_{n}}$ maps $I_{k'_1,\dots ,k'_{n-1},k'_n}$
onto $I_{k'''_1,\dots ,k'''_{n-1},k'_n+q_{n-1}}.$
\\
Finally note that the distance from the latter interval to $I_{k''_1,\dots ,k''_{n-1},k''_n}$ is bounded by $(2M_{n-1}+M_{n-1})\tfrac{1}{M_n}.$ Hence we
obtain \eqref{I3alpha} by applying $2M_{n-1}+M_{n-1}$ times assertion (i) and using $0\le q_{n-1}<M_{n-1}.$
\end{proof}

After this preparation we are ready for the inductive step from
$n-1$ to $n.$ Suppose that the following inductive hypotheses are
satisfied, for $1\le l\le n-1,$ functions $\tau_l :[0,1)\to
\{-M_l+1,\dots ,M_l-1\}$ and index sets $J^g_l,J^s_l$ contained in
$\{(k_1,\ldots ,k_l):1\le k_1\le m_1,\ldots , 1\le k_l \le m_l\}.$

\begin{enumerate}
\item
The measure preserving bijection $T_{\alpha_{n-1}}^{(\tau_{n-1})}:[0,1)\to [0,1)$ maps the intervals $I_{k_1,\dots ,k_l}$, for $1\le l <n-1$, and
$1\le k_1\le m_1, \ldots ,
1\le k_l \le m_l$, onto the intervals $T^{(\tau_l)}_{\alpha_l} (I_{k_1,\dots ,k_l}).$ It induces a permutation of the intervals
$I_{k_1,\dots ,k_{n-1}}$, where $1\le k_1\le m_1 ,\ldots ,1\le k_{n-1} \le m_{n-1}.$

    \item
When $\tau_{n-1}(x) >0$, we have
\begin{align}\label{B1}
T^i_{\alpha_{n-1}}(x) \notin I^{n-1}_{\text{middle}}, \qquad
i=0,\ldots ,\tau_{n-1}(x),
\end{align}
and, when $\tau_{n-1}(x)<0$, we have
\begin{align}\label{B1a}
T^i_{\alpha_{n-1}}(x)\notin I^{n-1}_{\text{middle}}, \qquad
i=\tau_{n-1}(x),\ldots ,0.
\end{align}

    \item There is a set of ``good'' indices
$J^g_{n-1} \subseteq \{1\le k_1 \le m_1,\ldots ,1\le k_{n-1}\le
m_{n-1}\}.$ For $(k_1,\ldots ,k_{n-2})\in J^g_{n-2}$ we have that
$(k_1,\ldots ,k_{n-2}, k_{n-1})\in J^g_{n-1}$ as well as
\begin{align}\label{C1}
\mu[I_{k_1,\ldots ,k_{n-2}} \cap \{\tau_{n-2}\neq \tau_{n-1}\}]\leq \tfrac{M_{n-2}}{m_{n-1}} \mu [I_{k_1,\ldots ,k_{n-2}}],
\end{align}
and
\begin{align}\label{I4}
\begin{split}
    \sum_{(k_1,\dots ,k_{n-2})\in J^g_{n-2}} & \int_{I_{k_1,\dots ,k_{n-2}}}
    \Big|[\phi^{n-2}(x)-\phi^{n-2}(T^{(\tau_{n-2})}_{\alpha_{n-2}}(x))]\\
   & \hskip 3cm -[\phi^{n-1}(x)-\phi^{n-1} (T^{(\tau_{n-1})}_{\alpha_{n-1}}(x))] \Big|
    \, dx\\
    & \hskip 1cm\le \tfrac{c(M_1,\dots ,M_{n-2})}{m_{n-1}}.
\end{split}
\end{align}

    \item
There is a set of ``singular'' indices $J_{n-1}^s\subseteq \{(k_1,\dots ,k_{n-1}):1\le k_1\le m_1,\dots ,1\le k_{n-1}\le m_{n-1}\}$, disjoint from $J_{n-1}^g$, such
that $J^s_{n-1}$ consists of less than $2M^2_{n-1}$ many elements and such that
\begin{align}\label{I5a}
\begin{split}
\phi^{n-1}(x)-\phi^{n-1}(T_{\alpha_{n-1}}^{(\tau_{n-1})}(x)) \ \le \ 0, \quad \mbox{for }&  x\in I_{k_1,\dots ,k_{n-1}}\\
&\mbox{and} \ (k_1,\dots ,k_{n-1}) \in J^s_{n-1},
\end{split}
\end{align}
and
\begin{align}\label{I5b}
\begin{split}
\sum\limits_{(k_1,\dots ,k_{n-1})\in J^s_{n-1}}\ \ \ \int\limits_{I_{k_1,\dots ,k_{n-1}}}
    [\phi^{n-1}&(x)-\phi^{n-1}(T_{\alpha_{n-1}}^{(\tau_{n-1})}(x))]\, dx\\
    &\le -1+\tfrac{3}{m_1} +\tfrac{c(M_1)}{m_2} + \dots +\tfrac{c(M_1,\dots ,M_{n-2})}{m_{n-1}},
\end{split}
\end{align}
where $c(\cdot)$ are constants depending only on $(\cdot)$.
\item
On the middle interval $I^1_{\text{middle}} =I^1_{\tfrac{M_1+1}{2}}$ we have $\tau_1=\tau_2=\dots =\tau_{n-1}=0$ and $I^1_\text{middle}$ together with the
intervals $(I_{k_1,\dots ,k_{n-1}})_{(k_1,\dots ,k_{n-1})\in J^g_{n-1} \cup J^s_{n-1}}$ form a partition of $[0,1)$.
\end{enumerate}
We have to define $\tau_n$ as well as $J^g_n$ and $J^s_n$ so that the above list is satisfied with $n-1$ replaced by $n$.

Let us illustrate graphically some  features of this construction.
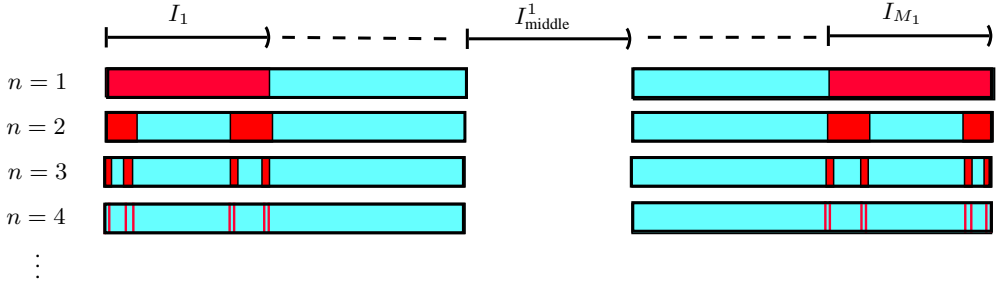
Namely, the fractal structure of the singular set and the
resulting quasi-cost.
\par\medskip
\input{devilish6}
\medskip
\input{devilish7}
\medskip

We start with a ``good'' interval $I_{k_1,\dots ,k_{n-1}}$,
i.e.~$(k_1,\dots ,k_{n-1})\in J_{n-1}^g$ and simply write $\tau$
for $\tau_{n-1}|_{I_{k_1,\ldots , k_{n-1}}}.$ If $\tau >0$, define
$J^{k_1,\ldots ,k_{n-1},c}$, where $c$ stands for ``change'', as
$\{m_n -\tau+1,\linebreak \ldots ,m_n\}.$ This set consists of
those indices $k_n$ such that the interval $I_{k_1,\dots ,k_n}$ is
not mapped into $T^{(\tau_{n-1})}_{\alpha_{n-1}} (I_{k_1,\dots
,k_{n-1}})$  under $T^{(\tau_{n-1})}_{\alpha_n}$. If $\tau <0$, we
define $J^{k_1,\dots ,k_{n-1},c}$ as $\{1,\ldots ,|\tau|\}.$ The
complement $\{1,\ldots ,m_n\}\backslash J^{k_1,\dots ,k_{n-1},c}$
is denoted by $J^{k_1,\dots ,k_{n-1},u}$, where $u$ stands for
``unchanged''.

Define $\tau_n:=\tau_{n-1}=\tau$ on the intervals $I_{k_1,\dots
,k_{n-1},k_n}$, for $k_n\in J^{k_1,\dots ,k_{n-1},u}.$ For $x$ in
one of those intervals we have by \eqref{B1}, \eqref{B1a} and
\eqref{istep} that
$$
\phi^n(x)-\phi^n(T^{(\tau_n)}_{\alpha_n}(x))=\phi^{n-1}(x)-\phi^{n-1}(T^{(\tau_{n-1})}_{\alpha_{n-1}}(x)),
$$
which yields \eqref{C1} with $n-1$ replaced by $n$.

On the remaining intervals $I_{k_1,\dots ,k_n}$ with $k_n\in J^{k_1,\dots ,k_{n-1},c}$ we define $\tau_n$ such that it takes constant values in $\{-M_n+1,\ldots ,
M_n-1\}$ on each of these intervals, such that \eqref{B1} (resp.~\eqref{B1a}) is satisfied, and such that these intervals $I_{k_1,\dots ,k_n}$ are mapped onto the
``remaining gaps'' in $T^{(\tau_{n-1})}_{\alpha_{n-1}}(I_{k_1,\dots ,k_{n-1}}).$

The crucial observation is that the intervals $I_{k_1,\dots
,k_{n-1},k_n}$ where we have $\tau_n\ne \tau_{n-1}$, i.e.~where
$k_n\in J^{k_1, \ldots , k_{n-1},c},$ are all on the ``boundary''
of $I_{k_1,\dots ,k_{n-1}}$: they are the $|\tau|$ many intervals
on the left or right end of $I_{k_1,\dots ,k_{n-1}},$ depending on
the sign of $\tau.$ Similarly, the ``remaining gaps'' in
$T^{(\tau_{n-1})}_{\alpha_{n-1}} (I_{k_1,\dots ,k_{n-1}})$ are the
$|\tau|$ many intervals on the opposite end of
$T^{(\tau_{n-1})}_{\alpha_{n-1}} (I_{k_1,\dots ,k_{n-1}}).$ Hence
we may apply assertion (iii) of Lemma \ref{res-1} to conclude that
$$
|\phi^n(x)-\phi^n (T^{(\tau_{n})}_{\alpha_n} (x))|\le c(M_1,\dots ,M_{n-1}),
$$
for those $x\in I_{k_1,\dots ,k_{n-1}}$ where $\tau_n(x)\ne \tau_{n-1}(x).$
Summing over all ``good intervals'' $I_{k_1,\dots ,k_{n-1}}$, where $(k_1,\ldots ,k_{n-1})\in J^g_{n-1},$ we conclude that the contribution
to \eqref{I4}, with $n-1$ replaced by $n$, is controlled by
the following factors:
$M_{n-1}$, which is a bound for the number of elements in $J^g_{n-1}$, times $M_{n-1}$, which is a bound for $|\tau|$, times $\tfrac{1}{M_n}$, which is the
length of the
intervals $I_{k_1,\dots ,k_{n}}$, times the above found constant $c(M_1,\ldots ,M_{n-1}).$ In total, this implies the estimate \eqref{I4}, with $n-1$
replaced by $n$.

\medskip

We now turn to item (iv), i.e.~to the ``singular'' indices: fix
$k_1,\dots ,k_{n-1}\in J^s_{n-1}$ and let $\Delta\phi$ denote the
constant
$$
\Delta\phi :=
\phi^{n-1}(T^{(\tau_{n-1})}_{\alpha_{n-1}}(x))-\phi^{n-1}(x),
\quad x\in I_{k_1,\dots ,k_{n-1}},
$$
and again $\tau$ the constant ${\tau_{n-1}}_{|{I_{k_1,\dots
,k_{n-1}}}},$ so that $0 \le \Delta\phi\le |\tau|<M_{n-1}.$
\\
Similarly as for the case $n=2$ define
$$
J^{k_1,\dots ,k_{n-1},g,l}=\{k^l_n,k^l_n+1\dots
,\tfrac{m_n-1}{2}\}, \ \ J^{k_1,\dots
,k_{n-1},g,r}=\{\tfrac{m_n+1}{2},\dots , k^r_n\}.
$$
Here $k^r_n$ is the largest number such that, for the orbit
$(T^i_{\alpha_n}(x))_{i=\tau}^{\tau +\Delta\phi M_{n-1}-1}$ and
for $x\in I_{k_1,\dots , k_{n-1},k^r_n},$ all its members lie in
the right half of the respective intervals $I_{k'_1,\dots
,k'_{n-1}}.$ In fact, we get as in the step $n=2$ that $k^r_n=
m_n-(\tau +\Delta\phi M_{n-1}).$
\\
Similarly $k^l_n$ is the smallest number such that, for the orbit
$(T^i_{\alpha_n}(x))_{i=\tau}^{\tau -\Delta\phi M_{n-1}+1}$ and
for $x\in I_{k_1,\dots , k_{n-1},k^l_n},$ all its members are in
the left half of the respective intervals $I_{k'_1,\dots
,k'_{n-1}}.$ We get $k^l_n=\tau -\Delta\phi M_{n-1}+1.$
\\
Now we define $\tau_n$ as
$$\tau_n(x) =\tau+\Delta\phi M_{n-1}, \quad \mbox{for} \ x\in I_{k_1,\dots ,k_{n-1},k_n}, k_n\in J^{k_1,\dots ,k_{n-1},g,r},$$
and
$$
\tau_n (x)=\tau -\Delta\phi M_{n-1}, \quad \mbox{for} \ x\in I_{k_1,\dots ,k_{n-1},k_n},k_n\in J^{k_1,\dots ,k_{n-1},g,l}.
$$
Similarly as in \eqref{K17} at step $n=2$, we get for $k_n\in
J^{k_1,\dots ,k_{n-1},g}:=J^{k_1,\dots ,k_{n-1},g,l} \cup
J^{k_1,\dots ,k_{n-1},g,r},$ and $x\in I_{k_1,\dots ,k_{n-1},k_n}$
that
\begin{align*}
\begin{split}
\phi^n(x)-\phi^n & (T^{(\tau_n)}_{\alpha_n}(x))\\
    & =[\phi^n(x)-\phi^n(T^{(\tau_{n-1})}_{\alpha_n}(x))]
        +[\phi^n(T^{(\tau_{n-1})}_{\alpha_n}(x))-\phi^n(T^{(\tau_n)}_{\alpha_n}(x))] \\
    & =[\phi^{n-1}(x)-\phi^{n-1}( T^{(\tau_{n-1})}_{\alpha_{n-1}}(x))]
        +[\phi^n(T^{(\tau_{n-1})}_{\alpha_n}(x))-\phi^n(T^{(\tau_n)}_{\alpha_n}(x))] \\
& =-\Delta\phi+\Delta\phi=0.
\end{split}
\end{align*}
We still have to deal with the ``singular'' indices
$$
J^{k_1,\dots ,k_{n-1},s}:=\{1,\dots ,m_n\} \setminus J^{k_1,\dots
,k_{n-1},g}=\{1,\dots ,k^l_n-1\} \cup \{k^r_n+1,\dots ,m_n\},
$$
which consists of $2\Delta\phi M_{n-1}$ many indices. This number
is bounded by $2M^2_{n-1}$ as $\Delta\phi\le |\tau|<M_{n-1}.$
These intervals have to be mapped onto the ``remaining gaps'' in
the interval $T^{(\tau_{n-1})}_{\alpha_{n-1}}(I_{k_1,\dots
,k_{n-1}}).$ Make the crucial observation that, while the
intervals $I_{k_1,\dots ,k_{n-1},k_n},$ for $k_n\in J^{k_1,\dots
,k_{n-1},s}$, are at the boundary of $I_{k_1,\dots ,k_{n-1}}$, the
``remaining gaps'' are in the middle of the interval
$T^{(\tau_{n-1})}_{\alpha_{n-1}}(I_{k_1,\dots ,k_{n-1}}).$ This
fact is analogous to the situation for $n=1$ and $n=2.$

Now define $\tau_n$ on the intervals $I_{k_1,\dots ,k_{n-1},k_n}$
for $k_n\in J^{k_1,\dots ,k_{n-1},s}$, in such a way that
$T^{(\tau_n)}_{\alpha_n}$ maps these intervals onto the
``remaining gaps'' in
$T^{(\tau_{n-1})}_{\alpha_{n-1}}(I_{k_1,\dots ,k_{n-1}})$ and such
that $\tau_n$ is constant on each of these intervals, takes values
in $\{-M_n+1,\ldots ,M_n-1\}$ and such that \eqref{B1}
(resp.~\eqref{B1a}) is satisfied with $n-1$ replaced by $n$.
Applying Lemma \ref{res-1}, assertion (ii) as well as
$2(M_{n-1}+1)|\tau|$ many times assertion (i) we obtain, for $x\in
I_{k_1,\dots ,k_{n-1},k_n}$ and $k_n\in J^{k_1,\dots ,k_{n-1},s},$
$$
\phi^n(x)-\phi^n (T^{(\tau_n)}_{\alpha_n}(x))\le
-\tfrac{m_n}{2M_{n-1}}+c(M_1,\dots ,M_{n-1}).
$$
Assuming that $m_n$ is sufficiently large as compared to $M_{n-1}$ we have that the right hand side is negative.
\\
Keeping in mind that there are $2\Delta\phi M_{n-1}$ many indices
in $J^{k_1,\ldots ,k_{n-1},s}$, we may estimate the ``singular
mass'' on the interval $I_{k_1,\dots ,k_{n-1}}$ by
\begin{align}\label{I12}
\begin{split}
\sum_{k_n\in J^{k_1,\dots ,k_{n-1},s}} \int_{I_{k_1,\dots
,k_{n-1},k_n}} & [\phi^n(x)-\phi^n(T^{(\tau_n)}_{\alpha_n}(x))]\,
dx
\\ & \le 2\Delta\phi M_{n-1}[-\tfrac{m_n}{2M_{n-1}}+c(M_1,\dots ,M_{n-1})] \ \ \tfrac{1}{M_n}
\\ & =-\tfrac{\Delta\phi}{M_{n-1}} \ \ [1-\tfrac{c(M_1,\dots ,M_{n-1})}{m_n}].
\end{split}
\end{align}
We have by the inductive hypothesis that
$$
    \sum_{k_1,\dots ,k_{n-1}\in J^s_{n-1}} \int_{I_{k_1,\dots,k_{n-1}}}
    [\phi^{n-1}(x)-\phi^{n-1}(T^{(\tau_{n-1})}_{\alpha_n}(x))]\, dx
$$
$$\le -1+\tfrac{3}{m_1}+\tfrac{c(M_1)}{m_2}+\dots +\tfrac{c(M_1,\dots ,M_{n-2})}{m_{n-1}},$$
or, writing now $\Delta\phi_{k_1,\dots ,k_{n-1}}$ for the above
value of $\Delta\phi$ on the interval $I_{k_1,\dots ,k_{n-1}}$,
$$
\tfrac{1}{M_{n-1}} \sum\limits_{k_1,\dots ,k_{n-1}\in J^s_{n-1}} \Delta\phi_{k_1,\dots ,k_{n-1}}\le -1+\tfrac{3}{m_1}+\tfrac{c(M_1)}{m_2}+\dots +
\tfrac{c(M_1,\dots ,M_{n-2})}{m_{n-1}}.
$$
Letting $J^s_n:=\bigcup\limits_{k_1,\dots ,k_{n-1}\in
J^s_{n-1}}\{(k_1,\ldots ,k_{n-1},k_n):k_n\in J^{k_1,\dots
,k_{n-1},s}\}$ we obtain from \eqref{I12}
\begin{align*}
\begin{split}
\sum\limits_{k_1,\dots ,k_n\in J^s_n} \int_{I_{k_1,\dots ,k_n}}&[\phi^n(x)-\phi^n(T^{(\tau_n)}_{\alpha_n}(x))]\, dx \\
&\le (-1+\tfrac{3}{m-1}+\ldots +\tfrac{c(M_1,\ldots ,M_{n-2})}{m_{n-1}})(1-\tfrac{c(M_1,\ldots ,M_{n-1})}{m_n}) \\
&= -1+\tfrac{3}{m_1} +\dots +\tfrac{c(M_1,\dots
,M_{n-2})}{m_{n-1}}+\tfrac{c(M_1,\dots ,M_{n-1})}{m_n}.
\hspace{1cm}
\end{split}
\end{align*}
where we may have increased the constant $c(1,\ldots ,M_{n-1})$ in the last line. This concludes the inductive step.

\bigskip

\begin{proof}[Construction of the Example:]

Let $\alpha =\lim_{n\to\i}\alpha_n$ so that $T_\alpha =\lim_{n\to\i} T_{\alpha_n}$ is the shift by the irrational number $\alpha$.

The sequence $(\tau_n)^\i_{n=1}$ of functions $\tau_n:[0,1)\to \Z$
converges, by \eqref{C1}, almost surely to a $\Z$-valued function
$\tau=\lim_{n\to\i} \tau_n$. Hence the maps
$(T^{(\tau_n)}_{\alpha_n})^\i_{n=1}$ converge almost surely to a
map \begin{equation*} T_\alpha^{(\tau)}: \left\{
    \begin{array}{rcl}
      [0,1) & \to & [0,1) \\
      x & \mapsto & T_\alpha^{(\tau)} (x)
=T_\alpha^{\tau(x)}(x). \\
    \end{array}
    \right.
 \end{equation*}
Using the fact that each $T_{\alpha_n}^{(\tau_n)}$ is  a measure
preserving almost sure bijection on $[0,1)$, it is straightforward
to check that $T^{(\tau)}_\alpha$ is so too.

Letting $\Gamma_\tau=\{(x, T^{(\tau)}_\alpha (x)), \ x\in[0,1)\}$
in analogy to the notations $\Gamma_0=\{(x,x), \ x\in[0,1)\}$ and
$\Gamma_1=\{(x,T_\alpha (x)), \ x\in[0,1)\}$, we define
$$c(x,y)=
\begin{cases}
h_+(x,y), & \mbox{if } (x,y)\in \Gamma_0 \cup\Gamma_1\cup\Gamma_\tau,\\
\infty & \mbox{otherwise,}
\end{cases}
$$ where $h$ is defined in \eqref{dxy} above.
From this definition we deduce the almost sure identity, for $\tau(x) >0,$
\begin{align}\label{E28}
\begin{split}
h(x,T^{(\tau)}_\alpha (x)) & = \#\{i\in \{0,\ldots ,\tau(x)-1\}:T^i_\alpha (x)\in[0,\tfrac12)\}\\
& \quad ~ -\#\{i\in\{0,\ldots ,\tau (x)-1\}:T^i_\alpha (x)\in [\tfrac12 ,1)\}+1 \\
& =\lim\limits_{n\to\i} [\phi^{n}(x)-\phi^{n}
(T^{(\tau_n)}_{\alpha_n}(x))]+1,
\end{split}
\end{align}
a similar formula holding true for $\tau(x)<0.$

As regards the Borel functions $(\phi_n,\psi_n)^\i_{n=1}$ announced in \eqref{iii,a}, \eqref{iii,b} and \eqref{iii,c} above, we need to slightly modify the
functions $(\phi^{n},\psi^{n})^\i_{n=1}$
constructed in the above induction to make sure that they satisfy the inequality
\begin{align}\label{29}
\varphi_n(x)+\psi_n(y)\leq c(x,y), \quad \mbox{for }x\in X,y\in Y.
\end{align}
As $c=\i$ outside of $\Gamma_0 \cup\Gamma_1\cup\Gamma_\tau$ it is
sufficient to make sure that the following inequalities hold true
almost surely, for $x\in[0,1):$
\begin{align*}
& (0) \quad \phi_n(x)+\psi_n(x) \le c(x,x)=1,
\\
& (1) \quad \phi_n(x)+\psi_n(T_{\alpha}(x))\le c(x,T_\alpha(x))=
\begin{cases}
2, \ \ \mbox{for} \ x\in[0,\tfrac12), \\
0, \ \ \mbox{for} \ x\in[\tfrac12,1),
\end{cases}
\\
& (\tau)  \ \ \ \phi_n(x)+\psi_n (T^{(\tau)}_\alpha(x))\le c(x,T^{(\tau)}_\alpha(x)).
\end{align*}
The above constructed $(\phi^{n},\psi^{n})^\i_{n=1}$ only satisfy
condition $(0).$ We still have to pass from $\phi^n$ to a smaller
function $\phi_n$ -- while leaving $\psi_n:=\psi^n$  unchanged --
to satisfy $(1)$ and $(\tau)$ too. Let
\begin{align}\label{D1}
\begin{split}
\phi_n(x):=\phi^n(x) & -[\phi^n(x)+\psi^n(T_\alpha(x))-c(x,T_\alpha(x))]_+ \\
&-[\phi^n(x)+\psi^n(T_\alpha^{(\tau)}(x))-c(x,T^{(\tau)}_\alpha(x)]_+.
\end{split}
\end{align}
Clearly $\phi_n\le \phi^n$ and the functions $(\phi_n,\psi_n)$ satisfy the inequality \eqref{29}.

\medskip\noindent\textit{We have to show that the functions $\phi_n$ defined in
\eqref{D1} satisfy that $\phi^n-\phi_n$ is small in the norm of
$L^1(\mu)$, as $n\to\i,$ that is
\begin{equation}\label{eq-2}
    \lim_{n\to\i} \int_{[0,1)} (\phi^n(x)-\phi_n(x))\,dx =0,
\end{equation}
provided that $(m_n)^\i_{n=1}$ increases sufficiently fast to
infinity. }
\medskip

We may estimate the first correction term in \eqref{D1} by
\begin{align*}
\begin{split}
    [\phi^n(x)+\psi^n&(T_\alpha(x)) -c(x,T_\alpha(x))]_+\\
    &\le [\psi^n(T_\alpha(x))-\psi^n(T_{\alpha_n}(x))]_++[\phi^n(x)+\psi^n(T_{\alpha_n}(x))-c(x,T_\alpha(x))]_+.
\end{split}
\end{align*}
The second term above is dominated by $\mathbbm{1}_{\IM^n}$ which
is harmless as
$\|\mathbbm{1}_{\IM^n}\|_{L^1(\mu)}=\tfrac{1}{M_n}.$ As regards
the first term, note that $T_\alpha(x)\ominus
T_{\alpha_n}(x)=\alpha -\alpha_n = \sum^\i_{j=n+1} \tfrac{1}{M_j}$
which we may bound by $\tfrac{2}{M_{n+1}}$ by assuming that
$(m_n)^\i_{n=1}$ increases sufficiently fast to infinity. As
$\psi^n$ is constant on each of the $M_n$ many intervals
$I_{k_1,\ldots ,k_n}$ we get
\begin{align*}
\mu\{ x\in[0,1): \psi^n(T_\alpha(x))\ne \psi^n (T_{\alpha_n}(x)\} \ \le \ M_n(\alpha -\alpha_n) <\tfrac{2}{m_{n+1}}.
\end{align*}
On this set we may estimate, using only the obvious bound $|\psi_n(x)|<M_n$, that
\begin{align*}
|\psi^n(T_\alpha(x))-\psi^n(T_{\alpha_n}(x))|\le 2M_n, \quad x\in[0,1),
\end{align*}
to obtain
\begin{align*}
\|\psi^n(T_\alpha(x))-\psi^n(T_{\alpha_n}(x))\|_{L^1(\mu)}<\tfrac{4M_n}{m_{n+1}}.
\end{align*}
Hence for $(m_n)^\i_{n=1}$ growing sufficiently fast to infinity, the first correction term in \eqref{D1} is also small in $L^1$-norm.
\vskip6pt

To estimate the second correction term in \eqref{D1} note that
\begin{align}\label{D4}
\phi^n(x)+\psi^n(T^{(\tau)}_\alpha(x))=\phi^n(x)+\psi^n(T^{(\tau_n)}_{\alpha_n}(x)), \quad \ \mbox{for} \ x\in[0,1).
\end{align}
Indeed, $T^{(\tau_n)}_{\alpha_n}$ induces a permutation between the intervals $I_{k_1,\ldots ,k_n}$ and, by assertion (i) preceding the formula \eqref{B1}, we have
that $T^{(\tau_{n+j})}_{\alpha_{n+j}}$ maps the intervals $I_{k_1,\ldots ,k_n}$ onto the intervals $T^{(\tau_n)}_{\alpha_n} (I_{k_1,\ldots ,k_n}),$
for each $j\geq 0.$ Noting that $\psi^n$ is
constant on each of the intervals $I_{k_1,\ldots ,k_n}$ we obtain \eqref{D4}, by letting $j$ tend to infinity.
\\
By \eqref{istep}, $\phi^n(x)+\psi^n (T^{(\tau_n)}_{\alpha_n}(x))$
is the number of visits to $L^n$ minus the number of visits to
$R^n$ plus one, of the orbit
$(T^j_{\alpha_n})^{\tau_n(x)-1}_{j=0}.$ Similarly, by \eqref{dxy},
 $h(x,T^{\tau(x)}_\alpha(x))$ is the number of visits to $L$ minus the number of visits to
$R$ plus one, of the orbit $(T^j_{\alpha})^{\tau(x)-1}_{j=0}.$ We
have to show that the positive part of the difference
\begin{align}\label{66}
f_n(x):= [\phi^n(x)+\psi^n(T^{(\tau_n)}_{\alpha_n}(x))-h_+(x,T^{(\tau)}_{\alpha}(x))]_+, \quad x\in[0,1),
\end{align}
is small in $L^1$-norm, as $n\to\i.$ To do so, we argue separately
on $\IM^1=[\tfrac12-\tfrac{1}{2M_1},\tfrac12 +\tfrac{1}{2M_1}],$
on the union of the ``good'' intervals at level $n:$
    $
G_n=\bigcup_{(k_1,\ldots ,k_n)\in J^g_n}   I_{k_1,\ldots ,k_n},
    $
and the union of the ``singular'' intervals at level $n,$
    $
S_n=\bigcup_{(k_1,\ldots ,k_n)\in J^s_n}   I_{k_1,\ldots ,k_n}.
    $
\begin{itemize}
    \item[-] For $x\in\IM^1,$ the correction term $f_n(x)$ in \eqref{66} simply
equals zero as $\tau_n(x)=\tau(x)=0.$

    \item[-] For $x\in S_n$, we have by \eqref{I5a} that
$\phi^n(x)+\psi^n(T^{\tau_n(x)}_{\alpha_n}(x))\le 1$ so that
$f_n(x)\le 1$ too; hence
$\lim_{n\to\i}\|f_n\mathbbm{1}_{S_n}\|_{L^1(\mu)}=0.$

    \item[-] For $x\in G_n$, we use
\begin{align*}
\begin{split}
f_n(x) & \le [\phi^n(x)+\psi^n(T^{(\tau_n)}_{\alpha_n}(x))-h(x,T^{\tau(x)}_{\alpha}(x))]_+ \\
& \le\sum\limits^\i_{k=n+1}
[(\phi^{k-1}(x)+\psi^{k-1}(T^{(\tau_{k-1})}_{\alpha_{k-1}}(x)))-(\phi^k(x)+\psi^k(T^{(\tau_k)}_{\alpha_k}(x)))]_+
\end{split}
\end{align*}
and \eqref{I4} to conclude that
\begin{align*}
\lim\limits_{n\to\i} \|f_n\mathbbm{1}_{G_n}\|_{L^1(\mu)}\le
\lim\limits_{n\to\i}\sum\limits^\i_{k=n+1} \tfrac{c(M_1,\ldots
,M_{k-1})}{m_{k}}=0.
\end{align*}
\end{itemize}
This proves \eqref{eq-2}.
\\
Hence \eqref{iii,a}, \eqref{iii,b} and \eqref{iii,c} are satisfied.

\medskip
As regards assertion \eqref{Seite10}, let us verify that $\pi_0$
and $\pi_1$ are optimal transport plans. Indeed, it follows from
\eqref{iii,a} and \eqref{iii,b} that the dual value of the present
transport problem is greater than or equal to one which implies
that $\langle c,\pi_0\rangle= \langle c,\pi_1\rangle =1$ is the
optimal primal value.

The fact that $\langle c,\pi_\tau\rangle >1$ should be rather
obvious to a reader who has made it up to this point of the
construction. It follows from rough estimates. The set
$\{[0,\tfrac12)\cap
\{\tau=-1\}\}\cup\{[\tfrac12,1)\cap\{\tau=1\}\}$ has measure
bigger than $1-\tfrac{3}{M_1}+\sum^\i_{i=2} \tfrac{c(M_1,\ldots
,M_{i-1})}{m_i}$, which is bigger than, say, $\tfrac34$, for
$(m_n)^\i_{n=1}$ tending sufficiently quick to infinity. As
$c(x,T^{(\tau)}_\alpha(x))$ equals $2$ on this set we get
$$\langle c,\pi_\tau\rangle \geq\tfrac32 >1.$$ A slightly more involved argument, whose verification is left to the energetic reader,
shows that, for $\varepsilon >0$,  we may choose $(m_n)^\i_{n=1}$ such that
\begin{align}\label{E1}
\langle h,\pi_\tau\rangle \geq 2-\varepsilon.
\end{align}

Finally, we show assertion (iv) at the beginning of this section (see \eqref{tag4}).
Let $\hh \in L^1(\pi)^{**}$ be a dual optimizer in the sense of  \cite[Theorem 4.2]{BeLS09a}.
We know from this theorem that there is a sequence $(\phi_n,\psi_n)^\i_{n=1}$ of bounded Borel functions\footnote {The $(\phi_n ,\psi_n)$ need not be the
same as the special sequence constructed above; still we find it convenient to use the same notation.} such that
\begin{align}\label{alpha}
& (\alpha) \lim\limits_{n\to\i}
\|[\phi_n\oplus\psi_n-c]_+\|_{L^1(\pi)}=0 \quad
\\
\label{beta} & (\beta) \lim\limits_{n\to\i}(\int_X \phi_n(x)\,
d\mu(x)+\int_Y\psi_n(y) ~ d\nu(y))=1,
\\
\label{gamma} & (\gamma) \lim\limits_{n\to\i}
\phi_n\oplus\psi_n=\hh^r, \quad \pi\mbox{-a.s.},
\\
\label{delta}
&
(\delta)
\mbox{ $\hat h$ is a $\sigma (L^1(\pi)^{**}, L^\i(\pi))$ cluster point of $(\phi_n\oplus\psi_n)_{n=1}^{\infty}$.}
\end{align}
Here $\hh=\hh^r+\hh^s$ is the decomposition of $\hh\in L^1(\pi)^{**}$ into its regular part $\hh^r\in L^1(\pi)$ and into its purely singular part
$\hh^s\in L^1(\pi)^{**}.$
\\
We shall show that $\hh^r$ equals $h,$ $\pi$-almost surely. Indeed
by assertions \eqref{alpha} and \eqref{beta} above we have that,
for $x\in[0,1),$
$$\lim\limits_{n\to\i}(\phi_n(x)+\psi_n(x))=c(x,x)=h(x,x)=1,$$
and
\begin{align*}
\lim\limits_{n\to\i}(\phi_n(x)+\psi_n(T_\alpha(x)))=c(x,T_\alpha(x))=h(x,T_{\alpha}(x))=
\begin{cases}
2, \  \mbox{for} \ x\in [0,\tfrac12), \\
0, \  \mbox{for} \ x\in [\tfrac12, 1),
\end{cases}
\end{align*}
the limit holding true in $L^1([0,1],\mu)$ as well as for $\mu$-a.e.~$x\in [0,1)$, possibly after passing to a subsequence. As in the discussion following
\cite[Theorem 4.2]{BeLS09a} this implies that, for each fixed $i\in\Z$,
$$\lim\limits_{n\to\i} (\phi_n(x)+\psi_n(T^i_\alpha(x)))=h(x, T^i_\alpha(x)),\quad i\in\Z,$$
the limit again holding true in $L^1(\mu)$ and $\mu$-a.s., after
possibly passing to a diagonal subsequence. Whence, we obtain with
\eqref{gamma} that
\begin{align*}
\lim\limits_{n\to\i}(\phi_n(x)+\psi_n(T^{(\tau)}_\alpha(x)))=h(x,T^{(\tau)}_\alpha(x))=\hh^r(x,T^{(\tau)}_\alpha(x)),
\end{align*}
convergence now holding true for $\mu$-a.e.\ $x\in [0,1]$.
\\
As $x\to T^{(\tau)}_\alpha(x)$ is a measure preserving bijection we get
\begin{align*}
\int_{[0,1)}[\phi_n(x)+\psi_n(T^{(\tau)}_\alpha(x))]\,dx =
\int_{[0,1)} (\phi_n(x)+\psi_n(x))\, dx=1,
\end{align*}
so that, using \eqref{E1} we get
\begin{eqnarray*}
   &&
   \lim_{n\to\i}\int_{[0,1)}[\phi_n(x)+\psi_n(T^{(\tau)}_\alpha(x))]
        \mathbbm{1}_{\{\phi_n(x)+\psi_n(T^{(\tau)}_\alpha(x))<h(x,T^{(\tau)}_\alpha(x))\}}(x) \,dx\\
  &=&
  1-\lim\limits_{n\to\i}\int_{[0,1)}[\phi_n(x)+\psi_n(T^{(\tau)}_\alpha(x))]
    \mathbbm{1}_{\{\phi_n(x)+\psi_n(T^{(\tau)}_\alpha(x))\geq h(x,T^{(\tau)}_\alpha(x))\}} (x)\,dx\\
  &=& 1-\langle h,\pi_\tau\rangle\\
  &<& 0.
\end{eqnarray*}
From
$\lim_{n\to\i}\mu\big\{x:\phi_n(x)+\psi_n(T^{(\tau)}_\alpha(x))<
h(x,T^{(\tau)}_\alpha(x))\big\}=0$ we conclude that each
$\sigma^*$-cluster point of
$([\phi_n(\cdot)+\psi_n(T^{(\tau)}_\alpha(\cdot))]_-)^\i_{n=1}$ is
a purely singular element of $L^1(\pi)^{**}$ of norm equal to
$\langle h,\pi_\tau\rangle -1.$

Finally, we still have to specify the prime numbers $(m_n)^\i_{n=1}$ in the above induction. It is now clear what we need: apart from satisfying the conditions of
Lemma 3.1 as well as the
requirements whenever we wrote ``{\it for $m_n$ tending sufficiently fast to infinity}'', we choose the $(m_n)^\i_{n=1}$ inductively such that in \eqref{C1} we have
$\tfrac{M_{n-2}}{m_{n-1}}<2^{-n}$, that in \eqref{I4} we have $\tfrac{c(M_1,\ldots ,M_{n-2})}{m_{n-1}}<2^{-n}$ and in \eqref{I5b} we have
$\tfrac{3}{m_1} <\tfrac14$ as well as again
$\tfrac{c(M_1,\ldots ,M_{n-2})}{m_{n-1}}<2^{-n}$.

Hence we have shown all the assertions (i)-(iv) of Example 3.1 and the construction of the example is complete.
\end{proof}


\section{A Relaxation of the Dual Problem}

As in \cite[Remark 3.4]{BeLS09a}, for a given cost function $c:X\times Y\to[0,\i],$ we consider the family of pairs of functions
$$\Psi^{\textrm{rel}}(\mu,\nu)= \left\{
\begin{array}{lllll}
(\phi,\psi):\phi,\psi \ \mbox{Borel, integrable and}\\
\phi(x)+\psi(y)\le c(x,y), \ \pi\mbox{-a.s}, \\
\mbox{for each finite transport plan} \ \pi\in\Pi (\mu,\nu,c)
\end{array} \right\}$$
and define the relaxed value of the dual problem as
\begin{align}\label{TagAuf33}
\Drel =\sup\Big\{\int_X \phi \ d\mu +\int_Y \psi \ d\nu
:(\phi,\psi)\in \Psi^{\mathrm{rel}} (\mu,\nu)\Big\}.
\end{align}
Using the notation of \cite{BeLS09a} it is obvious that $D\le
\Drel$ and it is straightforward to verify that the trivial
duality inequality
$\Drel\le P$ still is satisfied. 
One might conjecture -- and the present authors did so for some
time -- that $\Drel= P$ holds true in full generality, i.e.~for
arbitrary Borel measurable cost functions $c:X\times Y\to[0,\i],$
defined on the product of two polish spaces $X$ and $Y$. In this
section we construct a counterexample showing that this is not the
case, i.e.\ it may happen that we have a duality gap $P-\Drel >0$.
The example will be a variant of the example in the previous
section, i.e.\ the $(n+1)$'th variation of \cite[Example
3.2]{AmPr03}.

In section 3 we constructed a measure preserving bijection $T^{(\tau)}_{\alpha} :[0,1)\to [0,1)$ having certain properties; we now shall construct
a sequence $(T^{(\tau_n)}_{\alpha})^\i_{n=0}$ of such maps and consider as cost function the restriction of $h_+$, where $h$ is defined in \eqref{dxy}
to the graphs $(\Gamma_n)^\i_{n=0}$
of the maps $(T^{(\tau_n)}_\alpha)^\i_{n=0}$. This sequence also ``builds up a singular mass'', which now is positive as opposed to the negative
singular mass in the previous section, but it does so in a different way. We resume the properties of these maps which we shall construct in the following
proposition.
\begin{proposition}\label{P4.1}
With the notation of section 3 there is an irrational $\alpha\in[0,1)$ and a sequence $(\tau_n)^\i_{n=0}$ of maps $\tau_n :[0,1)\to\Z$, with
$\tau_0=0$ and $\tau_1=1$, such that the transformations
$T^{(\tau_n)}_{\alpha} :[0,1)\to [0,1)$, defined by
$$T^{(\tau_n)}_{\alpha}(x)=T^{\tau_n(x)}_{\alpha}(x), \qquad x\in [0,1),$$
have the following properties.
\begin{enumerate}
\item
Each $\tau_n$ is constant on a countable collection of disjoint, half open intervals in $[0,1)$ whose union has full measure. For $n\geq 0$, the map
$T^{(\tau_n)}_{\alpha}$ defines  a measure
preserving almost sure bijection of $([0,1),\mu)$ onto itself, where $\mu =\nu$ denotes Lebesgue measure on $[0,1).$
We have, for each $n\geq 0$,
\begin{align}\label{P32}
\int_{[0,1)} h(x,T^{(\tau_n)}_\alpha(x)) \, dx=1.
\end{align}
\item
The function
$$f_n(x):=h(x, T^{(\tau_n)}_{\alpha}(x)), \qquad x\in[0,1),$$
where $h$ is defined in \eqref{dxy}, satisfies
\begin{align}\label{L3}
\| f_n-g_n \|_{L^1(\mu)} < 2^{-n}
\end{align}
where $g_n$ is a Borel function on $[0,1)$ such that
\begin{align}\label{L3a}
\mu\{g_n=0\}=1-\eta_n ,\qquad \mu\{g_n=\tfrac{1-\eta_n}{\eta_n}\}=\eta_n
\end{align}
for some sequence $(\eta_n)^\i_{n=1}$ tending to zero.
\item
There is a sequence $(\phi_n,\psi_n)^\i_{n=1}$ of bounded Borel functions such that, for every fixed $n\in\N$,
$$\lim\limits_{m\to\i}\| h(x,T^{(\tau_n)}_{\alpha}(x)) -[\phi_m(x)+\psi_m(T^{(\tau_n)}_{\alpha}(x))]\|_{L^1(\mu)}=0,$$
and
$$\lim\limits_{n\to\i} \Big[\int_{[0,1)} \phi_n(x) \,dx +\int_{[0,1)} \psi_n(y) \,dy\Big]=1.$$
\item The sequence $(T^{(\tau_n)}_\alpha)^\i_{n=1}$ converges to
the identity map in the following sense:
\begin{align}\label{L4}
\delta(x,T^{(\tau_n)}_\alpha(x)) <2^{-n},\qquad x\in[0,1), \ n\geq 1,
\end{align}
where $\delta(\cdot ,\cdot)$ denotes the Riemannian metric on $\T=[0,1)$.
\end{enumerate}
\end{proposition}

We postpone the proof of the proposition and first draw some
consequences. Suppose that $\alpha$ as well as
$(T^{(\tau_n)}_\alpha)^\i_{n=0}$ have been defined and satisfy the
assertions of Proposition \ref{P4.1}.

\begin{proposition}\label{L4.2}
Fix $M\geq 2$ and define the cost function $c_M:[0,1)\times [0,1)\to[0,\i]$ by
\begin{align*}
c_M(x,y)=
\begin{cases}
h_+(x,y), \ &\mbox{for} \ (x,y) \ \mbox{in the graph of}\ T^0_\alpha, T^1_\alpha, T^{(\tau_2)}_\alpha,T^{(\tau_3)}_\alpha,\dots ,T^{(\tau_M)}_\alpha, \\
\i, \ &\mbox{otherwise.} \\
\end{cases}
\end{align*}
For this cost function $c_M$ we find that the  primal value, denoted by $P^M$, as well as the dual value, denoted by $D^M$, of the Monge--Kantorovich problem
both are equal to 1.

In addition, there is $\beta=\beta(M)>0$, such that,
for every partial transport 
$$\sigma\in\Pipart(\mu,\nu):=\{\sigma:{\mathcal M}(X\times Y):p_X(\pi)\leq \mu, p_Y(\pi)\leq \nu\}$$
with
$$\|\sigma\|\geq\tfrac23 \ \mbox{and} \ \int_{X\times Y} c_M(x,y) \ d\sigma(x,y)\le\tfrac12,$$
there is no partial transport $\varrho\in\Pipart(\mu,\nu)$ with
$$\|\sigma+\varrho \|=1 \ \mbox{and} \ \sigma+\varrho
\in\Pi(\mu,\nu)$$ with the property that $\varrho$ is supported by
$$\Delta^\beta=\{(x,y)\in[0,1)^2: \delta (x,y)<\beta\}.$$
\end{proposition}

\begin{proof}{}
First note that there is an open and dense  subset $G\subseteq
[0,1)$ of full measure $\mu(G)=1$ such that $c_M$, restricted to
$G\times G$ is lower semi-continuous. This follows from assertion
(i) of Proposition \ref{P4.1} by replacing the half open intervals
by their open interior. Noting that $G$ is polish we may apply the
general duality theory \cite{Kell84} to the cost function $c_M$
restricted to $G\times G$ to conclude that there is no duality gap
for the cost function $c_M|_{G\times G}.$ It follows that there is
also no duality gap for the original setting of $c_M$, defined on
$[0,1)\times [0,1),$ either.

We claim that, for every $M\geq 0$, the value $D^M$ of the  dual
problem equals 1. Indeed, let $(\phi_n,\psi_n)^\i_{n=1}$ be a
sequence as in Proposition \ref{P4.1} (iii). Defining
\begin{align*}
\tilde{\phi}_n:=\phi_n -\sum\limits_{j=0}^M
[\phi_n(x)+\psi_n(T^{(\tau_j)}_\alpha
(x))-h(x,T^{(\tau_j)}_\alpha(x))]_+
\end{align*}
and $\tilde{\psi}_n ={\psi}_n$, we have that
\begin{align*}
\tilde{\phi}_n(x)+\tilde{\psi}_n(y)\le h(x,y)\le h_+(x,y),
\end{align*}
for all $(x,y)$ in the graph of $T^0_\alpha ,T^1_\alpha
,T^{(\tau_2)}_\alpha ,\ldots ,T^{(\tau_M)}_\alpha,$ and
\begin{align*}
\lim\limits_{n\to\i} \Big[\int_X \tilde{\phi}_n(x) ~dx+\int_Y \tilde{\psi}_n(y)\,dy\Big]=1,
\end{align*}
showing that $D^M\geq 1.$ It follows that $D^M=P^M=1.$

\medskip

Now suppose that the final assertion of the proposition is wrong
to find a sequence $(\sigma_n)^\i_{n=1} \in\Pipart(\mu,\nu)$ with
$\|\sigma_n\|\geq \tfrac23$ and $\int_{X\times Y} c_M(x,y) \
d\sigma_n (x,y) \le \tfrac12$, as well as a sequence
$(\varrho_n)^\i_{n=1} \in\Pipart(\mu,\nu)$ with $\|\pi_n
+\varrho_n\|=1$ and $\pi_n+\varrho_n\in\Pi(\mu,\nu)$ such that
$\varrho_n$ is supported by
\begin{align}\label{33}
\Delta^{1/n} =\{(x,y)\in [0,1)^2: \delta (x,y)< \tfrac1n\}.
\end{align}

Considering $(\sigma_n)^\i_{n=1}$ as measures on the product
$G\times G$ of the polish space $G$, we then can find by
Prokhorov's theorem a subsequence $(\sigma_{n_k})^\i_{k=1}$
converging weakly on $G\times G$ to some $\sigma\in\Pipart
(\mu,\nu)$, for which we find $\|\sigma\| \geq \tfrac23$ and
$\int_{X\times Y}c(x,y) \ d\sigma(x,y) \le \tfrac12.$ By passing
once more to a subsequence, we may also suppose that
$(\varrho_{n_k})^\i_{k=1}$ weakly converges (as measures on
$G\times G$ or $[0,1)\times [0,1)$; here it does not matter) to
some $\varrho\in\Pipart(\mu,\nu)$ for which we get $\|
\sigma+\varrho\|=1$ and $\sigma + \varrho\in\Pi (\mu,\nu).$ By
\eqref{33} we conclude that $\varrho$ induces the identity
transport from its marginal $p_X(\varrho)$ onto its marginal
$p_Y(\varrho)=p_X(\varrho).$ As $c_M(x,x)=1,$ for $x\in[0,1)$ we
find that $\int_{X\times Y} c_M(x,y) \ d\varrho(x,y) = \|\varrho\|
\le\tfrac13,$ which implies that
$$\int c_M(x,y) \ d(\pi+\varrho)(x,y) \le \tfrac12 +\tfrac13,$$
a contradiction to the fact that $P^M=1$ which finishes the proof.
\end{proof}

We now can proceed to the construction of the example.
\begin{proposition}
Assume the setting of Proposition \ref{P4.1}.
For a subsequence $(i_j)^\i_{j=2}$ of $\{2,3,\ldots \}$ we define the cost function $c :[0,1) \times [0,1) \to [0,\i]$ by
 \begin{align}\label{L8}
c(x,y)=
\begin{cases}
    h_+(x,y), &\mbox{for} \ (x,y) \ \mbox{in the support of}\ T^0_\alpha, T^1_\alpha,
        T^{(\tau_{i_2})}_\alpha,T^{(\tau_{i_3})}_\alpha,\dots ,T^{(\tau_{i_j})}_\alpha, \dots\\
\i, &\mbox{otherwise.} \\
\end{cases}
\end{align}
If $(i_j)^\i_{j=2}$ tends sufficiently fast to infinity we have
that, for this cost function $c$, the primal value $P$ is strictly
positive, while the relaxed primal value $\Prel$ (see
\cite[Example 4.3]{BeLS09a}) as well as the dual value $D$ and the
relaxed dual value $\Drel$ (see \eqref{TagAuf33}) all are equal to
$0$.

In particular there is a duality gap $P-\Drel >0$, disproving the
conjecture mentioned at the beginning of this section.
\end{proposition}

\begin{proof}{}
We proceed inductively: let $j\geq 2$ and suppose that
$i_0=0,i_1=1,i_2,\dots ,i_j$ have been defined. Apply Proposition
\ref{L4.2} to
 \begin{align*}
c_j(x,y)=
\begin{cases}
h_+(x,y), & \ \mbox{for} \ (x,y) \ \mbox{in the support of}\ T^0_\alpha, T^1_\alpha, T^{(\tau_{i_2})}_\alpha,T^{(\tau_{i_3})}_\alpha,\dots ,T^{(\tau_{i_j})}_\alpha,\\
\i, &\ \mbox{otherwise,} \\
\end{cases}
\end{align*}
to find $\beta_j >0$ satisfying the conclusion of Proposition \ref{L4.2}. We may and do assume that $\beta_j \le \min (\beta_1,\dots ,\beta_{j-1})$. Now choose
$i_{j+1}$ such that
\begin{align}\label{P37}
\delta(x, T^{(\tau_{i_{j+1}})}_\alpha(x)) < \beta_j, \quad x\in[0,1).
\end{align}
This finishes the inductive step and well-defines the cost function $c(x,y)$ in \eqref{L8}.

By \eqref{P32} each $T^{(\tau_{i_j})}_\alpha$ induces a Monge transport $\pi_{i_j} \in\Pi(\mu,\nu)$ which satisfies
$$\int_{X\times Y} h(x,y) \, d\pi_{i_j} (x,y)=\int_X h(x, T^{(\tau_{i_j})}_\alpha) \, dx=1.$$

The fact that the relaxed primal value $\Prel$ for the cost
function $c$ equals zero, directly follows from the definition of
$\Prel$ \cite[Section 1.1]{BeLS09a}, \eqref{L3} and \eqref{L3a} by
transporting the measure $\mu \mathbbm{1}_{\{g_n=0\}}$, which has
mass $1-\eta_n$, via the Monge transport map $T^{(\tau_n)}_\alpha$
where $n$ is a large element of the sequence $(i_j)_{j=1}^\infty$.
Hence we conclude from \cite[Theorem 1.2]{BeLS09a} that the dual
value $D$ of the Monge--Kantorovich problem for the cost function
$c$ defined in \eqref{L8} also equals zero.

Finally observe that we have $D=\Drel$ in the present example:
indeed, the set $\{(x,y)\in[0,1)^2:c_(x,y) <\i\}$ is the countable
union of the supports of the finite cost Monge transport plans
$T^0_\alpha, T^1_\alpha,
T^{(\tau_{i_1})}_\alpha,T^{(\tau_{i_2})}_\alpha,\dots
,T^{(\tau_{i_j})}_\alpha,\dots ,$ so that the requirements
$\phi(x)+\psi(y) \le c(x,y),$ for all $(x,y)\in [0,1)^2$, and
$\phi(x)+\psi(y)\le c(x,y), \ \pi$-a.s., for each finite transport
plan $\pi\in\Pi(\mu,\nu)$, coincide (after possibly modifying
$\phi(x)$ on a $\mu$-null set).

\medskip

What remains to prove is that the primal value $P$ satisfies $P>0$. We shall show that, for every transport plan $\pi\in\Pi(\mu,\nu),$ we have
$\int_{X\times Y} c(x,y) \ d\pi (x,y)\geq \tfrac12$.
Assume to the contrary that there is $\pi\in\Pi(\mu,\nu)$ such that
$$
\int\limits_{X\times Y} c(x,y) \ d\pi (x,y)<\tfrac12.
$$
Denoting by $\sigma_j$ the restriction of $\pi$ to the union of
the graphs of the maps $T^0_\alpha,$ $ T^1_\alpha,$ $
T^{(\tau_{i_1})}_\alpha,$ $T^{(\tau_{i_2})}_\alpha,$ \dots$
,T^{(\tau_{i_j})}_\alpha,$ each $\sigma_j$ is a partial transport
in $\Pipart(\mu,\nu)$ and the norms $(\|\sigma_j\|)^\i_{j=1}$
increase to one. Choose $j$ such that $$\|\sigma_j\| >\tfrac23.$$
We apply Proposition \ref{L4.2} to conclude  that there is no
partial transport plan $\varrho_j$ such that $\pi_j+\varrho_j
\in\Pi(\mu,\nu)$, and such that $\varrho_j$ is supported by
$\Delta^{\beta_j}.$
But this is a contradiction  as $\varrho_j =\pi -\sigma_j$ has precisely these properties by \eqref{P37}.
\end{proof}

\begin{proof}[Proof of Proposition \ref{P4.1}:] The construction of the example described by Proposition \ref{P4.1} will be an extension of the construction in the
previous section from which we freely use the notation.

We shall proceed by induction on $j\in\N$ and define a double-indexed family of maps $\tau_{n,j}:[0,1)\to \Z$, where $1\le n\le j.$

{\bf
Step $j=1$:} Define
$$\tau_{1,1}:[0,1)\to\Z$$
as $$\tau_{1,1} =-\tau_1,$$ where we have $m_1=M_1, \alpha_1=\tfrac{1}{M_1}$ and $\tau_1$ as in \eqref{cal} above. At this stage the only difference to
the previous section is that we change the sign of $\tau_1$ as we now shall build up a ``positive singular mass'', as opposed to the ``negative singular mass''
which we constructed in the previous section. More precisely, defining $\phi^{1},\psi^{1}$ as in \eqref{choice}, we obtain, similarly as in
\eqref{shortcal}
\begin{align*}
\phi^{1}(x)+\psi^{1} (T^{(\tau_{1,1})}_{\alpha_1}(x))=
\begin{cases}
0, &  \mbox{for} \ x \in I_{k_1}, k_1\in \{2,\dots ,(M_1-1)/2,\\
    & \phantom{ \mbox{for} \ x \in } (M_1+3)/2,\dots ,M_1-1\},\\
(M_1-1)/2, & \mbox{for} \ x \in I_{k_1}, k_1=1, M_1, \\
1, & \mbox{for} \ x \in I_{(M_1+1)/2}.\\
\end{cases}
\end{align*}

This finishes the inductive step for $j=1.$

{\bf Step $j=2$:} Let $m_2$ and $M_2=M_1m_2$ be as in section 3, where $m_2$ satisfies the requirements of Lemma 3.1, and still is free to be eventually
specified. To define $\tau_{1,2}:[0,1)\to \Z$ we want to make sure that the map $T^{(\tau_{1,2})}_{\alpha_2}$ maps the intervals $I_{k_1}$ bijectively onto
$T^{(\tau_{1,1})}_{\alpha_1}(I_{k_1}).$ Using the notation of the previous section, we consider {\it all} the intervals $I_{k_1}$ as ``{\it good}''
intervals so that we do not have to take extra care of some {\it ``singular''} intervals.

More precisely, fix $1\le k_1\le M_1,$ and write $\tau$ for $\tau_{1,1}|_{I_{k_1}}.$
If $\tau >0$, define $J^{k_1,c}$ as $\{m_2-\tau +1,\ldots ,m_2\},$ i.e.~the set of those indices $k_2$ such that the interval $I_{k_1,k_2}$ is not mapped into
$T^{(\tau_{1,1})}_{\alpha_1}(I_{k_1})$ under $T^{(\tau_{1,1})}_{\alpha_2}.$ If $\tau <0$, we define $J^{k_1,c}$ as $\{1,\ldots ,|\tau|\},$ and if $\tau =0,$
we define $J^{k_1,c}$ as the empty set. The complement $\{1,\ldots ,m_2\}\backslash J^{k_1,c}$ is denoted by $J^{k_1,u}.$

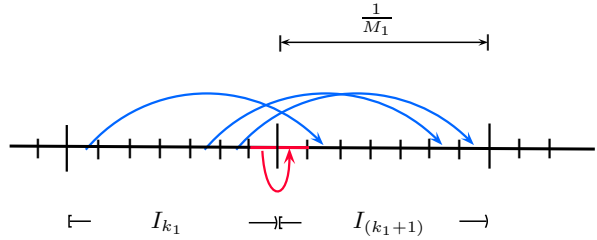
Define $\tau_{1,2}:=\tau_{1,1}=\tau$ on the intervals $I_{k_1,k_2},$ for $k_2\in J^{k_1,u}.$ On the remaining intervals $I_{k_1,k_2}$ with $k_2\in J^{k_1,c}$
we define $\tau_{1,2}$ such that it takes constant values in $\{-M_2+1,\ldots ,M_2-1\}$ on each of these intervals, such that \eqref{p24a} (resp.~\eqref{p24b}
is satisfied, and such that these intervals $I_{k_1,k_2}$ are mapped onto the ``remaining gaps'' in $T^{(\tau_{1,1})}_{\alpha_1}(I_{k_1}).$

Using again Lemma 3.3 we resume the properties of the thus constructed map $T^{(\tau_{1,2})}_{\alpha_2}:[0,1)\to [0,1).$

\begin{enumerate}
\item
The measure-preserving bijection $T^{(\tau_{1,2})}_{\alpha_2}$ maps each interval $I_{k_1}$ onto $T^{(\tau_{1,1})}_{\alpha_1}(I_{k_1}).$
It induces a permutation of the intervals
$I_{k_1,k_2},$ where $1\le k_1\le M_1, 1\le k_2\le m_2.$
\item
Defining $\phi^2,\psi^2$ as in \eqref{K1a} we get, for each $1\le k_1\le M_1,$ similarly as in \eqref{p33} and \eqref{K8a}
\begin{align*}
\mu[I_{k_1} \cap \{\tau_{1,2}\neq\tau_{1,1}\}]\leq \tfrac{M_1}{m_2} \mu [I_{k_1}],
\end{align*}
as well as
\begin{align*}
\sum\limits^{M_1}_{k_1=1} \ \int_{I_{k_1}} |(\phi^1(x)-\phi^1(T^{(\tau_{1,1})}_{\alpha_1}(x)) -(\phi^2 (x)-\phi^2 (T^{(\tau_{1,2})}_{\alpha_2} (x))|dx
 < \frac{4M^2_1}{m_2} .
\end{align*}
\item
On the middle interval $I^1_\text{middle} =I_{\frac{M_1+1}{2}}$ we have $\tau_{1,2} =\tau_{1,1} =0$.
\end{enumerate}

\medskip

We now pass to the construction of the map $\tau_{2,2}:[0,1)\to\Z.$ We define, for each $1\le k_1\le M_1,$ and $x\in I_{k_1,k_2},$
\begin{align*}
\tau_{2,2}(x)=
\begin{cases}
a_2 (k_2), & \mbox{for} \ k_2\in\{1,\ldots ,M_1\} \\
-M_1, & \mbox{for} \ k_2\in\{M_1+1,\ldots ,(m_2-1)/ 2\}, \\
0, & \mbox{for} \ k_2=(m_2+1)/ 2, \\
M_1, & \mbox{for} \ k_2 \in\{(m_2+3)/ 2,\ldots, m_2-M_1\},\\
a_2 (k_2), & \mbox{for} \ k_2\in\{m_2-M_1+1,\ldots ,m_2\}.
\end{cases}
\end{align*}

The definition of the function $a_2$ on the ``singular'' intervals $I_{k_1,k_2}$, where $k_2\in\{1,\ldots ,M_1\} \cup \{m_2 -M_1+1,\ldots ,m_2\}$ is done
such that $T^{(\tau_{2,2})}_{\alpha_2}$ maps these intervals onto ``remaining gaps'' $I_{k_1,l_2}$, where $l_2$ runs through the set
$$\{(m_2-1)/2-M_1+1,\ldots ,(m_2-1)/2\} \cup \{(m_2+3)/2,\ldots ,(m_2+3)/2+M_1-1\}$$
in the middle region of the interval $I_{k_1}.$ As above we require in addition that $a_2$ on each $I_{k_1,k_2}$ takes constant values  in
$\{-M_2+1,\ldots ,M_2-1\}$ and that \eqref{p24a} (resp.~\eqref{p24b}) is satisfied.

The function $\tau_{2,2}$ mimics the construction of $\tau_{1,1}$ above, with the role of $[0,1)$ replaced by each of the intervals $I_{k_1}$,
for $1\le k_1\le M_1.$ The idea is that, $T^{M_1}_{\alpha_1}$ being the identity map, we have that $T^{M_1}_{\alpha_2}$ satisfies $T^{M_1}_{\alpha_2}(x) =x
\oplus\tfrac{M_1}{M_2}$ and $\tfrac{M_1}{M_2}=\tfrac{1}{m_2}$ is small. Hence the role of $T_{\alpha_1}$ in the previous section now is taken by
$T^{M_1}_{\alpha_2}$.

More precisely, we have, for each $k_1=1,\ldots ,M_1$, and $x\in I_{k_1,k_2}$
\begin{align}\label{page41}
 \phi^2(x)+&\psi^2(T^{(\tau_{2,2})}_{\alpha_2}(x))= \nonumber\\
& =
\begin{cases}
0,  & \mbox{for} \ k_2\in\{M_1+1,\ldots ,(m_2-1)/ 2, \\
& \phantom{\mbox{for} \ k_2\in}(m_2+3)/2,\ldots ,m_2-M_1\}, \\
\tfrac{m_2}{2M_1} +\gamma(M_1), & \mbox{for}\  k_2\in\{1,\ldots ,M_1\}\cup \{m_2-M_1+1,\ldots ,m_2\}, \\
1,  & \mbox{for} \ k_2=(m_2+1)/2.
\end{cases}
\end{align}
The notation $\gamma(M_1)$ denotes a quantity verifying $|\gamma(M_1)|\le c(M_1)$ for some constant $c(M_1)$, depending only on $M_1.$ The verification of
\eqref{page41} uses Lemma 3.3 and is analogous as in section 3.

\medskip

As $T^{(\tau_{2,2})}_{\alpha_2}$ defines a measure preserving bijection on $[0,1)$, we get
\begin{align}\label{97}
\int^1_0(\phi^2(x)+\psi^2 (T^{(\tau_{2,2})}_{\alpha_2}(x))~dx= \int^1_0(\phi^2(x)+\psi^2(x))\,dx=1.
\end{align}

This finishes the inductive step for $j=2.$

{\bf General Inductive step:} For prime numbers $m_1,\ldots ,m_{j-1}$ as in the previous section suppose that we have defined, for $1\le n\le j-1$ maps
$\tau_{n,j}:[0,1)\to\Z$ such that the following inductive hypotheses are satisfied.
\begin{enumerate}
\item
For $1\le n\le j-1,$ the measure preserving bijection $T^{(\tau_{n,j-1})}_{\alpha_i}:[0,1)\to [0,1)$ maps the intervals $I_{k_1,\ldots , k_{n-1}}$ onto
themselves. It induces a permutation of the intervals $I_{k_1,\ldots ,k_{j-1}},$ where $1\le k_1\le m_1,\ldots ,1\le k_{j-1}\le m_{j-1}.$
\item
For $1\le n <j-1$ we have, for $1\le k_1\le m_1,\ldots ,1\le k_{j-2}\le m_{j-2},$
\begin{align}\label{page41a}
\mu[I_{k_1,\ldots ,k_{j-2}} \cap \{\tau_{n,j-2}\neq\tau_{n,j-1}\}]\leq \tfrac{M_{j-2}}{m_{j-1}} \mu [I_{k_1,\ldots ,k_{j-2}}],
\end{align}
and
\begin{align}\label{The111}
\begin{split}
& \sum_{1\le k_1\le m_1,\ldots ,1\le k_{j-2}\le m_{j-2}} \ \ \int_{I_{k_1,\dots ,k_{j-2}}} \Big|\Big(\phi^{j-2}(x)-\phi^{j-2}
(T^{(\tau_{n,j-2})}_{\alpha_{j-2}}(x))\Big)  \\
& -\Big(\phi^{j-1}(x)-\phi^{j-1}
(T^{(\tau_{n,j-1})}_{\alpha_{j-1}}(x))\Big) \Big| dx <\tfrac{c(M_1,\dots ,M_{j-2})}{m_{j-1}}.
\end{split}
\end{align}
\end{enumerate}

\medskip

We now shall define $\tau_{n,j}:[0,1)\to\Z,$ for $1\le n \le j$ and $\tau_{j,j}:[0,1)\to\Z.$

Fix $1\le n\le j-1$ as well as $1\le k_1\le m_1,\ldots ,1\le k_{j-1}\le m_{j-1}.$ Denote by $\tau$ the constant value $\tau_{n, j-1}|_{I_{k_1,\ldots ,k_{j-1}}}.$
If $\tau >0$ define $J^{k_1,\ldots ,k_{j-1},c}$ as $\{m_j-\tau +1,\ldots ,m_j\}$, similarly as for the case $j=2$ above. If $\tau \le 0$ define
$J^{k_1,\ldots ,k_{j-1},c}$ as $\{1,\ldots ,|\tau|\}$ which, for $\tau=0$, equals the empty set.
On the intervals $I_{k_1,\ldots ,k_{j-1},k_j}$ where $k_j$ lies in the complement $J^{k_1,\ldots ,k_{j-1},u}=\{1,\ldots ,m_j\}\backslash J^{k_1,\ldots ,k_{j-1},c}$
we define $\tau_{n,j}:=\tau_{n,j-1}.$ On the remaining intervals $I_{k_1,\ldots ,k_{j-1},k_j},$ where $k_j\in J^{k_1,\ldots ,k_{j-1},c},$ we define $\tau_{n,j}$ in
such a way that it takes constant values in $\{-M_j+1,\ldots ,M_j-1\}$ on each of these intervals, such that \eqref{p24a} (resp. \eqref{p24b}) is satisfied,
and such that these intervals $I_{k_1,\ldots ,k_{j-1},k_j}$ are mapped onto the ``remaining gaps'' in $T^{(\tau_{n,j-1})}_{\alpha_{j-1}}(I_{k_1,\ldots ,k_{j-1}}).$

Similarly as in the previous section we thus well-define the function $\tau_{n,j}$ which then verifies \eqref{page41a} and \eqref{The111}, with $j-1$
replaced by $j$.

\medskip

We still have to define $\tau_{j,j}:[0,1)\to\Z.$ For $1\le k_1\le m_1,\ldots ,1\le k_{j-1}\le m_{j-1}$, we define $\tau_{j,j}(x)$ on the intervals
$I_{k_1,\ldots ,k_{j-1},k_j}$ by
\begin{align*}
\tau_{j,j}(x)=
\begin{cases}
a_j (k_j), & \mbox{for} \ k_j\in\{1,\ldots ,M_{j-1}\} \\
-M_j, & \mbox{for} \ k_j\in\{M_{j-1}+1,\ldots ,(m_j-1)/ 2\}, \\
0, & \mbox{for} \ k_j=(m_j+1)/ 2, \\
M_j, & \mbox{for} \ k_j \in\{(m_j+3)/ 2,\ldots, m_j-M_{j-1}\},\\
a_j (k_j), & \mbox{for} \ k_j\in\{m_j-M_{j-1}+1,\ldots ,m_j\}.
\end{cases}
\end{align*}

Similarly as in step $j=2$ the $\{-M_j+1,\ldots ,M_j-1\}$-valued function $a_j(k_j)$ is defined in such a way that $T^{(\tau_j)}_{\alpha_j}$ maps the intervals
$I_{k_1,\ldots ,k_{j-1},k_j}$ with $k_j\in\{1,\ldots ,M_{j-1}\}\cup \{m_j-M_{j-1}+1,\ldots ,m_j\}$ to the intervals $I_{k_1,\ldots ,k_{j-1},k_j},$ where
$k_j$ runs through the ``middle region''
$$\{(m_j-1)/2-M_{j-1}+1,\ldots ,(m_j-1)/2\}\cup \{(m_j+3)/2,\ldots ,(m_j+3)/2+M_{j-1}-1\}.$$ We now deduce from
Lemma 3.3 that, for $x\in I_{k_1,\ldots ,k_{j-1},k_j}$
\begin{align*}
&\phi^j(x)+\phi^j(T^{(\tau_{j,j})}_{\alpha_j}(x))=\\
&=
\begin{cases}
0,  & \mbox{for} \ k_j\in \{M_{j-1}+1,\ldots ,(m_j-1)/2\} \\
&\phantom{\mbox{for} \ k_j\in} \cup\{(m_j+3)/2,\ldots ,m_j-M_{j-1}\}, \\
\tfrac{m_j}{2M_{j-1}} +\gamma(M_1,\ldots ,M_{j-1}), &  \mbox{for} \ k_j\in    \{1,\ldots ,M_{j-1}\}\cup \{m_j-M_{j-1}+1,\ldots ,m_j\}, \\
1,  & \mbox{for} \ k_j=(m_j+1)/2,
\end{cases}
\end{align*}
where $\gamma(M_1, \ldots, M_{j-1})$ denotes a quantity which is bounded in absolute value by a constant $c(M_1, \ldots, M_{j-1})$ depending only on $M_1, \ldots, M_{j-1}$.

This completes the inductive step.

\bigskip

We now define $\tau_0=0,\tau_1=1$ and, for $n\geq 2$
\begin{align}\label{SSS}
\tau_n=\lim\limits_{j\to\i} \tau_{n-1,j}.
\end{align}


It follows from \eqref{page41a} that, for each $n\geq 2$, the
limit \eqref{SSS} exists almost surely provided the sequence
$(m_n)_{n=1}^\infty$ converges sufficiently fast to infinity,
similarly as in section 3 above. The $(\tau_n)^\i_{n=0}$ and the
above constructed functions $(\phi_n, \psi_n)_{n=1}^\infty$
satisfy the assertions of Proposition \ref{P4.1}. The verification
of items (i), (ii), and (iii) is analogous to the arguments of
section 3 and therefore skipped. As regards assertions (iv) note
that, for $1\leq n\leq j$ the function
$T_{\alpha_j}^{(\tau_{n,j})}$ maps the intervals $I_{k_1,\ldots,
k_{n-1}} $ onto themselves. It follows that
$T_{\alpha}^{(\tau_n)}$ does so too, whence
$$\delta(x, T_\alpha^{(\tau_n)}(x))< M_n^{-1},$$
which readily shows \eqref{L4}.
\end{proof}

\def\ocirc#1{\ifmmode\setbox0=\hbox{$#1$}\dimen0=\ht0 \advance\dimen0
  by1pt\rlap{\hbox to\wd0{\hss\raise\dimen0
  \hbox{\hskip.2em$\scriptscriptstyle\circ$}\hss}}#1\else {\accent"17 #1}\fi}

\end{document}

%% file: devilish1.tex
\begin{center}
\par\vskip 1cm

\scalebox{1} 
{
\begin{pspicture}(0,-3.3767188)(11.61375,1.6667187)
\definecolor{color2341}{rgb}{0.0,0.8,0.0}
\definecolor{color2382}{rgb}{0.0,0.4,1.0}
\definecolor{color2392}{rgb}{1.0,0.2,0.4}
\definecolor{color3217}{rgb}{0.2,0.8,0.0}
\psbezier[linewidth=0.02,linecolor=color3217,arrowsize=0.093cm
2.0,arrowlength=1.4,arrowinset=0.4]{<-}(5.6409373,-0.84328127)(5.8373675,-0.64800173)(5.9609375,-0.38502038)(5.738854,-0.36415082)(5.516771,-0.34328124)(5.1409373,-0.41632473)(5.619271,-0.8232812)
\psline[linewidth=0.02cm,linestyle=dashed,dash=0.16cm
0.16cm](6.1409373,-0.86328125)(6.1609373,1.6367188)
\psline[linewidth=0.02cm,linestyle=dashed,dash=0.16cm
0.16cm](5.1609373,-0.84328127)(5.1609373,1.6567187)
\psarc[linewidth=0.02,linecolor=color2341,arrowsize=0.0512cm
2.0,arrowlength=2.08,arrowinset=0.47]{<-}(10.2109375,-1.0132812){0.47}{12.264773}{159.77515}
\psarc[linewidth=0.02,linecolor=color2341,arrowsize=0.0512cm
2.0,arrowlength=2.08,arrowinset=0.47]{<-}(9.2109375,-1.0132812){0.47}{12.264773}{159.77515}
\psarc[linewidth=0.02,linecolor=color2341,arrowsize=0.0512cm
2.0,arrowlength=2.08,arrowinset=0.47]{<-}(8.230938,-0.97328126){0.47}{12.264773}{159.77515}
\psarc[linewidth=0.02,linecolor=color2341,arrowsize=0.0512cm
2.0,arrowlength=2.08,arrowinset=0.47]{<-}(7.1709375,-0.97328126){0.47}{12.264773}{159.77515}
\psarc[linewidth=0.02,linecolor=color2341,arrowsize=0.0712cm
2.0,arrowlength=2.08,arrowinset=0.47]{<-}(8.690937,0.68671876){2.63}{-144.46233}{-36.060425}
\psarc[linewidth=0.02,linecolor=color2341,arrowsize=0.0512cm
3.0,arrowlength=2.08,arrowinset=0.47]{->}(4.1509376,-1.0132812){0.47}{25.016893}{159.77515}
\psarc[linewidth=0.02,linecolor=color2341,arrowsize=0.0512cm
3.0,arrowlength=2.08,arrowinset=0.47]{->}(3.1709375,-0.99328125){0.47}{25.016893}{159.77515}
\psarc[linewidth=0.02,linecolor=color2341,arrowsize=0.0512cm
3.0,arrowlength=2.08,arrowinset=0.47]{->}(2.1509376,-1.0132812){0.47}{25.016893}{159.77515}
\psarc[linewidth=0.02,linecolor=color2341,arrowsize=0.0512cm
3.0,arrowlength=2.08,arrowinset=0.47]{->}(1.1909375,-1.0132812){0.47}{25.016893}{159.77515}
\psarc[linewidth=0.02,linecolor=color2341,arrowsize=0.0712cm
3.0,arrowlength=2.08,arrowinset=0.47]{->}(2.6709375,0.70671874){2.63}{-144.46233}{-36.060436}
\psline[linewidth=0.04cm](0.1809375,-0.80328125)(11.180938,-0.86328125)
\psline[linewidth=0.04cm,linecolor=color2382](10.180938,-0.40328124)(11.180938,-0.40328124)
\psline[linewidth=0.04cm,linecolor=color2382](9.160937,0.19671875)(10.160937,0.19671875)
\psline[linewidth=0.04cm,linecolor=color2382](8.180938,0.65671873)(9.180938,0.65671873)
\psline[linewidth=0.04cm,linecolor=color2382](7.1609373,1.1767187)(8.160937,1.1767187)
\psline[linewidth=0.04cm,linecolor=color2382](5.1409373,1.6367188)(7.1809373,1.6367188)
\psline[linewidth=0.04cm,linecolor=color2382](4.1809373,1.1567187)(5.1809373,1.1567187)
\psline[linewidth=0.04cm,linecolor=color2382](3.1809375,0.63671875)(4.1809373,0.63671875)
\psline[linewidth=0.04cm,linecolor=color2382](2.1809375,0.13671875)(3.1809375,0.13671875)
\psline[linewidth=0.04cm,linecolor=color2382](1.1809375,-0.31325272)(2.1809375,-0.31325272)
\psline[linewidth=0.04cm,linecolor=color2382](0.1809375,-0.8232242)(1.1809375,-0.8232242)
\psline[linewidth=0.04cm,linecolor=color2392](5.1809373,-0.84328127)(6.1809373,-0.84328127)
\psdots[dotsize=0.198,dotstyle=|](0.1809375,-0.80328125)
\psdots[dotsize=0.198,dotstyle=|](11.180938,-0.84328127)
\psdots[dotsize=0.198,dotstyle=|](1.1609375,-0.8232812)
\psdots[dotsize=0.198,dotstyle=|](2.1409376,-0.8232812)
\psdots[dotsize=0.198,dotstyle=|](3.1609375,-0.86328125)
\psdots[dotsize=0.198,dotstyle=|](4.1609373,-0.86328125)
\psdots[dotsize=0.198,dotstyle=|](7.1609373,-0.8232812)
\psdots[dotsize=0.198,dotstyle=|](8.160937,-0.84328127)
\psdots[dotsize=0.198,dotstyle=|](9.180938,-0.86328125)
\psdots[dotsize=0.198,dotstyle=|](10.160937,-0.86328125)
\psline[linewidth=0.02cm,tbarsize=0.07055555cm
5.0,bracketlength=0.15]{[-}(0.1609375,-2.3032813)(1.9409375,-2.3232813)
\psline[linewidth=0.02cm,tbarsize=0.07055555cm
5.0,rbracketlength=0.15]{[-)}(5.1809373,-1.8232813)(6.1809373,-1.8432813)
\psdots[dotsize=0.198,dotstyle=|](5.1609373,-0.86328125)
\psdots[dotsize=0.198,dotstyle=|](6.1609373,-0.80328125)
\usefont{T1}{ptm}{m}{n}
\rput(2.5623438,-2.3132813){$L^1$}
\usefont{T1}{ptm}{m}{n}
\rput(8.632343,-2.3132813){$R^1$}
\usefont{T1}{ptm}{m}{n}
\rput(5.6584377,-1.5432812){\small $\IM^1$}
\usefont{T1}{ptm}{m}{n}
\rput(0.23234375,-1.2532812){$0$}
\usefont{T1}{ptm}{m}{n}
\rput(11.172344,-1.2732812){$1$}
\psline[linewidth=0.02cm,tbarsize=0.07055555cm
5.0,rbracketlength=0.15]{-)}(9.280937,-2.3232813)(11.160937,-2.3432813)
\psline[linewidth=0.02cm,tbarsize=0.07055555cm
5.0,bracketlength=0.15]{[-}(6.1809373,-2.3432813)(7.9409375,-2.3232813)
\psline[linewidth=0.02cm,tbarsize=0.07055555cm
5.0,rbracketlength=0.15]{-)}(3.2809374,-2.3432813)(5.1609373,-2.3432813)
\psline[linewidth=0.02cm,tbarsize=0.07055555cm
5.0,rbracketlength=0.15]{[-)}(1.2009375,-2.8032813)(2.2009375,-2.8232813)
\psline[linewidth=0.02cm,tbarsize=0.07055555cm
5.0,rbracketlength=0.15]{[-)}(10.160937,-2.7632813)(11.160937,-2.7832813)
\psline[linewidth=0.02cm,tbarsize=0.07055555cm
5.0,rbracketlength=0.15]{[-)}(4.2009373,-2.7832813)(5.2009373,-2.8032813)
\psline[linewidth=0.02cm,tbarsize=0.07055555cm
5.0,rbracketlength=0.15]{[-)}(0.1609375,-2.7832813)(1.1609375,-2.8032813)
\usefont{T1}{ptm}{m}{n}
\rput(0.6684375,-3.1232812){\small $I_1$}
\usefont{T1}{ptm}{m}{n}
\rput(1.6484375,-3.1632812){\small $I_2$}
\usefont{T1}{ptm}{m}{n}
\rput(4.6684375,-3.1032813){\small $I_{\frac{M_1-1}{2}}$}
\usefont{T1}{ptm}{m}{n}
\rput(10.768437,-3.0832813){\small $I_{M_1}$}
\end{pspicture}
}
    \vskip 0.5cm
\textit{Fig.\,1}.\quad Representations of  $\varphi^1$ and
$\tau^1.$
\end{center}

%% file: devilish2.tex
\begin{center}
\par\vskip 1cm
\scalebox{1} 
{
\begin{pspicture}(0,-2.18)(13.535,2.18)
\definecolor{color2873}{rgb}{0.0,0.4,1.0}
\definecolor{color2874}{rgb}{1.0,0.0,0.2}
\psline[linewidth=0.02cm,linestyle=dotted,dotsep=0.16cm](2.1021874,1.29)(8.422188,1.25)
\psline[linewidth=0.02cm](2.4221876,0.29)(13.422188,0.23)
\psdots[dotsize=0.198,dotstyle=|](2.4221876,0.29)
\psdots[dotsize=0.198,dotstyle=|](13.422188,0.25)
\psline[linewidth=0.04cm](7.3821874,0.73)(8.402187,0.73)
\psline[linewidth=0.04cm,linecolor=color2873](3.3421874,1.27)(7.4021873,1.27)
\psline[linewidth=0.04cm,linecolor=color2873](8.442187,1.25)(12.422188,1.25)
\psline[linewidth=0.04cm,linecolor=color2874](2.4221876,-1.71)(3.3821876,-1.71)
\psline[linewidth=0.04cm,linecolor=color2874](12.422188,-1.75)(13.402187,-1.75)
\psline[linewidth=0.02cm,linestyle=dotted,dotsep=0.16cm](2.4221876,0.27)(2.4221876,-1.69)
\psline[linewidth=0.02cm,linestyle=dotted,dotsep=0.16cm](3.3621874,1.25)(3.3621874,-1.73)
\psline[linewidth=0.02cm,linestyle=dotted,dotsep=0.16cm](7.3821874,1.25)(7.3821874,0.23)
\psline[linewidth=0.02cm,linestyle=dotted,dotsep=0.16cm](8.402187,1.21)(8.402187,0.23)
\psline[linewidth=0.02cm,linestyle=dotted,dotsep=0.16cm](12.422188,1.25)(12.402187,-1.71)
\psline[linewidth=0.02cm,linestyle=dotted,dotsep=0.16cm](13.402187,0.23)(13.402187,-1.71)
\psline[linewidth=0.02cm,arrowsize=0.093cm
2.0,arrowlength=1.4,arrowinset=0.4]{->}(2.0821874,-2.17)(2.1021874,2.17)
\psdots[dotsize=0.102,dotstyle=+](2.0821874,0.29)
\psdots[dotsize=0.102,dotstyle=+](2.0821874,1.29)
\psdots[dotsize=0.102,dotstyle=+](2.0821874,0.81)
\psdots[dotsize=0.102,dotstyle=+](2.0821874,-1.71)
\usefont{T1}{ptm}{m}{n}
\rput(1.7696875,0.25){\small $0$}
\usefont{T1}{ptm}{m}{n}
\rput(1.7896875,0.81){\small $1$}
\usefont{T1}{ptm}{m}{n}
\rput(1.7696875,1.25){\small $2$}
\usefont{T1}{ptm}{m}{n}
\rput(1.4096875,-1.67){\small $-\frac{M_1-5}{2}$}
\usefont{T1}{ptm}{m}{n}
\rput(2.5696876,-0.05){\small $0$}
\usefont{T1}{ptm}{m}{n}
\rput(13.229688,-0.17){\small $1$}
\usefont{T1}{ptm}{m}{n}
\rput(7.8996873,-0.13){\small $\IM^1$}
\psline[linewidth=0.02cm,linestyle=dotted,dotsep=0.16cm](3.4021876,-1.69)(12.362187,-1.73)
\psline[linewidth=0.02cm,linestyle=dotted,dotsep=0.16cm](2.1021874,0.79)(7.4021873,0.75)
\psline[linewidth=0.02cm,linestyle=dotted,dotsep=0.16cm](2.1021874,-1.69)(2.4421875,-1.69)
\psdots[dotsize=0.2,dotstyle=|](7.3821874,0.27)
\psdots[dotsize=0.2,dotstyle=|](8.402187,0.25)
\end{pspicture}
}
    \vskip 0.5cm
\textit{Fig.\,2}.\quad Representation of $\varphi^1+\psi^1\circ
T_{\alpha_1}^{(\tau_1)}.$
\end{center}

%% file: devilish3.tex
\begin{center}
\scalebox{1} 
{
\begin{pspicture}(0,-3.36)(13.575,3.36)
\definecolor{color852b}{rgb}{0.2,0.8,1.0}
\definecolor{color891b}{rgb}{1.0,0.0,0.4}
\psline[linewidth=0.02cm,linestyle=dotted,dotsep=0.16cm](1.8021874,2.47)(8.122188,2.43)
\psline[linewidth=0.02cm](2.0421875,1.47)(13.082188,1.41)
\psdots[dotsize=0.198,dotstyle=|](2.0821874,1.47)
\psdots[dotsize=0.198,dotstyle=|](13.082188,1.43)
\psline[linewidth=0.04cm](7.0621877,1.97)(8.082188,1.97)
\psline[linewidth=0.02cm,linestyle=dotted,dotsep=0.16cm](7.0821877,2.43)(7.0821877,1.41)
\psline[linewidth=0.02cm,linestyle=dotted,dotsep=0.16cm](8.102187,2.39)(8.102187,1.41)
\psline[linewidth=0.02cm,arrowsize=0.093cm
2.0,arrowlength=1.4,arrowinset=0.4]{->}(1.7621875,-0.47)(1.8021874,3.35)
\psdots[dotsize=0.102,dotstyle=+](1.7821875,1.47)
\psdots[dotsize=0.102,dotstyle=+](1.7821875,2.47)
\psdots[dotsize=0.102,dotstyle=+](1.7821875,1.99)
\usefont{T1}{ptm}{m}{n}
\rput(1.4696875,1.43){\small $0$}
\usefont{T1}{ptm}{m}{n}
\rput(1.4896874,1.99){\small $1$}
\usefont{T1}{ptm}{m}{n}
\rput(1.4696875,2.43){\small $2$}
\usefont{T1}{ptm}{m}{n}
\rput(1.9296875,1.17){\small $0$}
\usefont{T1}{ptm}{m}{n}
\rput(13.269688,1.19){\small $1$}
\usefont{T1}{ptm}{m}{n}
\rput(7.5996876,1.05){\small $\IM^1$}
\psdots[dotsize=0.2,dotstyle=|](7.0821877,1.45)
\psdots[dotsize=0.2,dotstyle=|](8.102187,1.43)
\psline[linewidth=0.02cm,linestyle=dashed,dash=0.16cm
0.16cm](1.7421875,-0.69)(1.7421875,-1.33)
\psline[linewidth=0.02cm](1.7621875,-1.47)(1.7621875,-3.35)
\psline[linewidth=0.04cm](2.3621874,1.97)(2.8621874,1.97)
\psline[linewidth=0.02cm,linestyle=dotted,dotsep=0.16cm](2.3421874,1.95)(2.3421874,-2.53)
\psline[linewidth=0.02cm,linestyle=dotted,dotsep=0.16cm](2.8421874,1.97)(2.8821876,-2.51)
\psline[linewidth=0.02cm,linestyle=dotted,dotsep=0.16cm](3.0421875,2.45)(3.0821874,-2.51)
\psline[linewidth=0.02cm,linestyle=dotted,dotsep=0.16cm](2.0621874,1.45)(2.0621874,-2.49)
\psline[linewidth=0.02cm,linestyle=dotted,dotsep=0.16cm](12.2421875,1.93)(12.2421875,-2.55)
\psline[linewidth=0.02cm,linestyle=dotted,dotsep=0.16cm](12.722187,1.93)(12.762188,-2.55)
\psline[linewidth=0.02cm,linestyle=dotted,dotsep=0.16cm](13.062187,1.43)(13.082188,-2.53)
\psline[linewidth=0.02cm,linestyle=dotted,dotsep=0.16cm](12.062187,2.45)(12.062187,-2.51)
\psline[linewidth=0.04cm](12.2421875,1.95)(12.7421875,1.95)
\psline[linewidth=0.02cm,linestyle=dotted,dotsep=0.16cm](1.8421875,1.99)(12.322187,1.95)
\psdots[dotsize=0.12,dotstyle=+](1.7421875,-2.55)
\usefont{T1}{ptm}{m}{n}
\rput(1.0496875,-2.53){\small $\propto -\frac{M_2}{M_1^2}$}
\psframe[linewidth=0.02,dimen=outer,fillstyle=solid,fillcolor=color852b](7.0621877,2.55)(3.0421875,2.33)
\psframe[linewidth=0.02,dimen=outer,fillstyle=solid,fillcolor=color852b](12.102187,2.53)(8.062187,2.33)
\psframe[linewidth=0.02,dimen=outer,fillstyle=solid,fillcolor=color891b](13.082188,-2.29)(12.782187,-2.89)
\psframe[linewidth=0.02,dimen=outer,fillstyle=solid,fillcolor=color891b](12.282187,-2.29)(12.062187,-2.89)
\psframe[linewidth=0.02,dimen=outer,fillstyle=solid,fillcolor=color891b](2.3621874,-2.31)(2.0621874,-2.91)
\psframe[linewidth=0.02,dimen=outer,fillstyle=solid,fillcolor=color891b](3.1221876,-2.31)(2.8821876,-2.93)
\psdots[dotsize=0.12,dotstyle=+](1.7421875,-2.55)
\psline[linewidth=0.02cm,fillcolor=color891b,linestyle=dotted,dotsep=0.16cm](3.1021874,-2.55)(12.042188,-2.55)
\psdots[dotsize=0.2,dotstyle=|](3.0821874,1.47)
\psdots[dotsize=0.2,dotstyle=|](12.062187,1.41)
\usefont{T1}{ptm}{m}{n}
\rput(3.2996874,1.17){\small $\frac{1}{M_1}$}
\usefont{T1}{ptm}{m}{n}
\rput(11.529688,1.13){\small $\frac{M_1-1}{M_1}$}
\end{pspicture}
}
  \vskip 0.5cm
\textit{Fig.\,3}.\quad Shape of the quasi-cost
$\varphi^2+\psi^2\circ T_{\alpha_2}^{(\tau_2)}.$
\end{center}
\textit{The strips in this graphic representation symbolize the
oscillations of the function $\varphi^2+\psi^2\circ
T_{\alpha_2}^{(\tau_2)}$. On the``singular'' set, it achieves
values of order $-M_2/M_1^2.$}

%% file: devilish4a.tex
\begin{center}
\scalebox{1} 
{
\begin{pspicture}(0,0.05171875)(8.144062,2.2348437)
\definecolor{color4894}{rgb}{0.0,0.4,1.0}
\definecolor{color4923}{rgb}{1.0,0.0,0.2}
\psarc[linewidth=0.03,linecolor=color4894,arrowsize=0.073cm
2.51,arrowlength=1.4,arrowinset=0.4]{->}(4.62,-0.09484375){2.1}{42.27369}{139.93921}
\psarc[linewidth=0.03,linecolor=color4894,arrowsize=0.073cm
2.51,arrowlength=1.4,arrowinset=0.4]{->}(4.2,-0.09484375){2.1}{42.27369}{139.93921}
\psarc[linewidth=0.03,linecolor=color4894,arrowsize=0.073cm
2.5,arrowlength=1.4,arrowinset=0.4]{->}(2.62,-0.09484375){2.1}{42.27369}{139.93921}
\psline[linewidth=0.04cm](0.0,1.3051562)(7.78,1.2651563)
\psline[linewidth=0.03cm](0.76,1.5851562)(0.76,0.96515626)
\psdots[dotsize=0.2,dotstyle=|](0.38,1.3051562)
\psdots[dotsize=0.2,dotstyle=|](1.18,1.2851562)
\psdots[dotsize=0.2,dotstyle=|](1.6,1.3051562)
\psdots[dotsize=0.2,dotstyle=|](2.0,1.3051562)
\psdots[dotsize=0.2,dotstyle=|](2.4,1.3051562)
\psdots[dotsize=0.2,dotstyle=|](2.8,1.2851562)
\psdots[dotsize=0.2,dotstyle=|](3.18,1.3051562)
\psdots[dotsize=0.2,dotstyle=|](3.96,1.2851562)
\psdots[dotsize=0.2,dotstyle=|](4.4,1.2851562)
\psdots[dotsize=0.2,dotstyle=|](4.8,1.2851562)
\psdots[dotsize=0.2,dotstyle=|](5.2,1.2851562)
\psdots[dotsize=0.2,dotstyle=|](5.58,1.3051562)
\psdots[dotsize=0.2,dotstyle=|](5.98,1.2851562)
\psdots[dotsize=0.2,dotstyle=|](6.78,1.3051562)
\psline[linewidth=0.03cm](6.38,1.6051563)(6.38,1.0051563)
\psline[linewidth=0.04cm,linecolor=color4923](3.2,1.2851562)(3.98,1.2851562)
\psline[linewidth=0.03cm](3.56,1.6051563)(3.56,1.0051563)
\psbezier[linewidth=0.03,linecolor=color4923,arrowsize=0.093cm
2.0,arrowlength=1.4,arrowinset=0.4]{->}(3.72,1.2788147)(3.72,0.5484861)(3.4,0.52515626)(3.36,1.2416221)
\usefont{T1}{ptm}{m}{n}
\rput(2.1275,0.28515625){\small $I_{(k_1-1)}$}
\usefont{T1}{ptm}{m}{n}
\rput(5.0175,0.26515624){\small $I_{k_1}$}
\psdots[dotsize=0.2,dotstyle=|](7.18,1.3051562)
\psline[linewidth=0.02cm,arrowsize=0.05291667cm
2.0,arrowlength=1.4,arrowinset=0.4]{<->}(6.78,1.6851562)(7.16,1.6851562)
\usefont{T1}{ptm}{m}{n}
\rput(6.9935937,2.0601563){\footnotesize $\frac{1}{M_2}$}
\psline[linewidth=0.02cm,tbarsize=0.07055555cm
5.0,bracketlength=0.15]{[-}(0.78,0.32515624)(1.06,0.32515624)
\psline[linewidth=0.02cm,tbarsize=0.07055555cm
5.0,rbracketlength=0.15]{-)}(5.98,0.30515626)(6.36,0.30515626)
\psline[linewidth=0.02cm,tbarsize=0.07055555cm
5.0,rbracketlength=0.15]{-)}(3.18,0.28515625)(3.56,0.28515625)
\psdots[dotsize=0.12,dotstyle=|](6.78,1.7051562)
\psdots[dotsize=0.12,dotstyle=|](7.18,1.7051562)
\psline[linewidth=0.02cm,tbarsize=0.07055555cm
5.0,bracketlength=0.15]{[-}(3.58,0.28515625)(3.86,0.28515625)
\end{pspicture}
}
\\
 \textit{Fig.\,4-a}.\quad $k_1\in J_1^g$ on the left
side.\footnote{Figure 3 is built with the small value $m_2=7$ for
the sake of clarity of the drawing. But this value is not feasible
since with the lowest $m_1=5,$ \eqref{KongOne} implies that $m_2$
is at least equal to 11; other requirements of the construction
imply that it has to be even larger.}
\end{center}

%% file: devilish4b.tex
\begin{center}
\scalebox{1} 
{
\begin{pspicture}(0,-0.41546875)(7.8,2.7020311)
\definecolor{color5290}{rgb}{0.0,0.4,1.0}
\definecolor{color5319}{rgb}{1.0,0.0,0.2}
\psarc[linewidth=0.03,linecolor=color5290,arrowsize=0.073cm
2.0,arrowlength=1.4,arrowinset=0.4]{<-}(4.62,-0.56203127){2.1}{42.27369}{139.93921}
\psarc[linewidth=0.03,linecolor=color5290,arrowsize=0.073cm
2.0,arrowlength=1.4,arrowinset=0.4]{<-}(4.2,-0.56203127){2.1}{42.27369}{139.93921}
\psarc[linewidth=0.03,linecolor=color5290,arrowsize=0.073cm
2.0,arrowlength=1.4,arrowinset=0.4]{<-}(2.62,-0.56203127){2.1}{42.27369}{139.93921}
\psline[linewidth=0.04cm](0.0,0.83796877)(7.78,0.79796875)
\psline[linewidth=0.03cm](0.76,1.1179688)(0.76,0.49796876)
\psdots[dotsize=0.2,dotstyle=|](0.38,0.83796877)
\psdots[dotsize=0.2,dotstyle=|](1.18,0.8179687)
\psdots[dotsize=0.2,dotstyle=|](1.6,0.83796877)
\psdots[dotsize=0.2,dotstyle=|](2.0,0.83796877)
\psdots[dotsize=0.2,dotstyle=|](2.4,0.83796877)
\psdots[dotsize=0.2,dotstyle=|](2.8,0.8179687)
\psdots[dotsize=0.2,dotstyle=|](3.18,0.83796877)
\psdots[dotsize=0.2,dotstyle=|](3.96,0.8179687)
\psdots[dotsize=0.2,dotstyle=|](4.4,0.8179687)
\psdots[dotsize=0.2,dotstyle=|](4.8,0.8179687)
\psdots[dotsize=0.2,dotstyle=|](5.2,0.8179687)
\psdots[dotsize=0.2,dotstyle=|](5.58,0.83796877)
\psdots[dotsize=0.2,dotstyle=|](5.98,0.8179687)
\psdots[dotsize=0.2,dotstyle=|](6.78,0.83796877)
\psline[linewidth=0.03cm](6.38,1.1379688)(6.38,0.53796875)
\psline[linewidth=0.04cm,linecolor=color5319](3.2,0.8179687)(3.98,0.8179687)
\psline[linewidth=0.03cm](3.56,1.1379688)(3.56,0.53796875)
\psbezier[linewidth=0.03,linecolor=color5319,arrowsize=0.093cm
2.0,arrowlength=1.4,arrowinset=0.4]{<-}(3.72,0.8116272)(3.72,0.081298605)(3.4,0.05796875)(3.36,0.7744346)
\usefont{T1}{ptm}{m}{n}
\rput(2.0975,-0.18203124){\small $I_{k_1}$}
\usefont{T1}{ptm}{m}{n}
\rput(5.0475,-0.20203125){\small $I_{(k_1+1)}$}
\psdots[dotsize=0.2,dotstyle=|](7.18,0.83796877)
\psline[linewidth=0.02cm,tbarsize=0.07055555cm
5.0,bracketlength=0.15]{[-}(0.78,-0.16203125)(1.08,-0.16203125)
\psline[linewidth=0.02cm,tbarsize=0.07055555cm
5.0,rbracketlength=0.15]{-)}(5.98,-0.16203125)(6.36,-0.16203125)
\psline[linewidth=0.02cm,tbarsize=0.07055555cm
5.0,rbracketlength=0.15]{-)}(3.18,-0.18203124)(3.56,-0.18203124)
\psline[linewidth=0.02cm,arrowsize=0.05291667cm
2.0,arrowlength=1.4,arrowinset=0.4]{<->}(3.6,2.2179687)(6.36,2.2179687)
\usefont{T1}{ptm}{m}{n}
\rput(4.8775,2.5179687){\small $\frac{1}{M_1}$}
\psdots[dotsize=0.16600001,dotstyle=|](3.6,2.2179687)
\psdots[dotsize=0.16600001,dotstyle=|](6.36,2.2379687)
\psline[linewidth=0.02cm,tbarsize=0.07055555cm
5.0,bracketlength=0.15]{[-}(3.58,-0.18203124)(3.88,-0.18203124)
\end{pspicture}
}
   \vskip 0.3cm
\textit{Fig.\,4-b}.\quad $k_1\in J_1^g$ on the right side.
\end{center}

%% file: devilish5.tex
\begin{center}
\scalebox{1} 
{
\begin{pspicture}(0,-2.6459374)(13.4,1.64)
\definecolor{color1468}{rgb}{0.0,0.6,0.8}
\definecolor{color1472}{rgb}{1.0,0.0,0.2}
\definecolor{color1477}{rgb}{0.4,1.0,0.0}
\definecolor{color1479}{rgb}{0.2,1.0,0.0}
\psline[linewidth=0.04cm](1.96,0.44)(5.24,0.44)
\psline[linewidth=0.04cm](7.44,0.42)(10.72,0.42)
\psline[linewidth=0.04cm,linestyle=dashed,dash=0.16cm
0.16cm](5.6,0.42)(6.98,0.42)
\psdots[dotsize=0.24,dotstyle=|](7.58,0.48)
\psdots[dotsize=0.24,dotstyle=|](10.36,0.46)
\psdots[dotsize=0.24,dotstyle=|](4.76,0.48)
\psdots[dotsize=0.24,dotstyle=|](1.94,0.48)
\psarc[linewidth=0.03,linecolor=color1468,arrowsize=0.073cm
2.0,arrowlength=1.4,arrowinset=0.4]{->}(5.26,4.48){4.86}{-122.50566}{-57.52881}
\psline[linewidth=0.04cm,linecolor=color1472](1.96,0.44)(2.34,0.44)
\psline[linewidth=0.04cm,linecolor=color1472](4.6,0.44)(4.76,0.44)
\psline[linewidth=0.04cm,linecolor=color1472](8.56,0.42)(9.16,0.42)
\psline[linewidth=0.04cm,linecolor=color1468](2.38,0.44)(3.36,0.44)
\psline[linewidth=0.04cm,linecolor=color1468](7.58,0.42)(8.56,0.42)
\psline[linewidth=0.04cm,linecolor=color1477](9.2,0.42)(10.34,0.4)
\psline[linewidth=0.04cm,linecolor=color1477](3.4,0.44)(4.56,0.44)
\psarc[linewidth=0.03,linecolor=color1479,arrowsize=0.073cm
2.0,arrowlength=1.4,arrowinset=0.4]{->}(6.76,3.96){4.6}{-129.69267}{-51.797882}
\psarc[linewidth=0.04,linecolor=color1472,arrowsize=0.05291667cm
2.0,arrowlength=1.4,arrowinset=0.4]{<-}(5.48,-3.86){5.48}{51.759083}{128.55324}
\psarc[linewidth=0.04,linecolor=color1472,arrowsize=0.073cm
2.0,arrowlength=1.4,arrowinset=0.4]{<-}(6.75,-2.33){3.49}{51.759083}{127.19101}
\psdots[dotsize=0.198,dotstyle=|](3.38,0.44)
\psdots[dotsize=0.198,dotstyle=|](8.96,0.4)
\psline[linewidth=0.04cm,linecolor=color1468](2.4,-1.14)(3.38,-1.14)
\psline[linewidth=0.04cm,linecolor=color1477](3.42,-1.14)(4.58,-1.14)
\usefont{T1}{ptm}{m}{n}
\rput(2.8075,-0.68){\small $I^{1,g,l}$}
\usefont{T1}{ptm}{m}{n}
\rput(4.0775,-0.7){\small $I^{1,g,r}$}
\psdots[dotsize=0.198,dotstyle=|](3.38,-1.12)
\psdots[dotsize=0.198,dotstyle=|](2.38,-1.14)
\psdots[dotsize=0.198,dotstyle=|](4.58,-1.14)
\usefont{T1}{ptm}{m}{n}
\rput(3.4475,-1.64){\small $I^{1,s}$}
\usefont{T1}{ptm}{m}{n}
\rput(11.8375,-2.48){\small $\IM^1$}
\usefont{T1}{ptm}{m}{n}
\rput(1.9275,0.06){\small $0$}
\psdots[dotsize=0.198,dotstyle=|](4.58,0.44)
\psdots[dotsize=0.198,dotstyle=|](8.58,0.4)
\psdots[dotsize=0.198,dotstyle=|](9.16,0.4)
\psdots[dotsize=0.198,dotstyle=|](2.36,0.46)
\usefont{T1}{ptm}{m}{n}
\rput(4.8975,0.08){\small $\frac 1{M_1}$}
\psline[linewidth=0.04cm,tbarsize=0.07055555cm
5.0,rbracketlength=0.15]{[-)}(1.96,-2.18)(4.76,-2.18)
\psline[linewidth=0.04cm,tbarsize=0.07055555cm
5.0,rbracketlength=0.15]{[-)}(7.54,-2.16)(10.34,-2.14)
\usefont{T1}{ptm}{m}{n}
\rput(3.2875,-2.4){\small $I_1$}
\usefont{T1}{ptm}{m}{n}
\rput(9.0075,-2.42){\small $I_{\frac{M_1-1}{2}}$}
\psline[linewidth=0.04cm,tbarsize=0.07055555cm
5.0,rbracketlength=0.15]{[-)}(10.38,-2.14)(13.38,-2.16)
\psline[linewidth=0.04cm,linestyle=dashed,dash=0.16cm
0.16cm](5.32,-2.16)(6.92,-2.16)
\psline[linewidth=0.04cm,linecolor=color1472](2.0,-1.74)(2.38,-1.74)
\psline[linewidth=0.04cm,linecolor=color1472](4.58,-1.76)(4.74,-1.76)
\psdots[dotsize=0.198,dotstyle=|](4.74,-1.76)
\psdots[dotsize=0.198,dotstyle=|](4.56,-1.76)
\psdots[dotsize=0.198,dotstyle=|](2.38,-1.74)
\psdots[dotsize=0.198,dotstyle=|](1.98,-1.74)
\psline[linewidth=0.02cm,arrowsize=0.05291667cm
2.0,arrowlength=1.4,arrowinset=0.4]{<-}(2.5,-1.78)(3.04,-1.68)
\psline[linewidth=0.02cm,arrowsize=0.05291667cm
2.0,arrowlength=1.4,arrowinset=0.4]{->}(3.82,-1.7)(4.42,-1.78)
\end{pspicture}
} \medskip
    \textit{Fig.\,5}.\quad $\tau_2$ for the ``singular"
indices on the left side.
 \end{center}
\textit{In this drawing, the interval $I^{1,g,l}$ is the union of
the intervals $I_{1,k_2}$ with $k_2\in J^{1,g,l}.$ A similar
convention holds for $I^{1,g,r}$ and $I^{1,s}$ (which is not an
interval anymore). }

%% file: devilish6.tex
\begin{center}
\scalebox{1} 
{
\begin{pspicture}(0,-1.805)(13.282187,1.805)
\definecolor{color2624b}{rgb}{0.4,1.0,1.0}
\definecolor{color2626b}{rgb}{1.0,0.0,0.2}
\psframe[linewidth=0.02,dimen=outer,fillstyle=solid,fillcolor=color2624b](6.2703233,0.9009375)(1.4621875,0.4809375)
\psframe[linewidth=0.02,dimen=outer,fillstyle=solid,fillcolor=color2624b](13.282187,0.9009375)(8.462188,0.4609375)
\psframe[linewidth=0.02,dimen=outer,fillstyle=solid,fillcolor=color2626b](3.6421876,0.9009375)(1.4421875,0.5009375)
\psframe[linewidth=0.02,dimen=outer,fillstyle=solid,fillcolor=color2626b](13.262188,0.8809375)(11.062187,0.4809375)
\psframe[linewidth=0.02,dimen=outer,fillstyle=solid,fillcolor=color2624b](6.2421875,0.3209375)(1.4421875,-0.0990625)
\psframe[linewidth=0.02,dimen=outer,fillstyle=solid,fillcolor=color2624b](13.2421875,0.3209375)(8.442187,-0.0990625)
\psframe[linewidth=0.02,dimen=outer,fillstyle=solid,fillcolor=color2624b](6.2421875,-0.2790625)(1.4421875,-0.6990625)
\psframe[linewidth=0.02,dimen=outer,fillstyle=solid,fillcolor=color2624b](13.2421875,-0.2790625)(8.442187,-0.6990625)
\psframe[linewidth=0.02,dimen=outer,fillstyle=solid,fillcolor=color2624b](6.2421875,-0.8990625)(1.4421875,-1.3190625)
\psframe[linewidth=0.02,dimen=outer,fillstyle=solid,fillcolor=color2624b](13.2421875,-0.8990625)(8.442187,-1.3190625)
\psframe[linewidth=0.02,dimen=outer,fillstyle=solid,fillcolor=red](1.8821875,0.3209375)(1.4621875,-0.0990625)
\psframe[linewidth=0.02,dimen=outer,fillstyle=solid,fillcolor=red](3.6821876,0.3009375)(3.1021874,-0.0990625)
\psframe[linewidth=0.02,dimen=outer,fillstyle=solid,fillcolor=red](11.622188,0.3209375)(11.042188,-0.0990625)
\psframe[linewidth=0.02,dimen=outer,fillstyle=solid,fillcolor=red](13.262188,0.3209375)(12.842188,-0.0990625)
\psframe[linewidth=0.02,dimen=outer,fillstyle=solid,fillcolor=red](1.5421875,-0.2790625)(1.4221874,-0.6990625)
\psframe[linewidth=0.02,dimen=outer,fillstyle=solid,fillcolor=red](1.8221875,-0.2790625)(1.6821876,-0.6990625)
\psframe[linewidth=0.02,dimen=outer,fillstyle=solid,fillcolor=red](3.2221875,-0.2790625)(3.1021874,-0.6790625)
\psframe[linewidth=0.02,dimen=outer,fillstyle=solid,fillcolor=red](3.6421876,-0.2790625)(3.5221875,-0.6990625)
\psframe[linewidth=0.02,dimen=outer,fillstyle=solid,fillcolor=red](11.142187,-0.2790625)(11.022187,-0.6990625)
\psframe[linewidth=0.02,dimen=outer,fillstyle=solid,fillcolor=red](11.602187,-0.2790625)(11.482187,-0.6990625)
\psframe[linewidth=0.02,dimen=outer,fillstyle=solid,fillcolor=red](12.982187,-0.2790625)(12.862187,-0.6990625)
\psframe[linewidth=0.02,dimen=outer,fillstyle=solid,fillcolor=red](13.2421875,-0.2790625)(13.122188,-0.6990625)
\usefont{T1}{ptm}{m}{n}
\rput(0.5496875,0.7209375){\small $n=1$}
\usefont{T1}{ptm}{m}{n}
\rput(0.5296875,0.1209375){\small $n=2$}
\usefont{T1}{ptm}{m}{n}
\rput(0.5496875,-0.4790625){\small $n=3$}
\usefont{T1}{ptm}{m}{n}
\rput(0.5296875,-1.0790625){\small $n=4$}
\psline[linewidth=0.04cm,tbarsize=0.07055555cm
5.0,rbracketlength=0.15]{|-)}(1.4421875,1.3009375)(3.6421876,1.3009375)
\psline[linewidth=0.04cm,tbarsize=0.07055555cm
5.0,rbracketlength=0.15]{|-)}(6.2421875,1.2809376)(8.442187,1.2809376)
\psline[linewidth=0.04cm,tbarsize=0.07055555cm
5.0,rbracketlength=0.15]{|-)}(11.042188,1.3209375)(13.2421875,1.3209375)
\psline[linewidth=0.04cm,linestyle=dashed,dash=0.16cm
0.16cm](3.8021874,1.3009375)(5.9821873,1.2809376)
\psline[linewidth=0.04cm,linestyle=dashed,dash=0.16cm
0.16cm](8.662188,1.3009375)(10.902187,1.3009375)
\usefont{T1}{ptm}{m}{n}
\rput(2.4296875,1.6009375){\small $I_1$}
\usefont{T1}{ptm}{m}{n}
\rput(7.2596874,1.5809375){\small $\IM^1$}
\usefont{T1}{ptm}{m}{n}
\rput(12.049687,1.6209375){\small $I_{M_1}$}
\psline[linewidth=0.032cm,linecolor=color2626b](12.882188,-0.8990625)(12.882188,-1.2990625)
\psline[linewidth=0.032cm,linecolor=color2626b](12.962188,-0.8990625)(12.962188,-1.2990625)
\psline[linewidth=0.032cm,linecolor=color2626b](11.022187,-0.8790625)(11.022187,-1.2790625)
\psline[linewidth=0.032cm,linecolor=color2626b](11.082188,-0.8790625)(11.082188,-1.2790625)
\psline[linewidth=0.032cm,linecolor=color2626b](3.1021874,-0.8990625)(3.1021874,-1.2990625)
\psline[linewidth=0.032cm,linecolor=color2626b](3.1621876,-0.8990625)(3.1621876,-1.2990625)
\psline[linewidth=0.032cm,linecolor=color2626b](1.7221875,-0.8990625)(1.7221875,-1.2990625)
\psline[linewidth=0.032cm,linecolor=color2626b](1.8221875,-0.8990625)(1.8221875,-1.2990625)
\psline[linewidth=0.032cm,linecolor=color2626b](1.4421875,-0.9190625)(1.4421875,-1.3190625)
\psline[linewidth=0.032cm,linecolor=color2626b](1.5021875,-0.8990625)(1.5021875,-1.2990625)
\psline[linewidth=0.032cm,linecolor=color2626b](13.162188,-0.8990625)(13.162188,-1.2990625)
\psline[linewidth=0.032cm,linecolor=color2626b](13.222187,-0.8990625)(13.222187,-1.2990625)
\psline[linewidth=0.032cm,linecolor=color2626b](11.502188,-0.8790625)(11.502188,-1.2790625)
\psline[linewidth=0.032cm,linecolor=color2626b](11.562187,-0.8990625)(11.562187,-1.2990625)
\psline[linewidth=0.032cm,linecolor=color2626b](3.5621874,-0.8990625)(3.5621874,-1.2990625)
\psline[linewidth=0.032cm,linecolor=color2626b](3.6221876,-0.8990625)(3.6221876,-1.2990625)
\usefont{T1}{ptm}{m}{n}
\rput(0.5496875,-1.6390625){\small $\vdots$}
\psframe[linewidth=0.04,dimen=outer](6.2303233,-0.8990625)(1.4221874,-1.3190625)
\psframe[linewidth=0.04,dimen=outer](13.250323,-0.8790625)(8.442187,-1.2990625)
\psframe[linewidth=0.04,dimen=outer](13.230323,-0.2790625)(8.422188,-0.6990625)
\psframe[linewidth=0.04,dimen=outer](13.250323,0.3209375)(8.442187,-0.0990625)
\psframe[linewidth=0.04,dimen=outer](13.250323,0.9009375)(8.442187,0.4809375)
\psframe[linewidth=0.04,dimen=outer](6.2703233,0.9009375)(1.4621875,0.4809375)
\psframe[linewidth=0.04,dimen=outer](6.2503233,0.3209375)(1.4421875,-0.0990625)
\psframe[linewidth=0.04,dimen=outer](6.2303233,-0.2790625)(1.4221874,-0.6990625)
\end{pspicture}
}
  \vskip 0.5cm
\textit{Fig.\,6}.\quad The fractal structure of the ``singular''
set.
\end{center}
\textit{For the sake of simplicity of the drawing, the red area
which represents the singular set is thicker than it should be.
Note also that the effective singular set is not perfectly
balanced.}

%% file: devilish7.tex
\begin{center}
\scalebox{1} 
{
\begin{pspicture}(0,-3.36)(14.915,3.36)
\definecolor{color852b}{rgb}{0.2,0.8,1.0}
\definecolor{color891b}{rgb}{1.0,0.0,0.4}
\psline[linewidth=0.02cm,linestyle=dotted,dotsep=0.16cm](3.1421876,2.47)(9.462188,2.43)
\psline[linewidth=0.02cm](3.3821876,1.47)(14.422188,1.41)
\psdots[dotsize=0.198,dotstyle=|](3.4221876,1.47)
\psdots[dotsize=0.198,dotstyle=|](14.422188,1.43)
\psline[linewidth=0.04cm](8.402187,1.97)(9.422188,1.97)
\psline[linewidth=0.02cm,linestyle=dotted,dotsep=0.16cm](8.422188,2.43)(8.422188,1.41)
\psline[linewidth=0.02cm,linestyle=dotted,dotsep=0.16cm](9.442187,2.39)(9.442187,1.41)
\psline[linewidth=0.02cm,arrowsize=0.093cm
2.0,arrowlength=1.4,arrowinset=0.4]{->}(3.1021874,-0.47)(3.1421876,3.35)
\psdots[dotsize=0.102,dotstyle=+](3.1221876,1.47)
\psdots[dotsize=0.102,dotstyle=+](3.1221876,2.47)
\psdots[dotsize=0.102,dotstyle=+](3.1221876,1.99)
\usefont{T1}{ptm}{m}{n}
\rput(2.8096876,1.43){\small $0$}
\usefont{T1}{ptm}{m}{n}
\rput(2.8296876,1.99){\small $1$}
\usefont{T1}{ptm}{m}{n}
\rput(2.8096876,2.43){\small $2$}
\usefont{T1}{ptm}{m}{n}
\rput(3.2696874,1.17){\small $0$}
\usefont{T1}{ptm}{m}{n}
\rput(14.609688,1.19){\small $1$}
\usefont{T1}{ptm}{m}{n}
\rput(8.939688,1.05){\small $\IM^1$}
\psdots[dotsize=0.2,dotstyle=|](8.422188,1.45)
\psdots[dotsize=0.2,dotstyle=|](9.442187,1.43)
\psline[linewidth=0.02cm,linestyle=dashed,dash=0.16cm
0.16cm](3.0821874,-0.69)(3.0821874,-1.33)
\psline[linewidth=0.02cm](3.1021874,-1.47)(3.1021874,-3.35)
\psline[linewidth=0.04cm](3.7021875,1.97)(4.2021875,1.97)
\psline[linewidth=0.02cm,linestyle=dotted,dotsep=0.16cm](3.6821876,1.95)(3.6821876,-2.53)
\psline[linewidth=0.02cm,linestyle=dotted,dotsep=0.16cm](4.1821876,1.97)(4.2221875,-2.51)
\psline[linewidth=0.02cm,linestyle=dotted,dotsep=0.16cm](4.3821874,2.45)(4.4221873,-2.51)
\psline[linewidth=0.02cm,linestyle=dotted,dotsep=0.16cm](3.4021876,1.45)(3.4021876,-2.49)
\psline[linewidth=0.02cm,linestyle=dotted,dotsep=0.16cm](13.582188,1.93)(13.582188,-2.55)
\psline[linewidth=0.02cm,linestyle=dotted,dotsep=0.16cm](14.062187,1.93)(14.102187,-2.55)
\psline[linewidth=0.02cm,linestyle=dotted,dotsep=0.16cm](14.402187,1.43)(14.422188,-2.53)
\psline[linewidth=0.02cm,linestyle=dotted,dotsep=0.16cm](13.402187,2.45)(13.402187,-2.51)
\psline[linewidth=0.04cm](13.622188,1.95)(14.122188,1.95)
\psline[linewidth=0.02cm,linestyle=dotted,dotsep=0.16cm](3.1821876,1.99)(13.662188,1.95)
\psdots[dotsize=0.12,dotstyle=+](3.0821874,-2.55)
\usefont{T1}{ptm}{m}{n}
\rput(1.9096875,-2.53){\small $\propto -M_n/M_{n-1}^2$}
\psframe[linewidth=0.02,dimen=outer,fillstyle=solid,fillcolor=color852b](8.402187,2.55)(4.3821874,2.33)
\psframe[linewidth=0.02,dimen=outer,fillstyle=solid,fillcolor=color852b](13.442187,2.53)(9.402187,2.33)
\psframe[linewidth=0.02,dimen=outer,fillstyle=solid,fillcolor=color891b](14.422188,-2.29)(14.302188,-2.93)
\psframe[linewidth=0.02,dimen=outer,fillstyle=solid,fillcolor=color891b](13.482187,-2.29)(13.402187,-2.85)
\psdots[dotsize=0.12,dotstyle=+](3.0821874,-2.55)
\psline[linewidth=0.02cm,fillcolor=color891b,linestyle=dotted,dotsep=0.16cm](4.4421873,-2.55)(13.382188,-2.55)
\psdots[dotsize=0.2,dotstyle=|](4.4221873,1.47)
\psdots[dotsize=0.2,dotstyle=|](13.402187,1.41)
\usefont{T1}{ptm}{m}{n}
\rput(4.6396875,1.17){\small $\frac{1}{M_1}$}
\usefont{T1}{ptm}{m}{n}
\rput(12.869687,1.13){\small $\frac{M_1-1}{M_1}$}
\psframe[linewidth=0.02,dimen=outer,fillstyle=solid,fillcolor=color891b](3.5221875,-2.39)(3.4221876,-2.97)
\psframe[linewidth=0.02,dimen=outer,fillstyle=solid,fillcolor=color891b](14.202188,-2.33)(14.122188,-2.93)
\psframe[linewidth=0.02,dimen=outer,fillstyle=solid,fillcolor=color891b](13.642187,-2.29)(13.562187,-2.87)
\psframe[linewidth=0.02,dimen=outer,fillstyle=solid,fillcolor=color891b](4.4021873,-2.39)(4.3221874,-3.01)
\psframe[linewidth=0.02,dimen=outer,fillstyle=solid,fillcolor=color891b](4.2621875,-2.41)(4.1821876,-2.99)
\psframe[linewidth=0.02,dimen=outer,fillstyle=solid,fillcolor=color891b](3.7021875,-2.39)(3.6221876,-2.95)
\psline[linewidth=0.04cm,fillcolor=color891b](3.4821875,1.99)(3.6021874,1.97)
\psline[linewidth=0.04cm,fillcolor=color891b](4.2621875,1.97)(4.3421874,1.97)
\psline[linewidth=0.04cm,fillcolor=color891b](13.462188,1.95)(13.542188,1.95)
\psline[linewidth=0.04cm,fillcolor=color891b](14.222187,1.95)(14.302188,1.95)
\end{pspicture}
}
 \vskip 0.5cm
\textit{Fig.\,7}.\quad Shape of the quasi-cost
$\varphi^n+\psi^n\circ T_{\alpha_n}^{(\tau_n)}.$
\end{center}
\textit{The strips on this graphic representation symbolize the
oscillations of the function $\varphi^n+\psi^n\circ
T_{\alpha_n}^{(\tau_n)}$. On the``singular'' set, this finction
achieves values of order $-M_n/M_{n-1}^2.$ Of course, the
effective singular set is much more fragmented than it appears on
this figure.}